\documentclass[12pt,reqno]{amsart} 

\usepackage[T1]{fontenc}
\usepackage{enumitem}
\setlist[enumerate,1]{wide, label=\emph{(\roman*)}, labelindent=0pt, itemsep=.1in}

\usepackage{hyperref}
\hypersetup{
  colorlinks   = true,
  citecolor    = blue,
  linkcolor    = blue
}

\usepackage{amsmath,amsfonts,amsthm,color,tikz, comment, xcolor, xfrac,amssymb}
\usepackage{mathrsfs}
\usepackage{mathtools}
\usepackage[font={scriptsize}]{caption}

\usepackage[left=1in, right=1in, top=1.1in,bottom=1.1in]{geometry}
\usepackage{amsrefs}

\renewcommand{\MR}[1]{~\href{https://mathscinet.ams.org/mathscinet-getitem?mr=MR#1}{MR#1}.}

\BibSpec{article}{%
    +{}  {\PrintAuthors}                {author}
    +{,} { \textit}                     {title}
    +{.} { }                            {part}
    +{:} { \textit}                     {subtitle}
    +{,} { \PrintContributions}         {contribution}
    +{.} { \PrintPartials}              {partial}
    +{,} { }                            {journal}
    +{}  { \textbf}                     {volume}
    +{}  { \PrintDatePV}                {date}
    +{,} { \issuetext}                  {number}
    +{,} { \eprintpages}                {pages}
    +{,} { }                            {status}
    +{,} { available at \eprint}        {eprint}
    +{}  { \parenthesize}               {language}
    +{}  { \PrintTranslation}           {translation}
    +{;} { \PrintReprint}               {reprint}
    +{.} { }                            {note}
    +{.} {}                             {transition}
    +{}  {\SentenceSpace \PrintReviews} {review}
}

\numberwithin{equation}{section}  

\newcommand{\1}{{\bf 1}}

\newcommand{\ud}{\ensuremath{\mathrm{d}}}
\newcommand{\Norm}[1]{\left\|  #1   \right\|}

\newcommand{\ce}{\mathcal{E}}
\newcommand{\cf}{\mathcal{F}}
\newcommand{\cg}{\mathcal{G}}
\newcommand{\ch}{\mathcal{H}}

\newcommand{\cl}{\mathcal{L}}

\newcommand{\cs}{\mathcal{S}}

\newcommand{\cz}{\mathcal{Z}}

\newcommand{\bE}{\mathbf{E}}

\newcommand{\bP}{\mathbf{P}}

\newcommand{\bs}{\mathbf{s}}
\newcommand{\bt}{\mathbf{t}}
\newcommand{\bx}{\mathbf{x}}
\newcommand{\by}{\mathbf{y}}

\newcommand{\NN}{\mathbb{N}}
\newcommand{\ZZ}{\mathbb{Z}}
\newcommand{\RR}{\mathbb{R}}
\newcommand{\PP}{\mathbb{P}}

\newcommand{\E}{\mathbb{E}}
\newcommand{\EE}{\E}
\newcommand{\R}{\RR}
\DeclareMathOperator{\Var}{Var}

\newcommand{\ep}{\varepsilon}

\newcommand{\vp}{\varphi}
\newcommand{\om}{\omega}

\newtheorem{theorem}{Theorem}[section]

\theoremstyle{definition}

\newtheorem{condition}[theorem]{Condition}

\newtheorem{corollary}[theorem]{Corollary}

\newtheorem{definition}[theorem]{Definition}

\newtheorem{hypothesis}[theorem]{Hypothesis}
\newtheorem{lemma}[theorem]{Lemma}
\newtheorem{notation}[theorem]{Notation}

\newtheorem{proposition}[theorem]{Proposition}

\theoremstyle{remark}
\newtheorem{remark}[theorem]{Remark}

\theoremstyle{remark}
\newtheorem{example}[theorem]{Example}

\begin{document}

\title[Directed polymers in a Gaussian environment]{A class of $\mathbf{d}$-dimensional directed polymers\\ in a Gaussian environment}

\author[L. Chen]{Le Chen}
\address{L. Chen: Department of Mathematics and Statistics, Auburn University, Auburn, United States}
\email{le.chen@auburn.edu}

\author[C. Ouyang]{Cheng Ouyang}
\address{C. Ouyang: Department of Mathematics, Statistics, and Computer Science, University of Illinois at Chicago, Chicago, United States}
\email{couyang@uic.edu}

\author[S. Tindel]{Samy Tindel}
\address{S. Tindel: Department of Mathematics, Purdue University, West Lafayette, United States}
\email{stindel@purdue.edu}

\author[P. Xia]{Panqiu Xia}
\address{P. Xia: School of Mathematics, Cardiff University, Cardiff, United Kingdom}
\email{xiap@cardiff.ac.uk}

\date{March 2026}

\begin{abstract}
  We introduce and analyze a broad class of continuous directed polymers in
  $\R^d$ driven by Gaussian environments that are white in time and spatially
  correlated, under Dalang's condition. Using an It\^o-renormalized
  stochastic-heat-equation representation, we establish structural properties of
  the partition function, including positivity, stationarity, scaling,
  homogeneity, and a Chapman--Kolmogorov relation. On finite time intervals, we
  prove Brownian-type pathwise behavior, namely H\"older continuity and
  identification of the quadratic variation. We then obtain a sharp
  measure-theoretic dichotomy: the quenched polymer measure is singular with
  respect to Wiener measure if and only if $\widehat f(\R^d)=\infty$
  (equivalently, the noise is non-trace-class), and it is equivalent otherwise.
  Finally, in dimension $d\ge 3$, we prove diffusive behavior at large times in
  the high-temperature regime. This extends the Alberts--Khanin--Quastel
  framework from the $1+1$ white-noise setting to higher-dimensional Gaussian
  environments with general spatial covariance.
\end{abstract}

\maketitle

  {\hypersetup{linkcolor=black}
    \tableofcontents
  }

\section{Introduction}

This paper introduces and analyzes a broad class of directed polymers in
continuous time and space, driven by Gaussian environments that are white in
time and spatially correlated. Our aim is to go beyond the classical
one-dimensional white-noise framework and build a robust theory under Dalang's
condition in arbitrary dimension. We establish the structural properties of the
partition function needed to define the polymer measure, and then derive
pathwise and measure-theoretic consequences, including a sharp singularity
criterion and a high-temperature diffusive limit in dimension $d\ge 3$. In this
introduction, we provide background in Section~\ref{S:background}, present the
model in Sections~\ref{S:she}--\ref{S:polymer-models}, and summarize the
main results and proof strategy in Section~\ref{S:results}.

\subsection{Background}\label{S:background}
Directed polymers in random environments can be seen as an emblematic
example of disordered systems. The polymer measure is typically obtained by
exponentially tilting a reference path measure (such as a random walk or
Brownian motion) by an accumulated random potential along the path. At the heart
of the problem, one tries to quantify the competition between the energy given
by the random environment and the internal entropy of the system (represented by
the random walk or Brownian measure). Originating in statistical physics, these
models have become a central object of study in probability theory. This is
partially due to the fact that their study combines fascinating techniques,
including martingale theory, concentration inequalities, large deviations,
Malliavin calculus, stochastic PDEs, and intersection local times.

\subsubsection{Discrete polymer models}

Most of the results available for directed polymers have been obtained in a
discrete setting. A typical example of a polymer model on $\NN\times\ZZ^{d}$ is
given as follows: consider a sequence $\{\om(n,x); \, n\ge 0, x\in\ZZ^{d}\}$ of
i.i.d.\ random variables satisfying some mild moment assumptions, and a
temperature parameter $\beta > 0$. We define a measure $\bP_N^{\beta, \om}$ on
the space of paths from $\{0,\ldots,N\}$ to $\ZZ^{d}$, by specifying the finite
dimensional distributions of the canonical process $X$. Those finite dimensional
distributions are expressed by considering bounded functions $\phi$ defined on
$\ZZ^{d\times k}$ and some times $0 \le j_1 < \cdots < j_k \le N$. Then we set
\begin{equation}\label{E:discrete-polymer}
  \bE_N^{\beta,\om}
  \left[\phi(X_{j_1},\ldots ,X_{j_k})\right]
  = \frac{1}{Z_{N}}
  \bE_0\left[e^{\beta{\sum_{m=0}^{N} \om(m,X_m)}} \, \phi(X_{j_1},\ldots ,X_{j_k})\right] \, ,
\end{equation}
where the renormalization constant $Z_{N}$ is given by $\bE_0
[\exp(\beta{\sum_{m=0}^{N} \om(m,X_m)}) ]$, and $\bE_{0}$ is the expectation
with respect to a standard random walk $X$ starting from the origin.

There is now a wealth of properties established in the literature, allowing a
deep understanding of the measures $\bP_N^{\beta, \om}$ defined by~\eqref{E:discrete-polymer}. A
comprehensive account can be found in the Saint-Flour lecture notes of
Comets~\cite{comets:17:directed}, which remain a standard reference. A
fundamental organizing principle in the study of directed polymers is the
distinction between \emph{weak disorder} and \emph{strong disorder} regimes.
Roughly speaking, in weak disorder the polymer measure remains close to the
reference path measure (entropy prevails), while in strong disorder the path
localizes around favorable regions of the environment (energy dominates). One of
the earliest and most influential results in this direction is the
identification of the weak and strong disorder phases via the behavior of the
normalized partition function, which forms a positive martingale. In dimensions
$d \geq 3$, this martingale is non-degenerate at high temperature, leading to
weak disorder, while in low dimensions ($d=1,2$) strong disorder prevails at all
temperatures~\cites{imbrie.spencer:88:diffusion, carmona.hu:02:on}.

Beyond this qualitative picture, significant progress has been made in
understanding localization, free energy fluctuations, and pathwise growth. It is
impossible to briefly summarize the vast amount of contributions available at
this point
in this direction. We will just mention the sharp estimates on the partition
function $Z_{N}$ obtained in~\cites{berger.lacoin:17:high-temperature,
lacoin:10:new}, as well as the deep study of overlaps carried out
in~\cite{bates.chatterjee:20:localization}. Pathwise localization of the polymer
in the strong disorder regime is achieved in~\cite{bates:21:full-path}. Overall,
one can claim that recent advances have clarified the role of dimension,
temperature, and correlation structure of the environment, and introduced
techniques (such as fractional moment bounds and coarse-graining) that have
since become somewhat standard. Even more recently, discrete polymer techniques
have been fundamental in getting limiting results for the so-called
2d-stochastic heat flow~\cite{caravenna.sun.ea:23:critical}.

\subsubsection{Continuous polymer models}

The body of work on continuous-time, continuous-space directed polymers is much
scarcer than for the discrete setting. However, it should be stressed that
continuous models are particularly natural from an analytical perspective and
admit a direct connection with stochastic PDEs through Feynman--Kac
representations. Continuous models also allow the use of powerful methods which
are less straightforward to apply in the discrete setting, such as stochastic
calculus or estimates on intersection local times. Among the works dealing with
continuous polymers, we mention the following:
\begin{enumerate}[label=\textbf{(\roman*)}]

  \item The article~\cite{rovira.tindel:05:on} defines a Brownian polymer in an
    environment which is white in time and smooth in space. The
    sequel~\cite{lacoin:11:influence} provides refined estimates on the
    partition function and~\cite{bezerra.tindel.ea:08:superdiffusivity} exhibits
    superdiffusive behavior in dimension 1+1.

  \item The contributions~\cites{comets.yoshida:05:brownian,
    comets.yoshida:13:localization} discuss phase diagrams for the weak/strong
    disorder regimes of a Brownian polymer in an environment described by a
    Poisson point process in $\RR^{d}$. The methods used therein hinge on
    stochastic calculus tools for both the Brownian motion and the point
    process.

  \item The papers~\cites{alberts.khanin.ea:14:continuum,
    alberts.khanin.ea:14:intermediate} are probably the closest to our setting
    in the current contribution. The authors consider a 1+1-dimensional polymer
    model where the environment is given by a space-time white noise. This means
    that the polymer measure cannot be defined directly by the Feynman--Kac type
    representation~\eqref{E:discrete-polymer}. The authors replace~\eqref{E:discrete-polymer} by a
    representation relying on the stochastic heat equation with wedge initial
    condition. This will also be our point of view in~\eqref{E:fdd} below.

\end{enumerate}

Within this context, our paper extends the existing continuous polymer theory in
three directions. First, in contrast with the $1+1$ white-noise setting
of~\cite{alberts.khanin.ea:14:continuum}, we work in a general flat ambient
space $\R^d$ for all $d\ge 1$. Second, instead of smooth spatial correlations as
in~\cite{rovira.tindel:05:on}, we allow general spatial covariances under
Dalang's condition~\eqref{E:Dalang}, including non-trace-class noise regimes.
Third, in this unified framework we prove both local and global polymer
properties (regularity, singularity/equivalence criterion, and high-temperature
diffusive behavior in $d\ge 3$). Because the environment may be singular in
space, the polymer measure is built through It\^{o}-type renormalization of
stochastic-heat-equation transition densities, following the
Alberts--Khanin--Quastel (AKQ) framework.
Before describing our class of models more specifically, let us justify briefly
our endeavor:

\begin{enumerate}[label=\textbf{(\roman*)}]

  \item By considering a wide range of environments, we obtain access to a
    correspondingly broad range of polymer exponents. This type of variation in
    response to random environments has already been observed in the context of
    intermittency for parabolic or hyperbolic equations; see, for example,
    \cites{balan.chen.ea:22:exact, chen.dalang:15:moments,
    chen.kim:19:nonlinear, hu.huang.ea:15:stochastic}.

  \item At the same time, some aspects of polymer behavior are expected to be
    universal across different covariance structures. In this paper, we verify
    this mechanism for diffusive behavior at high temperature in dimensions
    $d \geq 3$. Other quantities (wandering or
    fluctuation exponents) are still open.

  \item Our It\^{o}-type polymers interpolate between the
    well-behaved classical polymer model in~\eqref{E:discrete-polymer} and polymers related to singular PDEs
    (see, e.g., \cite{perkowski.rosati:19:kpz} for the definition of such a
    singular polymer in dimension $1+1$). To the best of our knowledge, the
    basic properties of singular polymers remain largely unexplored.
    While~\cite{caravenna.sun.ea:23:critical} uses polymers as a tool for the
    construction of their very singular $2d$ stochastic heat flow, the singular
    polymer itself is not addressed therein. Our contribution develops
    foundational properties in this direction.

  \item Dealing with polymers defined in the It\^{o} sense enables us to invoke
    a large set of powerful tools, such as stochastic calculus, estimates for
    intersection local times, or Malliavin calculus. This framework can yield sharp
    estimates for various polymer exponents.

  \item Recent advances~\cites{baudoin.ouyang.ea:23:parabolic,
    baudoin.chen:25:parabolic} have analyzed the impact of ambient geometries on
    the behavior of parabolic Anderson models. The techniques we develop here
    can also be extended to geometric settings and provide
    new insights into the interactions between polymers and geometry, beyond the
    discrete cases addressed in~\cite{cosco.seroussi.ea:21:directed}.

\end{enumerate}

With these justifications in mind, we proceed to a description of our model of
interest. We start with some general notation concerning the stochastic heat
equation in Section~\ref{S:she}, before giving a full description of the
polymer measure in Section~\ref{S:polymer-models}.

\subsection{The stochastic heat equation}\label{S:she}

In order to describe our polymer, we first introduce the kind of random
environment we wish to consider as well as the related stochastic heat equation.
Note that the type of noisy environment considered here is quite standard in the
stochastic PDE literature; see~\cite{dalang:99:extending} for a basic reference.

\begin{definition}\label{D:Noise}
  We consider a centered Gaussian noise $\dot W$ defined on a complete
  probability space $(\Omega,\mathcal{F}, \PP)$. The noise $\dot{W}$ is white
  in time and homogeneously colored in space. Its covariance structure is given
  by the following identity, valid for any continuous and rapidly decreasing
  functions $\psi$ and $\phi$:
  \begin{align}\label{E:Cor}
    \E \left[ W\left(\psi\right) W\left(\phi\right) \right]
    = \int_0^\infty \ud s \int_{\R^d} \Gamma(\ud x)(\psi(s,\cdot)*\widetilde{\phi}(s,\cdot))(x),
  \end{align}
  where $\widetilde{\phi}(x) \coloneqq \phi(-x)$, and the symbol $*$ refers to
  the convolution in the spatial variable. In~\eqref{E:Cor}, the generalized
  function $\Gamma$ is a nonnegative and nonnegative definite tempered measure
  on $\R^d$ that is commonly referred to as the \emph{correlation measure}. The
  Fourier transform\footnote{We use the following convention and notation for
    Fourier transform: $\mathcal{F}\psi(\xi) = \widehat{\psi}(\xi) \coloneqq
    \int_{\R^d} e^{-i x\cdot \xi}\psi(x) \ud x$ for any Schwartz test function
  $\psi\in \mathcal{S}(\R^d)$.} of $\Gamma$ (in the generalized sense) is also
  a nonnegative and nonnegative definite tempered measure, which is usually
  called the \emph{spectral measure} and is denoted by $\widehat{f}(\ud \xi)$.
  Moreover, in the case where $\Gamma$ has a density $f$ (namely $\Gamma(\ud x)
  = f(x) \ud x$), we write $\widehat{f}(\ud \xi)$ as $\widehat{f}(\xi) \ud
  \xi$. In the sequel we will work under the so-called \emph{Dalang's
  condition} on $\widehat{f}$:
  \begin{align}\label{E:Dalang}
    \Upsilon(\gamma) \coloneqq \left(2\pi\right)^{-d} \int_{\R^d} \frac{\widehat{f}(\ud\xi)}{\gamma+\lvert\xi\rvert^2}<\infty \, ,
    \quad \text{for some (and hence all) $\gamma>0$.}
  \end{align}
\end{definition}

\begin{remark}\label{R:f0-convention}
  The total mass $\widehat{f}\left(\R^d\right)$ is a canonical quantity that
  depends only on the correlation measure $\Gamma$. When
  $\widehat{f}\left(\R^d\right) < \infty$, it also equals $\lim_{t\downarrow 0}
  k(t)$ (see~\eqref{E:k_t} below) and can be interpreted as the spatial
  variance of the noise at a point. If $\Gamma$ admits a bounded continuous
  density $f$, then $\widehat{f}\left(\R^d\right)=f(0)$; however, in general
  the value at one point is not intrinsic, so we avoid writing $f(0)$ unless
  such regularity is explicitly assumed.
\end{remark}

\begin{remark}\label{R:Filtration}
  For our stochastic calculus considerations, we equip the aforementioned
  probability space $(\Omega, \mathcal{F}, \PP)$ with the \textit{standard
  filtration} $\{\mathcal{F}_t\}_{t \ge 0}$, which is initially generated by
  the noise and then augmented by the $\sigma$-field $\mathcal{N}$ consisting
  of all $\PP$-null sets in $\mathcal{F}$. In addition, as introduced in
  Definition~\ref{D:Noise}, we use the notation $\PP$ and $\E$ to denote the
  probability and expectation with respect to $W$, respectively, throughout
  this paper.
\end{remark}

The partition function of our polymer is closely related to the stochastic heat
equation in $\R^{d}$, for which we recall a few facts (we refer
to~\cite{dalang:99:extending} again for further details). Namely for the noise
$W$ in Definition~\ref{D:Noise}, a measure valued initial condition $\mu$ and a
temperature parameter $\beta>0$, we consider the equation
\begin{equation}\label{E:SHE}
  \left(\dfrac{\partial}{\partial t}-\dfrac{1}{2}\Delta\right) u(t,x) = \beta \, u(t,x) \dot W(t,x),\qquad
  \text{for $x\in \R^d$ and $t>0$}.
\end{equation}
This equation has to be interpreted in the mild sense. That is, the solution
to~\eqref{E:SHE} is understood as the following mild solution:
\begin{align}\label{E:mild}
  u(t,x) = J_0(t,x) + \beta \int_0^t \int_{\R^d} p(t-s,x-y)\, u(s,y) \, W(\ud s,\ud y),
\end{align}
where the stochastic integral is interpreted as the Walsh integral
(\cites{dalang:99:extending, walsh:86:introduction}) and $J_0(t,x)$ refers to
the solution to the homogeneous equation, namely,
\begin{align}\label{E:def-J0}
  J_0(t,x) = J_0(t,x;\mu) \coloneqq \int_{\R^d} p(t,x-y)\mu(\ud y) = [p_t*\mu] (x).
\end{align}
Referring to~\eqref{E:def-J0} above, the measure $\mu$ has to be thought of as
the initial condition for $u$, namely $u(0,\cdot)=\mu$. Here and throughout the
rest of the paper, we use $p(t,x)$, or sometimes $p_t(x)$, to denote the
standard heat kernel over $\R^{d}$:
\begin{align}\label{E:heat}
  p(t,x) = p_t(x) \coloneqq \frac{1}{(2\pi t)^{d/2}} \exp\left(- \frac{|x|^2}{2t}\right),
\end{align}
where $|x| \coloneqq \sqrt{x_1^2 + \dots + x_d^2}$.

Existence and uniqueness of the solution to the stochastic heat equation, for
bounded (resp.~rough) initial conditions, were established
in~\cite{dalang:99:extending} (resp.~\cite{chen.kim:19:nonlinear}). We label the
result we need here for later use.

\begin{theorem}\label{T:she-well-posedness}
  Consider equation~\eqref{E:SHE} starting from an initial condition $\mu \ge
  0$, which is a nonnegative Borel measure on $\R^d$ such that
  \begin{align}\label{E:LinearExpID}
    \int_{\R^d} e^{-\gamma |x|} \, \mu(\ud x) < \infty \, , \quad \text{for all $\gamma >0$.}
  \end{align}
  Recall that the covariance of the noise is given by~\eqref{E:Cor}, and is
  specified by a correlation measure $f$ satisfying~\eqref{E:Dalang}. Then for
  any $\beta>0$, there exists a unique random field solution to
  equation~\eqref{E:SHE}, interpreted as in~\eqref{E:mild}.
\end{theorem}

In order to transition to polymer dynamics, we will extend our notion of
solution of the stochastic heat equation to solutions starting from a Dirac
measure at an arbitrary time $s \ge 0$. Specifically, for $s \ge 0$ and $y \in
\R^d$, we consider the equation
\begin{equation}\label{E:PAM-Delta}
  \left(\dfrac{\partial}{\partial t}-\dfrac{1}{2}\Delta\right) v(t,x) = \beta \, v(t,x) \dot W(t,x),\quad
  \text{for $x\in \R^d$ and $t>s$ with $v(s,x) = \delta_{y}(x)$.}
\end{equation}
This equation falls into the framework of Theorem~\ref{T:she-well-posedness},
and thus admits a unique solution in the sense of~\eqref{E:mild}. This solution
will feature prominently in the definition of the partition function for our
polymer measure.

\subsection{The polymer measure}\label{S:polymer-models}
Throughout the paper, we use $\cs_k([0,t])$ to denote the $k$-dimensional
simplex in $[0,t]^k$, namely,
\begin{align}\label{E:Simplex}
  \cs_{k}\left( (0,t]\right) = \{\bs_k \in [0,t]^k, 0< s_1 < \dots < s_k \le t\}.
\end{align}
We will also use the following conventions:
\begin{equation}\label{E:Simplex-Conv}
  \cs_{0}((a,b]) \coloneqq \{\varnothing\}, \qquad
  \1_{\cs_{0}((a,b])}(\varnothing) = 1, \quad \text{and} \quad
  \cs_k \coloneqq \cs_k ((0,1]).
\end{equation}
With this notation in hand, recall that in continuous time, whenever the noise
$\dot{W}$ is sufficiently regular in the spatial variable, the polymer measure
can be defined through Feynman-Kac representations. This is not the case anymore
when $\dot{W}$ is a distribution in space and equation~\eqref{E:SHE} is defined
in the It\^o sense. However, Alberts, Khanin, and
Quastel~\cite{alberts.khanin.ea:14:continuum} provided heuristic arguments based
on the discrete polymer model, which yield a reasonable definition for an
It\^{o}-type continuous polymer in $\R$. This polymer is, roughly speaking,
defined by transition densities derived from the solution
to~\eqref{E:PAM-Delta}. We adopt this perspective in the present work, which
leads to the following definition of the polymer measure.

\begin{definition}\label{D:PolyMeasure}
  For $0 \le s < t \le 1$ and $x, y \in \R^d$, we denote by $\cz(s,y ; \, t,x;
  \, \beta)$ the mild solution to equation~\eqref{E:PAM-Delta}, namely,
  \begin{align}\label{E:Mild-Z}
    \cz(s,y ; \, t,x; \, \beta) \coloneqq p_{t-s}(x-y) + \beta \int_s^t \int_{\R^d} p_{t-r}(x-z) \, \cz(s,y ; \, r,z; \, \beta) \, W(\ud r, \ud z).
  \end{align}
  We also set
  \begin{equation}\label{E:Z_star}
    \begin{aligned}
      \cz(s,y; \, t, * ; \, \beta)
       & \coloneqq
      \int_{\R^{d}} \cz(s,y; \, t, x ; \, \beta) \, \ud x, \\
      \cz_t
       & \coloneqq
      \cz(0,0; \, t, * ; \, \beta).
    \end{aligned}
  \end{equation}
  Then, we introduce the following objects:
  \begin{enumerate}[label=\textbf{(\roman*)}]

    \item\label{I:PolyMeasure:i} For a fixed realization $W$ of the noise in
      Definition~\ref{D:Noise}, and a temperature parameter $\beta > 0$, we
      define a measure $\bP_\beta^W$ on $C ([0,1]; \R^d)$ equipped with the
      Borel $\sigma$-algebra $\mathcal{G}$, by specifying the finite dimensional
      distributions of the canonical process $X$ as follows:
      \begin{equation}\label{E:fdd}
        \bP_\beta^W \!\left( X_{t_{1}} \in \ud x_{1}, \ldots,  X_{t_{k}} \in \ud x_{k} \right)
        =
        \frac{\ud x_1 \cdots \ud x_k }{\cz(0,0; \, 1, *;\, \beta)} \,
        \prod_{j=0}^{k} \cz \left( t_{j}, x_{j}; \, t_{j+1}, x_{j+1}  ; \, \beta\right)
        \, ,
      \end{equation}
      where $\left(\bt_k, \bx_k\right) = \left(t_1, \dots, t_k; x_1, \dots
      x_k\right)$ designates a generic element of $\cs_k ([0,1]) \times \R^{k
      d}$, with the convention that
      \begin{equation}\label{E:Convention-t}
        (t_{0}, x_{0}) \coloneqq (0, 0) \quad \text{and} \quad
        (t_{k + 1}, x_{k + 1}) \coloneqq (1, *).
      \end{equation}
      The measure $\bP_\beta^W$ is referred to as the \emph{quenched polymer
      measure}. Furthermore, the degenerate case $\bP_0^{W}$ corresponds to the
      Wiener measure on $C([0,1]; \R^d)$. For simplicity, we denote this measure
      by $\bP_0$, omitting the dependence on $W$.

    \item\label{I:PolyMeasure:ii} The so-called \emph{annealed measure} is defined as
      the product measure $\PP_{\beta} \coloneqq \PP \otimes \bP_\beta^W$, where
      $\PP$ is the probability introduced in Definition~\ref{D:Noise}. We also
      set $\PP_0 \coloneqq \PP \otimes \bP_0$. In other words, for any bounded
      functional $F \colon C([0,1]; \R^d) \to \R$, we have
      \begin{align*}
        \E_{\beta} [F(X)] = \E \left( \bE_\beta^W [F(X)]\right ).
      \end{align*}

    \item In the sequel we denote by $X$ the canonical process, which,
      conditioned on $W$, obeys the quenched measure $\bP_\beta^W$.

  \end{enumerate}
\end{definition}

\begin{remark}
  Let us check that equation~\eqref{E:fdd} defines a probability measure,
  considering a simple example. Namely suppose $k = 1$ and look at an instant
  $t\in[0,1]$. Then according to~\eqref{E:fdd} we have
  \begin{equation}\label{E:polymer-one-time-density}
    \bP_\beta^W \!\left( X_{t} \in \ud x \right)
    =
    \frac{1}{\cz(0,0; \, 1, *;\, \beta)} \,
    \cz \left( 0, 0; \, t, x  ; \, \beta\right) \cz \left(t , x; \, 1, *  ; \, \beta\right)
    \ud x .
  \end{equation}
  Furthermore, anticipating the Chapman-Kolmogorov equation~\eqref{E:chapman-kolmogorov} below,
  we have
  \begin{equation}\label{E:Z-CK}
    \int_{\R^d} \cz \left( 0, 0; \,t ,x   ; \, \beta\right) \cz \left(t ,x ; \, 1, *  ; \, \beta\right)
    \ud x  = \cz \left( 0, 0; \, 1, *  ; \, \beta\right) .
  \end{equation}
  Integrating equation~\eqref{E:polymer-one-time-density} in $x$ and plugging relation~\eqref{E:Z-CK},
  we thus get
  \begin{equation*}
    \int_{\R^{d}} \bP_\beta^W \!\left( X_{t} \in \ud x \right)
    =1,
  \end{equation*}
  which proves that $\{\bP_\beta^W \!( X_{t} \in \ud x ); \, x\in\R^{d}\}$
  defines a probability measure. We let the reader check that this property is
  also true for multidimensional distributions like in~\eqref{E:fdd}.
\end{remark}

\begin{remark}
  One can also express condition~\eqref{E:fdd} through test functions. That is,
  for any $F \in C_b (\R^{kd})$ and $\mathbf{t}_k \in \mathcal{S}_k ([0,1])$,
  we have (using convention~\eqref{E:Convention-t})
  \begin{align}\label{E:fdd-test}
    \begin{aligned}
      \MoveEqLeft \bE_\beta^W \left[F (X_{t_1}, \dots, X_{t_k} ) \right]         \\
       & = \frac{1}{\cz(0,0; \, 1, *;\, \beta)} \, \int_{\R^{kd}} \ud \mathbf{x}_k
      \; F(x_1, \dots, x_k ) \times \prod_{j=0}^{k} \cz \left( t_{j}, x_{j}; \, t_{j+1}, x_{j+1} ; \, \beta\right) .
    \end{aligned}
  \end{align}
\end{remark}

\begin{remark}\label{R:beta-0}
  In particular, when $\beta = 0$, $\bP_0^W = \bP_0$ is independent of $W$ and
  corresponds to the Wiener measure on $C([0,1]; \R^d)$. Consequently, as
  mentioned in Definition~\ref{D:PolyMeasure}, the annealed measure factorizes
  as $\PP_0 (\ud w, \ud x) = \PP ( \ud w) \, \bP_0 (\ud x) $.
\end{remark}

\begin{remark}
  In Proposition~\ref{P:same-def-polymer} below, we will see that the
  expression~\eqref{E:equiv-finite-dist} can be viewed as a renormalized limit
  of smoothed expressions similar to~\eqref{E:discrete-polymer}. This gives a more
  intuitive version of Definition~\ref{D:PolyMeasure}.
\end{remark}

\subsection{Main results and strategy}\label{S:results}

Our paper develops a systematic theory for the polymer measure in
Definition~\ref{D:PolyMeasure}. To keep the introduction concise, we state the
main results in summary form and refer to the theorem statements for full
details. The core contributions are as follows:

\begin{enumerate}[leftmargin=*, label={\textbf{(\roman*)}}]
  \setlength\itemsep{.1in}

  \item Theorem~\ref{T:transition-properties} develops the structural theory of the partition
    function family in~\eqref{E:Mild-Z}. In particular, it proves H\"older
    continuity, stationarity, scaling, homogeneity, and a
    Chapman--Kolmogorov relation. These properties provide the analytic
    backbone of the polymer construction and all subsequent arguments.

  \item \label{I:main-2} On the unit interval, we prove local Brownian-type
    behavior under the quenched polymer measure $\mathbb{P}_{\beta}^W$ in
    \eqref{E:fdd}. Specifically, Theorem~\ref{T:polymer-holder-regularity} yields
    H\"{o}lder continuity with exponent $\frac12-\epsilon$ for every
    $\epsilon>0$, while Theorem~\ref{T:polymer-quadratic-variation} identifies the
    $(d\times d)$-quadratic variation of $X$ at time $t$ as $tI_d$.

  \item In contrast to item~\ref{I:main-2}, Theorem~\ref{T:singularity-criterion} gives a
    complete and sharp criterion for the singularity of $\mathbb{P}_{\beta}^W$
    with respect to the Wiener measure $\mathbb{P}_0$: namely,
    $\mathbb{P}_{\beta}^W \perp \mathbb{P}_0$ if and only if
    $\widehat{f}(\mathbb{R}^d)=\infty$, where $\widehat{f}(\mathbb{R}^d)$ is the
    quantity introduced in Remark~\ref{R:f0-convention}. This condition admits a
    natural interpretation: $\widehat{f}(\mathbb{R}^d)<\infty$ is equivalent to
    the driving noise $\dot W$ being trace class in the sense of Da
    Prato--Zabczyk \cite{da-prato.zabczyk:14:stochastic}. Consequently,
    Theorem~\ref{T:singularity-criterion} shows that $\mathbb{P}_{\beta}^W$ is singular with
    respect to $\mathbb{P}_0$ as soon as the noise fails to be trace class.
    While the theorem yields a definitive measure-theoretic dichotomy, it
    remains an intriguing open problem to identify an explicit set $A\subset
    \mathcal{B}(C[0,1])$ such that $\mathbb{P}_{\beta}^W(A)=1$ and
    $\mathbb{P}_0(A)=0$. Any such characterization must rely on subtler pathwise
    properties than classical H\"older regularity or quadratic variation.

  \item In Section~\ref{S:Diffusive}, we prove a large-time diffusive limit:
    in dimension $d \geq 3$, the canonical process under $\mathbb{P}_W^{\beta}$
    scales to Brownian motion in the high-temperature regime
    (Theorem~\ref{T:diffusive-behavior}). The proof proceeds through regularized and
    renormalized polymer measures and a limiting argument.

\end{enumerate}

At a technical level, all of our results rely on quantitative control of the
partition field $\cz(s, y;\, t, x;\, \beta)$ over $(s,t)\in\mathcal{S}_2([0,T])$
and $(x,y)\in\R^{2d}$. The key inputs are sharp moment estimates for its
increments, together with chaos-expansion arguments adapted to singular initial
data. Those bounds are then combined with martingale methods to derive the
measure-theoretic conclusions (equivalence/singularity) and the large-time
diffusive limit. Finally, in the regularization-renormalization part of the
proof, Hölder regularity of $\cz$ provides the stability needed to pass to the
limit (see the proof of Lemma~\ref{L:F-K_Mom_Polymer} and
Remark~\ref{R:holder-relevance}).

This article is a first step toward a broader program on the effect of singular
covariance structures on polymer behavior. Natural next objectives include the
implementation of fractional-moment methods, sharper free-energy estimates, and
the analysis of overlaps and path localization. We will pursue these directions
in forthcoming work.

\subsection{Outline of the paper}
The remainder of the paper is organized as follows. In
Section~\ref{S:preliminary}, we introduce preliminary notions concerning the
Gaussian noise, the associated Wiener space, and the stochastic heat equation,
and we also establish basic properties of the partition function.
Section~\ref{S:first-properties-of-polymer} is devoted to local pathwise
properties of the polymer measure, including H\"older regularity, quadratic
variation, and a criterion for singularity with respect to Wiener measure. In
Section~\ref{S:Diffusive} we study the large-time behavior of the polymer, and
prove diffusive scaling in the high-temperature regime in dimensions $d \geq 3$.
Two appendices collect auxiliary analytic estimates and a compactness lemma used
in the paper.

\section{Preliminary results}\label{S:preliminary}

This section is first devoted (in Section~\ref{S:wiener-space}) to recall
basic notions of stochastic analysis on the Wiener space related to our noise
$W$. Subsequently, in Section~\ref{S:prop-heat-eq}, we recall some core
properties of the stochastic heat equation. Furthermore, we present and prove
the long Theorem~\ref{T:transition-properties}, which encapsulates the most pertinent
characteristics of the field $\cz$ as given in Definition~\ref{D:PolyMeasure}.

\subsection{Wiener space related to \texorpdfstring{$W$}{W}}\label{S:wiener-space}

Recall from Definition~\ref{D:Noise} that we consider a Gaussian noise $\dot{W}$
on a complete probability space $(\Omega,\cf,\PP)$, encoded by a centered
Gaussian family $\{W(\vp) ; \, \vp\in {\mathcal{S}}(\R_{+}\times\R^{d})\}$. The
covariance structure of
\begin{align}\label{E:noise-formal-derivative}
  \dot{W} (t,x) = \frac{\partial^{d + 1}}{\partial t \partial x_1 \cdots \partial x_d} W (t,x)
\end{align}
is given by~\eqref{E:Cor}. When the spatial covariance $\Gamma$ in~\eqref{E:Cor}
admits a density $f$, this can be read as
\begin{align}\label{E:W-Density}
  \E \left[W(\psi) W(\phi)\right] = \int_0^{\infty} \ud s \int_{\R^{2d}} \psi(s, x) f( x - y) \phi (s, y) \ud x \ud y .
\end{align}
Alternatively, this covariance function can be expressed in Fourier modes as
\begin{equation}\label{E:W-Space}
  \E \left[ W\left(\psi\right) W\left(\phi\right) \right]
  = \int_{\R_{+}}\int_{\R^{d}}
  \mathcal{F}\psi(s,\xi) \, {\bar{\mathcal{F}}} \phi(s,\xi)  \, \widehat{f}(\ud \xi) \, \ud s \, ,
\end{equation}
where $\widehat{f}$ designates the spatial Fourier transform of $\Gamma$. One
can thus consider $W$ as an isonormal Gaussian family $\left\{W(\vp) ; \, \vp\in
\ch\right\}$ on a space $\ch$ obtained as the completion of Schwartz functions
with respect to the inner product given by the right-hand side
of~\eqref{E:W-Space}. Note that in the above notation we adopt the convention,
in connection with~\eqref{E:noise-formal-derivative}, that
\begin{align}\label{E:Noise-Convention}
  W(\ud t, \ud x) = \dot{W} (t, x) \ud t \ud x, \quad \text{for all } (t, x) \in \R_+ \times \R^d ,
\end{align}
and therefore,
\begin{align*}
  W(\phi) = \int_{\R_+} \int_{\R^d} \phi (t,x) W(\ud t, \ud x)
  = \int_{\R_+} \int_{\R^d} \phi (t,x)  \dot{W} (t, x) \ud t \ud x, \quad \text{for all } \phi \in {\mathcal{S}}(\R_{+}\times\R^{d}).
\end{align*}

We refer to~\cite{nualart:06:malliavin} for a comprehensive discussion of
Malliavin calculus in the context of a Gaussian process. In this article, we
concentrate on the fundamental definitions that enable us to present our
results in subsequent sections. Let us begin by introducing the derivative
operator in the sense of Malliavin calculus, which will be denoted by $D$. That
is consider $f \in C^{\infty}_p \left(\R^n\right)$, where $C^{\infty}_p
\left(\R^n\right)$ is the set of functions $f$ such that $f$ and all its
partial derivatives have polynomial growth. A smooth and cylindrical random
variable $F$ will then be of the form
\begin{equation*}
  F = f(W(\phi_1),\dots,W(\phi_n))\,,
\end{equation*}
with $\phi_i \in \mathcal{H}$. In this case, $DF$ is the $\mathcal{H}$-valued
random variable defined by
\begin{equation*}
  DF = \sum_{j=1}^n\frac{\partial f}{\partial x_j}(W(\phi_1),\dots,W(\phi_n)) \, \phi_j\,.
\end{equation*}
The operator $D$ is closable from $L^2(\Omega)$ into $L^2(\Omega;
\mathcal{H})$, and we define the Sobolev space $\mathbb{D}^{1,2}$ as the
closure of the space of smooth and cylindrical random variables under the norm
\[
  \Norm{DF}_{1,2}=\left( \E[F^2]+\E\left[\Norm{DF}^2_{\mathcal{H}}\right] \right)^{1/2}\,.
\]

For any integer $n \ge 0$ we denote by $\mathcal{H}_n$ the $n$th Wiener chaos of
$W$. Recall that $\mathcal{H}_0$ is simply $\R$ and for $n \ge 1$, $\mathcal
{H}_n$ is the closed linear subspace of $L^2(\Omega)$ generated by the random
variables $\{ H_n(W(h));h \in \mathcal{H}, \|h\|_{\mathcal{H}} = 1 \}$, where
$H_n$ is the $n$th Hermite polynomial. For any $n\ge 1$, we denote by
$\mathcal{H}^{\otimes n}$ (resp. $\mathcal{H}^{\odot n}$) the $n$th tensor
product (resp. the $n$th symmetric tensor product) of $\mathcal{H}$. Then, the
mapping $I_n(h^{\otimes n})= H_n(W(h))$ can be extended to a linear isometry
between $\mathcal{H}^{\odot n}$ (equipped with the modified norm $\sqrt{n!}\|
\cdot\|_{\mathcal{H}^{\otimes n}}$) and $\mathcal{H}_n$.

Consider now a random variable $F\in L^2(\Omega)$ which is measurable with
respect to the $\sigma$-field $\mathcal{F}^W$ generated by $W$. This random
variable can be expressed as
\begin{equation}\label{E:chaos-dcp}
  F = \E \left[ F\right] + \sum_{n=1} ^\infty I_n(f_n),
\end{equation}
where the series converges in $L^2(\Omega)$. Moreover, the functions $f_n \in
\mathcal{H}^{\odot n}$ are determined by $F$.  Identity~\eqref{E:chaos-dcp} is
called the Wiener chaos expansion of $F$. We also recall the following product
identity for multiple integrals, borrowed
from~\cite{nualart:06:malliavin}*{Proposition~1.1.3}: for $p, q \ge 1$ and $f
\in \mathcal{H}^{\odot p}$, $g \in \mathcal{H}^{\odot q}$, we have
\begin{align}\label{E:contraction-0}
  I_p (f) I_q (g) = \sum_{r = 0}^{p \wedge q} r! \binom{p}{r} \binom{q}{r} I_{p + q - 2 r} \left(f \otimes_r g \right) ,
\end{align}
where the contraction $f \otimes_r g$ is defined by
\begin{align}\label{E:contraction}
  \left[f \otimes_r g \right] (t_1, \dots, t_{p + q - 2r}) \coloneqq
  \left\langle f (t_{1:p - r}, \cdot), g (t_{p - r + 1: p + q - 2r}, \cdot) \right\rangle_{\mathcal{H}^{\otimes r}},
\end{align}
with the convention
\begin{equation}\label{E:Conv-t}
  t_{i:j} \coloneq (t_i, \dots, t_j), \quad \text{for all positive integers } i \le j.
\end{equation}

In order to state our first basic result on polymer measures (see
Theorem~\ref{T:transition-properties} below), we will need more conditions on the noise $W$.
We start by imposing additional assumptions, the first one being a more
restrictive Dalang type condition.

\begin{hypothesis}[\textbf{Strengthened Dalang's condition}]\label{H:stren}
  Recall that the measure $\widehat{f}$ is introduced in~\eqref{E:W-Space}. In
  addition to Dalang's condition~\eqref{E:Dalang}, we assume that there exists
  $\eta \in (0,1)$ such that the following holds:
  \begin{align}\label{E:stren}
    \Upsilon_{\eta} \coloneqq \left(2\pi\right)^{-d} \int_{\R^d} \frac{\widehat{f}(\ud\xi)}{\left( 1 + \lvert\xi\rvert^2\right)^{1 - \eta}} < \infty.
  \end{align}
\end{hypothesis}

\begin{remark}
  The strengthened Dalang condition~\eqref{E:stren}, also known as the
  \textit{Dalang-Sanz-Sol\'{e}-Sarr\`{a} condition}, was introduced by
  Sanz-Sol\'{e} and Sarr\`{a}~\cite{sanz-sole.sarra:02:holder} in order to
  establish H\"{o}lder continuity of solutions to~\eqref{E:SHE}.
\end{remark}

Our second additional assumption is a scaling property which will translate into
scaling properties for the polymer.

\begin{hypothesis}[\textbf{Scaling condition of $\Gamma$}]\label{H:scale-gmm-simple}
  Let $\Gamma$ be the covariance function introduced in
  Definition~\ref{D:Noise}. We assume that $\Gamma$ admits a density $f$ such
  that, for some constant $\gamma < 2$,
  \begin{align}\label{E:scale-gmm-simple}
    f(x) = |x|^{- \gamma} f \left(\frac{x}{|x|}\right) , \quad
    \text{for all $x \in \R^d \setminus \{0\}$.}
  \end{align}
\end{hypothesis}

\begin{remark}
  The scaling condition~\eqref{E:scale-gmm-simple} is sufficiently general,
  provided that $\Gamma$ admits a density, so that $\Gamma (\ud x) = f(x) \ud
  x$. Indeed, suppose there exists a function $\rho \colon \R_+ \to \R_+$ such
  that
  \begin{align}\label{E:scale-gmm}
    \Gamma (\theta \, \ud x) = \rho(\theta) \, \Gamma (\ud x) , \quad
    \text{for all $\theta > 0$.}
  \end{align}
  Also assume that $\rho$ is differentiable on $\R_{+}\setminus\{0\}$. Then we
  claim that any function $\rho$ that satisfies~\eqref{E:scale-gmm} must take
  the form $\rho (\theta) = \theta^{- \gamma}$ for some $\gamma \in \R$.
  Actually, assuming~\eqref{E:scale-gmm} and continuity, one can show that
  \[
    \rho(1) = 1, \quad \text{and} \quad \rho(\theta_1) \rho(\theta_2) = \rho (\theta_1 \theta_2), \text{ for all $\theta_1, \theta_2 > 0$.}
  \]
  This observation together with the differentiability assumption further
  implies that $\rho$ satisfies the differential equation
  \[
    \dot{\rho} (\theta) = \lim_{\theta' \to \theta} \frac{\rho (\theta') - \rho (\theta)}{\theta' - \theta} = \lim_{\theta' \to \theta}  \frac{[\rho ( \frac{\theta' }{\theta}) - \rho(1)]\rho(\theta) }{ (\frac{\theta' }{\theta} - 1) \theta} = \lim_{\delta \to 1}  \frac{[\rho ( \delta) - \rho(1)] }{ \delta - 1}\times \frac{\rho(\theta)}{\theta} = \frac{\dot{\rho} (1)}{\theta} \rho(\theta),
  \]
  for all $\theta > 0$. Solving this yields $\rho(\theta) = \theta^{-\gamma}$,
  where $\gamma \coloneqq -\dot{\rho}(1)$. Furthermore, assuming that $\Gamma$
  admits a density $f$, an application of Plancherel's theorem shows that the
  Dalang condition~\eqref{E:Dalang} reduces to the condition
  \[
    \int_{0^+} |x|^{- (d - 2)} f(x) \ud x < \infty,
  \]
  which in turn requires $\gamma < 2$. Consequently,
  condition~\eqref{E:scale-gmm} simplifies to~\eqref{E:scale-gmm-simple}.
\end{remark}

We now state the scaling lemma for the noise $\dot{W}$ deduced from
Hypothesis~\ref{H:scale-gmm-simple}.

\begin{lemma}
  Let $\dot{W}$ be the noise in Definition~\ref{D:Noise}, and suppose
  Hypothesis~\ref{H:scale-gmm-simple} holds true. For a parameter $r > 0$, we
  define
  \begin{align}\label{E:W-Scale}
    \dot{W}_r (t, x) \coloneqq \dot{W} (r^2 t , r x), \quad \text{for all $t > 0$ and $x \in \R^d$} .
  \end{align}
  Then, we have the following identity in distribution:
  \begin{align}\label{E:ScaleW}
    \left\{\dot{W}_r (t, x) ; t \ge 0, x \in \R^d \right\} \overset{(d)}{=}
    \left\{r^{- \frac{\gamma}{2} - 1} \dot{W} (t,x);  t \ge 0, x \in \R^d  \right\}.
  \end{align}
\end{lemma}
\begin{proof}
  Both processes in~\eqref{E:ScaleW} are centered Gaussian processes. Therefore,
  it is enough to check that the covariance functions coincide. Now note
  that~\eqref{E:W-Density} can be written formally as
  \[
    \E \left[\dot{W} (t, x) \, \dot{W} (s, y)\right] = \delta (t - s) \, f(x - y) \, ,
    \quad\text{for}\quad
    s,t\ge 0, \ x,y\in\R^{d}.
  \]
  Therefore, taking our Definition~\eqref{E:W-Scale} into account, we also have
  \[
    \E \left[\left(r^{\frac{\gamma}{2} + 1} \dot{W}_{r} (t, x) \right) \left(r^{\frac{\gamma}{2} + 1} \dot{W}_{r} (s, y) \right)\right] = r^{\gamma + 2}  \delta \left(r^2(t - s)\right) f \left(r (x - y) \right).
  \]
  Invoking the classical scaling $\delta (r u) = r^{-1} \delta (u)$ for the
  Dirac measure and our Hypothesis~\ref{H:scale-gmm-simple} for $f$, we end up
  with
  \[
    \E \left[\left(r^{\frac{\gamma}{2} + 1} \dot{W}_{r} (t, x) \right) \left(r^{\frac{\gamma}{2} + 1} \dot{W}_{r} (s, y) \right)\right] = \delta (t - s) f (x - y) ,
  \]
  which proves the desired identity in law~\eqref{E:ScaleW}. Note that a fully
  rigorous version of this proof would require working with~\eqref{E:W-Density}
  using two test functions, $\psi$ and $\phi$. We leave this routine elaboration
  to the interested reader.
\end{proof}

\begin{example}\label{Ex:scaling-examples}
  A typical example of $\Gamma$ satisfying~\eqref{E:scale-gmm-simple} is the
  Riesz kernel, given by $\Gamma (\ud x) = |x|^{-\gamma} \ud x$ with $\gamma \in
  (0, 2\wedge d)$. Another example arises from the correlation structure of
  fractional Brownian sheets, where $\Gamma (\ud x) = \prod_{i = 1}^d H_i (2H_i
  - 2) |x_i|^{2H_i - 2} \ud x$ with $H_i \in (\frac{1}{2}, 1)$ and $|H| = H_1 +
  \dots + H_d > d - 1$, and the corresponding parameter $\gamma = 2d - 2 |H|$;
  see Assumption A and Examples 1.3 and 1.4 of~\cite{balan.chen.ea:22:exact}.
\end{example}

Next, we introduce some heat kernel estimates related to our noise $\dot{W}$
that will be frequently used in this paper. First, let $E_{\alpha,\beta}(z)$
denote the \textit{two-parameter Mittag--Leffler function} defined by
\begin{equation}\label{E:mittag-leffler}
  E_{\alpha,\beta}(z) \coloneqq \sum_{m=0}^{\infty}\frac{z^{m}}{\Gamma(\alpha m+\beta)},
\end{equation}
which is well defined as an entire function for all $\alpha > 0$, $\beta \in
\mathbb{C}$, and $z \in \mathbb{C}$; see, e.g., \cite{podlubny:99:fractional}.
We use the common convention that $E_{\alpha, 1}(z) = E_{\alpha}(z)$. Next we
introduce a function $k$ which plays an important role in our future
computations.

\begin{notation}
  Recall that $\Gamma$ is the covariance measure introduced in
  Definition~\ref{D:Noise}, with Fourier transform $\widehat{f}$. We set $k
  \colon \R_+ \to \R_+$ as the function given by
  \begin{align}\label{E:k_t}
    k (t) \coloneqq \int_{\R^d} p_t (x) \Gamma (\ud x)
    = \left(2\pi\right)^{-d} \int_{\R^d} e^{- \frac{t |\xi|^2}{2}}  \widehat{f} (\ud \xi ) , \quad \text{for all } t \in \R_+.
  \end{align}
  Observe that $k (\cdot)$ is a nonincreasing function on $\R_+$. Next, let
  \begin{align}\label{E:hn-recursion}
    h_0~\equiv 1 \quad  \text{and} \quad
    h_n(t) \coloneqq \int_0^t h_{n - 1} (t - s) k(s) \ud s, \text{ for all } n \ge 1.
  \end{align}
\end{notation}

\begin{example}\label{Ex:riesz-kernel}
  Drawing on our Example~\ref{Ex:scaling-examples}, consider the Riesz kernel $\Gamma (\ud
  x) = f(x) \ud x = |x|^{- \gamma} \ud x$ with $\gamma \in (0, 2 \wedge d)$.
  Then, $\widehat{f}(\ud \xi) = |\xi|^{\gamma - d} \ud \xi$ and according
  to~\eqref{E:k_t}, we have
  \[
    k (t) = C_d \int_{\R^d} e^{- \frac{t |\xi|^2}{2}} \frac{\ud \xi}{|\xi|^{d - \gamma}} = C_{d, \gamma} t^{- \frac{\gamma}{2}},
  \]
  where the last identity is obtained thanks to the elementary change of
  variable $t^{\frac{1}{2}} \xi \mapsto \eta$.
\end{example}

\begin{lemma}\label{L:k-upper-bound}
  Let $k(\cdot)$ be defined by~\eqref{E:k_t}. Suppose that for some $\eta \in
  (0, 1)$, the strengthened Dalang's condition~\eqref{E:stren} holds. Set
  $C_\eta^* \coloneqq \sup_{x > 0} \exp \left(- x/2 \right) (1 + x)^{1 - \eta}$
  and recall that the constant $ \Upsilon_{\eta}$ is defined by~\eqref{E:stren}.
  Then for all $T \ge 1$, it holds that
  \begin{align}\label{E:kt-bdd}
    k (t) \le C_\eta^*\, \Upsilon_{\eta}\, \left(t/T\right)^{\eta - 1}, \quad \text{for all $t \in (0, T)$}.
  \end{align}
\end{lemma}
\begin{proof}
  Inequality~\eqref{E:kt-bdd} is a simple consequence of
  Hypothesis~\ref{H:stren} as follows:
  \begin{align*}
    k(t)
     &  = (2 \pi)^{-d} \int_{\R^d} e^{- \frac{t |\xi|^2}{2}}  \widehat{f} (\ud \xi )
    \le C_\eta^* (2 \pi)^{-d} \int_{\R^d}  \frac{\widehat{f} (\ud \xi ) }{(1 + t |\xi|^2)^{1 - \eta}}                               \\
     &  = C_\eta^*\, t^{\eta - 1} (2 \pi)^{-d} \int_{\R^d}  \frac{\widehat{f} (\ud \xi ) }{( t^{-1} +  |\xi|^2)^{1 - \eta}}
    \le C_\eta^*\, t^{\eta - 1} (2 \pi)^{-d} \int_{\R^d}  \frac{\widehat{f} (\ud \xi ) }{( T^{-1} +  |\xi|^2)^{1 - \eta}}           \\
     &  = C_\eta^*\, t^{\eta - 1} T^{1 - \eta} (2 \pi)^{-d} \int_{\R^d}  \frac{\widehat{f} (\ud \xi ) }{( 1 + T |\xi|^2)^{1 - \eta}}
    \le C_\eta^*\, t^{\eta - 1} T^{1 - \eta} \Upsilon_\eta
     =  C_\eta^*\, \Upsilon_{\eta}\, \left(t/T\right)^{\eta - 1},
  \end{align*}
  where we have used the definition of $C_\eta^*$, the inequality $t<T$, and the
  fact that $T \ge 1$ implies $(1 + T|\xi|^2)^{1 - \eta} \ge (1 + |\xi|^2)^{1 -
  \eta}$. This proves Lemma~\ref{L:k-upper-bound}.
\end{proof}

\begin{lemma}\label{L:hn-upper-bound}
  Suppose that for some $\eta \in (0, 1)$, the strengthened Dalang's
  condition~\eqref{E:stren} holds. Let $h_n(\cdot)$ be given by~\eqref{E:hn-recursion} for
  $n \ge 0$. Then, the following inequality holds
  \begin{align}\label{E:hn-bdd}
    h_n (t) \le  \left[ C_{\eta}^* \, \Upsilon_\eta\, \Gamma(\eta) \right]^n \frac{t^{n \eta}}{\Gamma (n \eta + 1)} ,
    \quad \text{for all $t \in (0, 1)$ and $n \ge 0$.}
  \end{align}
  where $C_\eta^*$ is defined in Lemma~\ref{L:k-upper-bound} and $\Upsilon_{\eta}$ is
  defined by~\eqref{E:stren}.
\end{lemma}

\begin{proof}
  We first apply Lemma~\ref{L:k-upper-bound} with $T = 1$ to see that $k (t) \le C_\eta^*
  \Upsilon_\eta t^{\eta - 1}$. Now we prove~\eqref{E:hn-bdd} by induction in
  $n$. For $n = 0$, inequality~\eqref{E:hn-bdd} holds trivially. Next, suppose
  $n \ge 1$. Using the induction hypothesis, we can write
  \begin{align}\label{E:hn-ind-step}
    h_n (t)
    \le & \frac{\left[ C_{\eta}^* \, \Upsilon_\eta\, \Gamma(\eta) \right]^{n - 1}}{\Gamma ((n - 1) \eta + 1 )} \int_0^t (t - s)^{(n - 1) \eta} k(s) \ud s \\ \nonumber
    \le & \frac{\left[ C_{\eta}^* \, \Upsilon_\eta\, \Gamma(\eta) \right]^{n - 1} C_\eta^* \Upsilon_\eta}{\Gamma ((n - 1) \eta + 1 )} \int_0^t (t - s)^{(n - 1) \eta} s^{\eta - 1} \ud s
    = \frac{\left[ C_{\eta}^* \, \Upsilon_\eta\, \Gamma(\eta) \right]^n\,  t^{n \eta}}{\Gamma (n \eta + 1)},
  \end{align}
  where we have used the Beta integral. This proves~\eqref{E:hn-bdd} for general
  $n \ge 0$.
\end{proof}

\begin{remark}\label{R:riesz-hn-formula}
  As computed in Example 1.2 of the Appendix in~\cite{chen.kim:19:nonlinear},
  when the correlation function is the Riesz kernel $f (x) = |x|^{- \gamma}$
  with $\gamma < 2 \wedge d$, the expression $h_n (t)$ in~\eqref{E:hn-bdd}
  admits the following explicit formula:
  \begin{align*}
    h_n (t) = C_{\gamma, d}^n \frac{t^{n(1 - \gamma/2)}}{\Gamma\left(\left(1 - \gamma/2\right) n + 1\right)} \quad \text{with }
    C_{\gamma, d} \coloneqq \frac{\Gamma \left((d - \gamma)/2\right) \Gamma \left(1 - \gamma/2\right)}{2^{1 + \gamma/2} \Gamma \left(1 + d/2\right)} .
  \end{align*}
  Note that the Riesz kernel satisfies Hypothesis~\ref{H:stren} for any $\eta
  \in (0, 1 - \gamma/2)$. This is consistent with the bound established in
  Lemma~\ref{L:hn-upper-bound}, but is slightly sharper, as it allows for $\eta = 1 -
  \gamma/2$.
\end{remark}

For our future computations we will need the following corollary of
Lemma~\ref{L:hn-upper-bound}.

\begin{corollary}\label{C:hn-via-h1}
  Assume the conditions in Lemma~\ref{L:hn-upper-bound} prevail. Then, recalling that
  $h_{n}(t)$ is given by~\eqref{E:hn-recursion}, it holds that
  \begin{align*}
    h_n (t) \le \frac{\left[ C_{\eta}^* \, \Upsilon_\eta\, \Gamma(\eta) \right]^{n - 1} t^{(n - 1) \eta} h_1 (t)}{\Gamma ((n - 1) \eta + 1 )},
    \quad \text{for all $n \ge 1$ and $t \in (0, 1)$.}
  \end{align*}
\end{corollary}
\begin{proof}
  Starting from~\eqref{E:hn-ind-step}, we can bound $(t - s)^{(n - 1) \eta}$ by
  $t^{(n - 1) \eta}$, as $n \ge 1$ and $\eta > 0$. This proves the corollary.
\end{proof}

We close this subsection by recalling the following elementary relation borrowed
from~\cite{chen.dalang:15:moments}*{Lemma A.4}.

\begin{lemma}\label{L:heat-kernel-factorization}
  Let $p_t$ be the heat kernel defined by~\eqref{E:heat}. Then, for any $s, t >
  0$ and $x, y \in \R^d$, we have
  \begin{align}\label{E:heat-kernel-factorization}
    p_t (x) \, p_s (y) = p_{\frac{t s}{t + s}} \left(\frac{s x + t y}{t + s}\right) p_{t + s} (x - y).
  \end{align}
  In particular, for any $s < r < t$ and $x, y, u \in \R^d$,
  \begin{align}\label{E:brownian-bridge-density}
    \frac{p_{t-r}(x-u) \, p_{r-s}(u-y)}{p_{t-s}(x-y)}
    = p_{\frac{(r-s)(t-r)}{t-s}}
    \left(u - \frac{(t-r)y + (r-s)x}{t-s}\right),
  \end{align}
  and observe that the left hand side of~\eqref{E:brownian-bridge-density} represents the
  density at time $r$ of a Brownian bridge from $(s,y)$ to $(t,x)$.
\end{lemma}

\subsection{Some properties of the stochastic heat equation}\label{S:prop-heat-eq}

Due to its linear nature, the solution to equation~\eqref{E:SHE} admits a chaos
expansion representation. The main proposition in this direction, borrowed
from~\cite{balan.chen:18:parabolic}*{Theorems~2.3 and~1.1}, is summarized below.

\begin{proposition}\label{P:she-chaos-expansion}
  Under condition~\eqref{E:Dalang} on the noise $\dot{W}$, let $u$ be the unique
  solution to equation~\eqref{E:SHE} as assessed in Theorem~\ref{T:she-well-posedness}.
  Then, for every $(t, x) \in \R_{+} \times \R^d$, the random variable $u(t,x)$
  is an element of $L^2 (\Omega)$. It admits a chaos decomposition like
  in~\eqref{E:chaos-dcp}, which is spelled out as:
  \begin{equation}\label{E:Chaos-Exp-U}
    u(t,x) = \E[u(t,x)]+\sum_{k=1}^{\infty} \beta^k I_k(f_k(\cdot \, ; t,x)).
  \end{equation}
  Moreover, the kernels $f_k$ in~\eqref{E:Chaos-Exp-U} are elements of $\mathcal
  H^{\otimes k}$ uniquely determined by $u$ and where the series converges in
  $L^2(\Omega)$. Specifically, the functions $f_k$ are given by
  \begin{equation}\label{E:she-chaos-kernel}
    f_k(\bs_k, \by_k ; t,x) = \int_{\R^d} p_{t-s_{k}}(x-y_{k})\cdots
    p_{s_{2}-s_{1}}(y_{2}-y_{1})  \, p_{s_{1}} (y_{1} - z) \mu(\ud z) \, \1_{\cs_{k}\left( (0,t]\right)} (\bs_k).
  \end{equation}
\end{proposition}

Our future computations will also rely on properties of the four-parameter
function $\cz$ introduced in Definition~\ref{D:PolyMeasure}. We label the chaos
expansion for $\cz$ below, and observe that it has to be seen as a simple
extension of Proposition~\ref{P:she-chaos-expansion}. For our moment bounds, we will resort
to the following standard convention.

\begin{notation}\label{N:Lp}
  Throughout the paper, we use $\Norm{\cdot}_p$ to denote the norm in
  $L^p(\Omega)$.
\end{notation}

\noindent
Note that the bound~\eqref{E:moment} on moments is taken
from~\cite{balan.chen:18:parabolic}*{Eq.~(1.8)}, applied to the Delta initial
condition.

\begin{proposition}\label{P:polymer-transition-chaos}
  Let the assumptions of Proposition~\ref{P:she-chaos-expansion} prevail. Then,
  equation~\eqref{E:PAM-Delta} admits a unique mild solution, denoted by $\cz
  (s, y; \, t, x ; \, \beta)$. This random field is such that:
  \begin{enumerate}[leftmargin=*,
      label={\textbf{(\roman*)}}]
    \setlength\itemsep{.1in}

    \item For every quadruple $(s, t, x, y) \in \cs_2 \times \R^{2d}$, the
      random variable $\cz (s, y; \, t, x ; \, \beta)$ is in $L^2 (\Omega)$ and
      it admits the chaos expansion
      \begin{align}\label{E:Chaos-Exp-Z}
        \cz (s,y; \, t,x; \, \beta) = & \E \left[ \cz (s,y; \, t,x; \, \beta)\right]
        + \sum_{k=1}^{\infty} \beta^k I_k (g_k(\cdot \, ; s,y,t,x)) ,
      \end{align}
      where the constant term in~\eqref{E:Chaos-Exp-Z} is such that
      \begin{align} \label{E:mean-cz-heat}
        \E \left[ \cz (s,y; \, t,x; \, \beta) \right] = p_{t - s} (x - y)
      \end{align}
      and the functions $g_k$ are elements of $\mathcal H^{\otimes k}$ given by
      \begin{equation}\label{E:gk}
        g_k(\bs_k, \by_k; s,y,t,x) = p_{t-s_{k}}(x-y_{k})\cdots
        p_{s_{2}-s_{1}}(y_{2}-y_{1})  \, p_{s_{1}-s}(y_{1}-y) \, \1_{\cs_{k}\left( (s,t]\right)} (\bs_k).
      \end{equation}

    \item For all $0 \le s < t \le 1$ and $p \ge 1$, the $(2p)$-moments of $\cz
      (s,y; \, t,x; \, \beta)$ can be bounded as
      \begin{equation}\label{E:moment}
        \Norm{\cz (s,y; \, t,x; \, \beta)}_{2p}
        \le C_{\beta, p} \, p_{t - s} (x - y) .
      \end{equation}

  \end{enumerate}
\end{proposition}

\begin{remark}
  The kernels $g_k$ in~\eqref{E:gk} can be expressed in terms of Brownian bridge
  transitions. That is, as mentioned in~\cite{alberts.khanin.ea:14:continuum},
  the quantity
  \begin{equation*}
    \frac{g_k(\bs_k, \by_k; s,y,t,x)}{p_{t-s}(x-y)}
  \end{equation*}
  represents the finite-dimensional density for a Brownian bridge from $(s,y)$
  to $(t,x)$.
\end{remark}

\begin{remark}
  Using the following convention
  \begin{align}\label{E:g0-convention}
    g_0 (\cdot \, ; s,y, r,z) = g_0 (s,y, r,z) = I_0 (g_0 (\cdot \, ; s,y, r,z)) = p_{r - s} (z - y),
  \end{align}
  the chaos expansion~\eqref{E:Chaos-Exp-Z} can be rewritten more compactly as
  \begin{align}\label{E:Chaos-Exp-Z0}
    \cz (s,y; \, t,x; \, \beta) =  \sum_{k = 0}^{\infty} \beta^k I_k (g_k(\cdot \, ; s,y,t,x)).
  \end{align}
\end{remark}

We now collect some useful properties of the field $\cz$, in
Theorem~\ref{T:transition-properties} below. This
generalizes~\cite{alberts.khanin.ea:14:continuum}*{Theorem 3.1} to dimension $(1
+ d)$ with spatially correlated noise.

\begin{theorem}\label{T:transition-properties}
  Assume Dalang's condition~\eqref{E:Dalang} holds, and fix $\beta > 0$. Let
  \begin{equation*}
    \left\{ \cz (s,y; \, t,x; \, \beta) ; 0 \le s < t, \, x, y \in \R^d \right\}
  \end{equation*}
  be the unique solution to equation~\eqref{E:PAM-Delta} understood in the mild
  sense. Then, there exists a version of $\cz$ which satisfies the following
  properties:

  \begin{enumerate}[label=\textnormal{\textbf{(\roman*)}}]

    \item \textbf{Centrality.} \label{T:transition-properties:center} The expected value of $\cz
      (s,y; \, t,x; \, \beta)$ is given by
      \[
        \E\left[\cz (s,y; \, t,x; \, \beta) \right] = p_{t - s} (x - y).
      \]

    \item \textbf{H\"{o}lder continuity.~} \label{T:transition-properties:holder} Assume
      Hypothesis~\ref{H:stren} holds. Then $\cz$ admits a version which is
      jointly H\"{o}lder continuous on any compact set of $\{(s,t)\in\R_+^2
      \colon s<t\} \times \R^{2d}$, with exponent $\frac{1}{2} (\eta -
      \epsilon)$ in temporal variables and $\eta - \epsilon$ in spatial
      variables for all $\epsilon \in (0, \eta)$.

    \item \label{T:transition-properties:stationarity} \textbf{Stationarity.~} The field $\cz$
      satisfies the following identity in law:
      \[
        \cz (s,y; \, t,x; \, \beta) \overset{(d)}{=} \cz (s + u, y + z; \, t + u, x + z; \, \beta), \quad \text{for all $z \in  \R^d$ and $u > 0$.}
      \]

    \item\label{T:transition-properties:scale} \textbf{Scaling property.~} Assume
      Hypothesis~\ref{H:scale-gmm-simple} holds. It holds that
      \begin{align}\label{E:scale}
        \cz \left(r^2 s, r y; \, r^2 t, rx; \, \beta\right)
        \overset{(d)}{=} r^{- d} \cz \left(s, y ; \, t , x; \, r^{\frac{1 - \gamma}{2}} \beta \right) , \quad \text{for all $r > 0$.}
      \end{align}

    \item\label{T:transition-properties:positive} \textbf{Positivity.~} With probability one, we
      have
      \[
        \cz ( s, y; \, t, x; \, \beta) > 0, \quad \text{for all $0 \le s < t$ and $x, y \in \R^d$.}
      \]
      As a result, with probability one,
      \[
        \cz ( s, y; \, t, *; \, \beta) > 0, \quad \text{for all $0 \le s < t$ and $y \in \R^d$.}
      \]

    \item\label{T:transition-properties:homogeneity} \textbf{Homogeneity.~} The law of
      \[
        \frac{\cz (s,  y; \, t, x; \, \beta) }{p_{t - s} (x - y)}
      \]
      does not depend on $x$ or $y$.

    \item\label{T:transition-properties:independence} \textbf{Independence.~} For any finite disjoint
      intervals $(s_i, t_i] \subseteq \R_+$, $i = 1, \dots, n$, and for any
      $x_i$ and $y_i \in \R^d$, the collection of random variables $\{\cz(s_i,
      y_i; \, t_i, x_i; \, \beta), i = 1, \dots, n\}$ are mutually independent.

    \item \textbf{The Chapman-Kolmogorov equation.~} \label{T:transition-properties:ck} With
      probability one, it holds that
      \begin{align}\label{E:chapman-kolmogorov}
        \cz (s,  y; \, t, x; \, \beta) = \int_{\R^d} \cz (s,  y; \, r, z; \, \beta) \cz (r,  z; \, t, x; \, \beta) \ud z ,
      \end{align}
      for all $0 \le s < r < t$ and $x, y \in \R^d$.

  \end{enumerate}
\end{theorem}

\begin{proof}
  We will address each property above separately. \medskip

  \noindent \textit{Property~\ref{T:transition-properties:center}.~} This property has already
  been stated in~\eqref{E:mean-cz-heat}. It follows directly from the mild
  formulation~\eqref{E:mild} and the fact that the Walsh-Dalang integral has
  zero mean. \medskip

  \noindent \textit{Property~\ref{T:transition-properties:holder}.~} The H\"older continuity of
  $\cz (s,y; \, t,x; \, \beta)$ in $t$ and $x$ was established
  in~\cite{chen.huang:19:comparison}*{Theorem 1.8}, while the H\"older
  continuity in $s$ and $y$ follows from Lemma~\ref{L:holder} below, together
  with the Kolmogorov continuity theorem. \medskip

  \noindent \textit{Property~\ref{T:transition-properties:stationarity}.~} The stationarity is
  due to the stationarity of the homogeneous noise. \medskip

  \noindent \textit{Property~\ref{T:transition-properties:scale}.~} Thanks to the chaos
  expansion~\eqref{E:Chaos-Exp-Z}, to prove~\eqref{E:scale}, it suffices to show
  that for all $k \ge 1$,
  \begin{align}\label{E:scale-chaos}
    \beta^k I_k\left(g_k(\cdot \, ; r^2 s, ry, r^2t, r x)\right)
    \overset{(d)}{=} r^{- d}  \left(r^{1 - \frac{\gamma}{2}}\beta \right)^k I_k\left(g_k(\cdot \, ; s, y, t, x)\right).
  \end{align}
  In order to prove~\eqref{E:scale-chaos}, we resort to the following scaling
  property of the heat kernel, specifically,
  \begin{align}\label{E:scale-heat}
    p_{r^2 t} (r x) = r^{- d} p_t (x), \quad \text{for all $r > 0$, $t > 0$ and $x \in \R^d$}.
  \end{align}
  Next, we perform the change of variables $(\bs_k, \by_k) = (r^2
  \widehat{\bs}_k, r \widehat{\by}_k)$ in every $I_k$ in~\eqref{E:scale-chaos}.
  It is readily checked from the expression~\eqref{E:gk} for $g_k$ and relation~\eqref{E:scale-heat} that
  \begin{equation}\label{E:gk-scaling}
    g_k \left(\bs_k, \by_k ; \, r^2 s, ry, r^2t, r x \right) = r^{- (k + 1)d} g_k  \left(\widehat{\bs}_k, \widehat{\by}_k ; \, s, y, t, x \right).
  \end{equation}
  Note that we get a factor $r^{- (k + 1)d}$ in~\eqref{E:gk-scaling}, since
  $g_k$ is a product of $k + 1$ heat kernels. Plugging this identity into the
  left hand side of~\eqref{E:scale-chaos}, taking the Jacobian $|J| = r^{k (d
  +2)}$ of our transformation into account, and recalling the
  definition~\eqref{E:W-Scale} of $W_r$, we end up with
  \begin{align*}
      & I_k\left(g_k(\cdot \, ; r^2 s, ry, r^2t, r x)\right)                                                                                                                                                                                   \\
    = & r^{- (k + 1)d + k (d + 2)} \int_{(\R_+ \times \R^d)^k} g_k  \left(\widehat{\bs}_k, \widehat{\by}_k ; \, s, y, t, x \right) W_r (\ud \widehat{s}_1, \ud \widehat{y}_1) \times \dots \times W_r (\ud \widehat{s}_k, \ud \widehat{y}_k) .
  \end{align*}
  Therefore, applying the scaling~\eqref{E:ScaleW}, it is readily checked that
  \begin{align*}
    \beta^k I_k\left(g_k(\cdot \, ; r^2 s, ry, r^2t, r x)\right)
    \overset{(d)}{=} r^{-d} \left(r^{1 - \frac{\gamma}{2}}\beta \right)^k I_k\left(g_k(\cdot \, ;  s, y, t,  x)\right).
  \end{align*}
  This completes the proof of~\eqref{E:scale-chaos}, and thereby also
  of~\eqref{E:scale}. \medskip

  \noindent \textit{Property~\ref{T:transition-properties:positive}.~} The positivity for the
  solution has been verified in~\cite{chen.huang:19:comparison}*{Theorem 1.6}.
  \medskip

  \noindent \textit{Property~\ref{T:transition-properties:homogeneity}.~} This property has been
  established for the white noise case
  in~\cite{amir.corwin.ea:11:probability}*{Proposition~1.4}, where its validity
  is attributed to the following two key observations:

  \begin{enumerate}[label=\textnormal{(\alph*)}]

    \item \label{I:homo:a} Brownian bridge transition probabilities are affine
      transformations of each other.

    \item \label{I:homo:b} Performing an appropriate change of variables, the
      white noise remains invariant in distribution.

  \end{enumerate}
  In our context, it is easily seen that fact~\ref{I:homo:a} continues to hold.
  Moreover, fact~\ref{I:homo:b} remains valid when the white noise is generalized
  to any spatially homogeneous noise. One can thus
  generalize~\cite{amir.corwin.ea:11:probability}*{Proposition~1.4} in a
  straightforward way. For completeness, we illustrate the argument at the first
  chaos level ($k = 1$), and the general case follows by iterating the same
  reasoning. To this aim, take $s = 0$ and $t = 1$. According to
  expression~\eqref{E:gk} we have
  \begin{align*}
      g_1 \left(s_1, y_1 \, ; s, y, t, x\right)
    = g_1 \left(s_1, y_1 \, ; 0, y, 1, x\right)
    = p_{1 - s_1} (x - y_1) p_{s_1} (y_1 - y) \1_{[0,1]} (s_1).
  \end{align*}
  Therefore, one has
  \begin{align*}
    I_1 (g_1) = \int_0^1 \int_{\R^d} p_{1 - s_1} (x - y_1) p_{s_1} (y_1 - y) W (\ud s_1, \ud y_1).
  \end{align*}
  We now invoke~\eqref{E:heat-kernel-factorization} to get
  \begin{equation}\label{E:I1-over-p1}
    \frac{I_1(g_1)}{p_{1} (x - y)}
    =  \int_0^1 \int_{\R^d} p_{(1 - s_1)s_1} \left(s_1 x + (1 - s_1) y - y_1 \right) W (\ud s_1, \ud y_1).
  \end{equation}
  In addition, the quantity $v\coloneq s_1 x + (1 - s_1) y$ has to be considered
  as a constant when one integrates with respect to $\ud y_1$
  in~\eqref{E:I1-over-p1}. Using the identity
  \[
    \dot{W} (s, v + \cdot) \overset{(d)}{=} \dot{W} (s, \cdot), \quad \text{for any } v \in \R^d,
  \]
  which stems easily from~\eqref{E:W-Density}, we thus get
  \begin{align}\label{E:chaos-1-bridge}
    \frac{I_1(g_1)}{p_{1} (x - y)}
    \overset{(d)}{=} \int_0^1 \int_{\R^d} p_{(1 - s_1)s_1} (y_1) W (\ud s_1 ,\ud y_1).
  \end{align}
  It is now readily checked that the right-hand side above does not depend on $x
  - y$. We now explain briefly why the same conclusion holds for every chaos
  level $k \ge 1$. Fix $0 < s_1 < \dots < s_k < 1$ and set $\ell_{s_i} \coloneqq
  s_i x + (1-s_i) y$. Since $g_k(\bs_k, \by_k; 0,y,1,x)/p_1(x-y)$ is the joint
  density of the Brownian bridge $(X_{s_1},\dots,X_{s_k})$ from $(0,y)$ to
  $(1,x)$, and since the process $(X_{s_i}-\ell_{s_i})_{i=1}^k$ is a centered
  Brownian bridge from $0$ to $0$, this density depends on $(x,y)$ only through
  the shifts $(y_i-\ell_{s_i})_{i=1}^k$. Writing the multiple integral defining
  $I_k(g_k)$ and performing the translations $y_i \mapsto y_i + \ell_{s_i}$, we
  conclude by spatial homogeneity of the noise (cf.~\eqref{E:W-Density}) that
  the law of $I_k(g_k)/p_1(x-y)$ does not depend on $x$ or $y$. Summing over $k$
  in the chaos expansion yields Property~\ref{T:transition-properties:homogeneity}. \medskip

  \noindent \textit{Property~\ref{T:transition-properties:independence}.~} The independence across
  disjoint time intervals arises from the independent temporal increments of
  white-in-time Gaussian noise. \medskip

  \noindent \textit{Property~\ref{T:transition-properties:ck}.~} Denote by $R$ the right-hand
  side of equation~\eqref{E:chapman-kolmogorov}. Next, write the chaos expansion for both $\cz
  (s,y; \, r,z; \, \beta)$ and $\cz (r,z; \, t,x; \, \beta)$. Then, from the
  expansion in~\eqref{E:Chaos-Exp-Z0}, the product
  \[
    \cz (s,y; \, r,z; \, \beta) \, \cz (r,z; \, t,x; \, \beta),
  \]
  can be written as a double sum of products of chaos elements, resulting in
  \begin{align}\label{E:R-chaos-sum}
    R = & \int_{\R^d} \cz (s,y; \, r,z; \, \beta)\cz (r,z; \, t,x; \, \beta) \ud z \nonumber \\
    =   & \sum_{m , n = 0}^{\infty} \beta^{m + n} \int_{\R^d}   I_m (g_m(\cdot \, ; s,y, r,z)) I_n (g_n (\cdot \, ; r,z,t,x)) \ud z.
  \end{align}
  Let us now simplify the products of iterated integrals
  in~\eqref{E:R-chaos-sum}. For $m = n = 0$, with our
  convention~\eqref{E:g0-convention} in mind, we have
  \begin{equation}\label{E:heat-convolution}
    \int_{\R^d} I_0 (g_0 (\cdot \, ; s,y, r,z)) I_0 (g_0 (\cdot \, ; r,z, t, x))  \ud z =  \int_{\R^d} p_{r - s} (z - y)p_{t - r} (x - z) \ud z
    =  p_{t - s} (x - y).
  \end{equation}
  For the summands of general $m$ and $n$ on the right-hand side
  of~\eqref{E:R-chaos-sum}, one invokes relation~\eqref{E:contraction-0} for the
  product of iterated integrals. This yields
  \begin{align}\label{E:product-contraction}
    \MoveEqLeft I_m (g_m(\cdot \, ; s,y, r,z)) I_n (g_n (\cdot \, ; r,z,t,x)) \nonumber                                                                \\
    = & \sum_{k = 0}^{m \wedge n} k! \binom{m}{k} \binom{n}{k} I_{m + n - 2k} \big(g_m(\cdot \, ; s,y, r,z) \otimes_k g_n (\cdot \, ; r,z,t,x) \big) .
  \end{align}
  In addition, notice that for any positive integers $m$ and $n$, $g_m(\cdot \,
  ; s,y, r,z)$ and $g_n (\cdot \, ; r,z,t,x)$ are supported on
  $\mathcal{S}_m(\left(s,r\right]) \times \R^m$ and
  $\mathcal{S}_n(\left(r,t\right]) \times \R^n$, respectively. This implies that
  for any $1 \le k \le m \wedge n$, the following coalescent terms vanish, i.e.,
  \[
    g_m (\cdot \, ; s,y, r,z) \otimes_k g_n (\cdot \, ; r,z,t,x)~\equiv 0,
    \quad \text{for all $1 \le k \le m \wedge n$.}
  \]
  As a consequence, the only non-vanishing term in~\eqref{E:product-contraction}
  is the one corresponding to $k = 0$, namely,
  \begin{align}\label{E:product-k0}
    I_m (g_m(\cdot \, ; s,y, r,z)) I_n (g_n (\cdot \, ; r,z,t,x))  = I_{m + n} \big(g_m(\cdot \, ; s,y, r,z) \otimes g_n (\cdot \, ; r,z,t,x) \big).
  \end{align}
  Therefore, plugging~\eqref{E:heat-convolution} and~\eqref{E:product-k0}
  into~\eqref{E:R-chaos-sum}, and setting $k = m + n$, we end up with
  \begin{align*}
    R = & \int_{\R^d} p_{r - s} (z - y) p_{t - r} (x - z) \ud z                                                                                                                                     \\
        & + \sum_{k = 1}^{\infty}  \beta^k \int_{\R^d} I_k \left(\sum_{m = 0}^k g_m (\bs_{1:m}, \by_{1:m} \, ; s,  y, r, z)  g_{k - m} (\bs_{m + 1:k}, \by_{m + 1:k} \, ; r,z,t,x)   \right) \ud z.
  \end{align*}
  An application of the stochastic Fubini theorem for multiple stochastic
  integrals thus yields
  \begin{align}\label{E:ck-1}
    R = p_{t - s} (x - y) + \sum_{k = 1}^{\infty} \beta^k I_k \left(\sum_{m = 0}^k g_{k, m}^r (\cdot\, ; s,y, t, x)   \right),
  \end{align}
  where, recalling our convention~\eqref{E:Conv-t} for multidimensional
  variables, the $k$-th order contraction function $g_{k,m}^r$ is defined by
  \[
    g^r_{k, m} (\bs_k , \by_k \, ; s,y, t, x)
    \coloneqq \int_{\R^d} g_m (\bs_{1:m}, \by_{1:m} \, ; s,  y, r, z)  g_{k - m} (\bs_{m + 1: k}, \by_{m + 1: k} \, ; r,z,t,x) \ud z.
  \]

  One can further simplify the right-hand side of~\eqref{E:ck-1}. To this aim,
  observe again that for all $0 \le m \le k$, we have $\bs_{1: m} = \bs_{m} \in
    \cs_m ((s,r])$, and $\bs_{m + 1:k} \in \cs_{k - m} ((r, t])$. Hence, we claim
  that the semigroup property for the heat kernel $p_t$ implies that
  \begin{equation}\label{E:gkm-factorization}
    g_{k, m}^r (\bs_k, \by_k \, ; s,y, t, x) =  g_k (\bs_k, \by_k  \, ; s,  y, t, x)
    \, \1_{\cs_m ((s, r])} (\bs_{m}) \, \1_{\cs_{k - m} ((r, t])} (\bs_{m+1: k}).
  \end{equation}
  In order to get a better grasp of relation~\eqref{E:gkm-factorization}, let us detail the
  example of $g_{3,2}^r$. In this case, going back to expression~\eqref{E:gk},
  one can write
  \begin{align*}
    g_{3,2}^r(\bs_3, \by_3  \, ; s,  y, t, x)
    = \1_{\cs_2 ((s, r])} (\bs_{1:2}) \1_{(r, t]} (s_3)
    \int_{\R^d}  g_2 (\bs_2, \by_2 \, ; s,  y, r, z)  g_{1} (s_3, y_3 \, ; r,z,t,x) \ud z.
  \end{align*}
  Denote the above integral by $I$. It can be explicitly computed as follows
  using the heat semigroup property:
  \begin{align*}
    I = & \int_{\R^d} p_{r - s_2} (z - y_2) p_{s_2 - s_1} (y_2 - y_1) p_{s_1 - s} (y_1 - y)  p_{t - s_3} (x - y_3) p_{s_3 - r}  (y_3 - z) \ud z \\
    =   & p_{s_3 - s_2} (y_3 - y_2) p_{s_2 - s_1} (y_2 - y_1) p_{s_1 - s} (y_1 - y)  p_{t - s_3} (x - y_3)                                      \\
    =   & g_3 (\bs_3, \by_3  \, ; s,  y, t, x).
  \end{align*}
  Hence, we have verified~\eqref{E:gkm-factorization} for the case $(k, m) = (3, 2)$. The
  general case for arbitrary $(k, m)$ follows by applying the same type of
  argument.

  With relation~\eqref{E:gkm-factorization} in mind, we now take care of the sum over $m$
  in~\eqref{E:ck-1}. Specifically, we write
  \begin{align}\label{E:decom-gk}
    \sum_{m = 0}^{k} g^r_{k, m} (\bs_k , \by_k \, ; s,y, t, x)
    = & g_{k} (\bs_k , \by_k \, ; s,  y, t, x)  \sum_{m = 0}^{k} \1_{\cs_{m} ((s, r])} (\bs_m) \1_{\cs_{k - m} ((r, t])} (\bs_{m + 1:k}) \nonumber \\
    = & g_{k} (\bs_k , \by_k \, ; s,  y, t, x),
  \end{align}
  where the last identity stems from elementary considerations on simplexes:
  \begin{align*}
    \1_{\cs_k ((s, t])} (\bs_k)
    = \sum_{m = 0}^{k} \1_{\cs_{m} ((s, r])} (\bs_m) \1_{\cs_{k - m} ((r, t])} (\bs_{m + 1:k}) .
  \end{align*}
  Reporting~\eqref{E:decom-gk} into~\eqref{E:ck-1}, we thus obtain
  \begin{align*}
    R = p_{t - s} (x - y) + \sum_{k = 1}^{\infty} \beta^k I_k \left(g_{k} (\cdot \, ; s,y, t, x)\right)
    = \cz (s,  y; \, t, x; \, \beta),
  \end{align*}
  where the last equality comes from the expansion~\eqref{E:Chaos-Exp-Z}. This
  completes the proof of Theorem~\ref{T:transition-properties}.
\end{proof}

We now state a technical moment estimate which has been used in the proof of
Theorem~\ref{T:transition-properties}. We will use the notation
\[
  \square_{\theta} (t, x)
  \coloneqq t^{\theta/2} + |x|^{\theta} \quad \text{for all $(\theta, t, x) \in \R_+ \times \R_+ \times \R^d$}.
\]

\begin{lemma}\label{L:holder}
  Suppose that for some $\eta \in (0, 1)$, the strengthened Dalang's
  condition~\eqref{E:stren} holds. Let $\cz (s,y; \, t,x; \, \beta)$ be the
  unique solution to equation~\eqref{E:PAM-Delta} understood in the mild sense.
  Then, for all $p \ge 1$ and $\alpha \in (0, \eta)$, there exists a constant $C
  = C(\alpha, \beta, p, \eta) > 0$ such that for all $0 \le s \le s' < t \le 1$
  and $x, y, y' \in \R^d$, the following inequality holds:
  \begin{multline}\label{E:holder}
    \Norm{\cz (s',y'; \, t,x ; \, \beta) - \cz(s, y; \, t,x ; \, \beta)}_{2p} \\
    \le C \times \frac{p_{4(t - s')} \left(x - y\right) + p_{4(t - s')} \left(x - y'\right)}{(t-s')^{\alpha/2}} \times
    \square_\alpha \left(s'-s, y'-y\right).
  \end{multline}
\end{lemma}
\begin{proof}
  Fix an arbitrary $p \ge 1$. Denote
  \begin{align}\label{E:Def-J}
    g (t, x)
    = g_{s, y; s', y'} (t, x)
    \coloneqq \Norm{\cz(s',y'; \, t,x ; \, \beta) - \cz(s, y; \, t,x ; \, \beta)}_{2p}.
  \end{align}
  Write the mild formulation in~\eqref{E:Mild-Z} as $\cz(s,y; \, t,x) =
  p_{t-s}(x-y) + \beta\; \mathcal{I}\left(s,y; t,x\right)$ with
  \begin{align*}
    \mathcal{I}\left(s,y; t,x\right) \coloneqq \int_s^t \int_{\R^d} p_{t-r}(x-z) \, \cz(s,y ; \, r,z) \, W(\ud r, \ud z).
  \end{align*}
  Then,
  \begin{align*}
    \mathcal{I}\left(s,y;t,x\right) - \mathcal{I}\left(s',y';t,x\right)
    = & \int_s^{s'} \int_{\R^d} p_{t-r}(x-z) \, \cz(s,y ; \, r,z) \, W(\ud r, \ud z)                                      \\
      & + \int_{s'}^t \int_{\R^d} p_{t-r}(x-z) \, \left[\cz(s,y ; \, r,z) - \cz(s',y'; \, r,z)\right] \, W(\ud r, \ud z).
  \end{align*}
  Therefore, we see that
  \begin{align*}
    g (t, x)^2 \le 3 \left(I_0 + \beta^2 I_1 + \beta^2 I_2\right),
  \end{align*}
  where the terms $I_0$, $I_1$, and $I_2$ are respectively defined by
  \begin{align*}
    I_0 & \coloneqq \left|p_{t - s'} (x - y') - p_{t - s} (x - y)\right|^2,                                                                               \\
    I_1 & \coloneqq \Norm{\int_{s}^{s'} \int_{\R^d} p_{t - r} (x - z) \cz(s, y; \, r,z )  W(\ud r, \ud z) }_{2p}^2 , \quad \text{and}                     \\
    I_2 & \coloneqq \Norm{\int_{s'}^{t} \int_{\R^d} p_{t - r} (x - z) \left[ \cz(s', y'; \, r,z )  - \cz(s, y; \, r,z ) \right]W(\ud r, \ud z) }_{2p}^2 .
  \end{align*}
  Those terms can be estimated separately in the following way: first a direct
  consequence of~\cite{chen.huang:19:comparison}*{Lemma~3.1} suggests that for
  all $\alpha \in (0, \eta)$, we have
  \begin{align}\label{E:Holder-0}
    I_0 \le \frac{C_\alpha}{(t-s')^{\alpha}}\left[p_{4(t-s)}(x-y') + p_{4(t-s)}(x-y) \right]^2\, \square_{\alpha}^2 (s' - s, y' - y).
  \end{align}

  Next for the term $I_1$, one can use classical arguments borrowed
  from~\cite{dalang:99:extending}. That is applying the Burkholder-Davis-Gundy
  inequality as in~\cite{dalang:99:extending}*{Page 14}, we get
  \begin{align}\label{E:holder-init-1a}
    \begin{aligned}
      I_1 \le C_p \,\bigg\| \int_s^{s'} \ud r \iint_{\R^{2d}} \ud z \ud z'\: f (z - z') \:
             & p_{t - r} (x - z) \cz(s, y; \, r,z )              \\
      \times & p_{t - r} (x - z') \cz(s, y; \, r,z' ) \bigg\|_p.
    \end{aligned}
  \end{align}
  In addition, one form of Minkowski's inequality can be read as follows for a
  nonnegative measure $\nu$ on a state space $S$ and a random function $v \colon
    S \to \R$:
  \[
    \Norm{ \int_{S} v(\xi) \nu (\ud \xi) }_p
    \le \int_S \Norm{ v (\xi) }_p \, \nu (\ud \xi).
  \]
  Applying this bound to~\eqref{E:holder-init-1a}, and then Cauchy-Schwarz's
  inequality for the quantity $\Norm{\cz(s, y; \, r,z) \cz(s, y; \, r,z' )}_p$,
  we end up with
  \begin{align}\label{E:Holder-pp}
    \begin{aligned}
      I_1 \le C_p \int_{s}^{s'} \ud r \iint_{\R^{2d}} \ud z \ud z' \: f( z - z') \:
             & p_{t - r} (x - z) \Norm{\cz(s, y; \, r, z )}_{2p}    \\
      \times & p_{t - r} (x - z') \Norm{\cz(s, y; \, r, z' )}_{2p}.
    \end{aligned}
  \end{align}
  Therefore, by applying the upper bound in~\eqref{E:moment}, we obtain
  \begin{align*}
    I_1
    \le & C_{p,\beta} \int_s^{s'}  \ud r \iint_{\R^{2d}} \ud z \, \ud z'  f( z - z' )   p_{t - r} (x - z) p_{r - s} (z - y)    p_{t - r} (x - z')  p_{r - s}(z' - y)
  \end{align*}
  Hence, identity~\eqref{E:heat-kernel-factorization} yields
  \begin{align*}
    I_1
    \le & C_{p,\beta} \, p_{t - s}^2 (x - y) \int_s^{s'}  \ud r \iint_{\R^{2d}} \ud z \, \ud z'  f( z - z' ) \\
        & \times p_{\frac{(t - r)(r-s)}{t-s}} \left(\frac{(t-r)y + (r-s)x}{t-s} - z \right)
    p_{\frac{(t - r)(r-s)}{t-s}} \left(\frac{(t-r)y + (r-s)x}{t-s} - z'\right) .
  \end{align*}
  Thus, resorting to an obvious spatial change of variable and invoking Parseval's identity, we end up with
  \begin{align*}
    I_1
    \le & C_{p,\beta} \, p_{t - s}^2 (x - y) \int_s^{s'}  \ud r \iint_{\R^{2d}} \ud z \, \ud z'  f( z - z' )
    p_{\frac{(t - r)(r-s)}{t-s}} \left( z \right)
    p_{\frac{(t - r)(r-s)}{t-s}} \left( z'\right)                                                                                                                \\
    =   & C_{p,\beta} (2\pi)^{-d} p_{t - s}^2 (x - y) \int_s^{s'}  \ud r \int_{\R^d} \widehat{f}(\ud\xi) \exp \left\{ - \frac{(t-r)(r-s)}{t-s} |\xi|^2 \right\}.
  \end{align*}
  Now by
  Lemma~\ref{L:Bridge-t}, we see that
  \begin{align*}
    \int_s^{s'}  \ud r \exp \left\{ - \frac{(t-r)(r-s)}{t-s} |\xi|^2 \right\}
    \le & C_\eta e^{\frac{t-s}{4}} (s'-s)^{\eta} \frac{1}{(1+|\xi|^2)^{1-\eta}}.
  \end{align*}
  Hence, considering $t \le 1$ and using the strengthened Dalang's
  condition~\eqref{E:stren}, we see that
  \begin{align}\label{E:Holder-1}
    I_1 \le & C_{p,\beta,\eta}\, p_{t - s}^2 (x - y)  (s'-s)^{\eta}.
  \end{align}

  The term $I_2$ is treated in a similar way. Again, the Burkholder-Davis-Gundy
  and Cauchy-Schwarz, and Minkowski inequalities entail that (along the same
  lines as for~\eqref{E:Holder-pp}),
  \begin{align*}
    I_2 \le C_{\beta, p} \int_{s'}^t \ud r \iint_{\R^{2d}} \ud z \ud z' f(z-z')
           & p_{t - r} (x - z ) \Norm{\cz(s', y'; \, r,z ) - \cz(s, y; \, r,z ) }_{2p}  \\
    \times & p_{t - r} (x - z') \Norm{\cz(s', y'; \, r,z') - \cz(s, y; \, r,z') }_{2p}.
  \end{align*}
  Recalling our notation~\eqref{E:Def-J}, we have thus obtained
  \begin{align}\label{E:Holder-I2}
    I_2
     & \le C_{\beta, p} \int_{s'}^t \ud r \iint_{\R^{2d}} \ud z  \ud z' f(z-z')
    p_{t - r} (x - z ) g (r, z)
    p_{t - r} (x - z') g (r, z').
  \end{align}

  Combining~\eqref{E:Holder-0},~\eqref{E:Holder-1} and~\eqref{E:Holder-I2},
  we see that
  \begin{align*}
    g (t, x)^2 \le
     & \frac{C_\alpha}{(t-s')^{\alpha}}\left[p_{4(t-s)}(x-y') + p_{4(t-s)}(x-y) \right]^2\, \square_{\alpha}^2 (s' - s, y' - y) \\
     & + C_{p,\beta, \eta}\: p_{t - s}^2 (x - y)  (s'-s)^{\eta}                                                                 \\
     & + C_{\beta, p} \int_{s'}^t \ud r \iint_{\R^{2d}} \ud z  \ud z' f(z-z')
    p_{t - r} (x - z ) g (r, z)
    p_{t - r} (x - z') g (r, z').
  \end{align*}
  Denote by
  \[
    \mu (\cdot) \coloneqq C_{p,\beta, \alpha,\eta} \square_\alpha^2 (s' - s, y' - y) \left(p_{4(s' - s)} (y - \cdot) + p_{4(s' - s)} (y' - \cdot)\right).
  \]
  Then the function $\mu$ satisfies condition~\eqref{E:J0finite}. In addition,
  we can write
  \begin{align*}
    g^2 (t, x)
    \le & (t-s')^{- \alpha} \big[p_{4(t - s')} * |\mu| \big]^2 (x) \\
    & + C_{\beta, p} \int_{s'}^t \ud r \iint_{\R^{2d}} \ud z \ud z'
    & f(z-z') p_{4(t - r)} (x - z )p_{4(t - r)} (x - z') \\
    & \times g (r, z) g (r, z'),
  \end{align*}
  where we have used the trivial bound
  $(s'-s)^{\eta} \le \square_\alpha^2 (s' - s, y' - y)$ for $\alpha \in (0,
    \eta)$ and $0 \le s' -s \le 1$.
  Then, an application of Theorem~\ref{T:IntIneq}
  yields
  \begin{align*}
    g (t, x) \le \frac{C_{p, \beta, \alpha, \eta}}{(t-s')^{\alpha/2}} \left[ p_{4(t - s)} (x - y) + p_{4(t - s)} (x - y') \right] \square_\alpha (s' - s, y' - y).
  \end{align*}
  This completes the proof of Lemma~\ref{L:holder}.
\end{proof}

\section{Local properties of the polymer}\label{S:first-properties-of-polymer}
In this section, we derive a few theorems concerning pathwise properties of our
polymer measure $\bP_\beta^W$ introduced in Definition~\ref{D:PolyMeasure}, as
well as its singularity with respect to the Wiener measure. Those results are
proper generalizations of~\cite{alberts.khanin.ea:14:continuum}*{Theorems 3.1
and 3.3, and results in Section 4.4}, for which we might omit some details.

\subsection{Pathwise properties}\label{S:pathwise}

The results in this section tend to show that on finite time horizons, the
polymer $X$ behaves similarly to a Brownian motion. The first theorem in this
direction states that the polymer measure has the same regularity as Brownian
motion on $[0,1]$.


\begin{theorem}\label{T:polymer-holder-regularity}
  Recall that $\PP$ is the probability encoding the noise (from
  Definition~\ref{D:Noise}), while the quenched polymer probability
  $\bP_\beta^W$ is introduced in Definition~\ref{D:PolyMeasure}. Then, for all
  $\beta > 0$ and $\ep \in (0,\frac12)$, $\PP$-almost surely, we have
  \[
    \bP_\beta^W \left(\Norm{X}_{C^{{1}/{2} - \ep}} < \infty\right) = 1,
  \]
  where $\Norm{\cdot}_{C^\alpha}$, with $\alpha \in (0,1)$, denotes the
  $\alpha$-H\"{o}lder norm on $[0, 1]$.
\end{theorem}

The proof of Theorem~\ref{T:polymer-holder-regularity} relies on the following fractional
Sobolev--Slobodeckij embedding theorem, tailored to our setting and cited from
\cite{di-nezza.palatucci.ea:12:hitchhikers}*{Theorem 8.2}. Specifically, we take
$\Omega = [0, 1]$, $p = 2 \gamma$, $n = 1$, and $s = \frac{\gamma - 1}{2\gamma}$
for a parameter $\gamma > 2$ therein.

\begin{lemma}\label{L:sobolev-holder}
  Let $m \in \NN$ and let $f \colon [0,1]\to \R^m$ be a measurable function. For
  $\gamma > 2$, we set
  \begin{align}\label{E:Sgamma-f}
    \mathscr{S}^{\gamma} (f) \coloneqq \int_0^1  |f (t)|^{2 \gamma}\, \ud t +
    \int_0^1 \int_0^1  \frac{|f (t) - f (s)|^{2 \gamma}}{|t - s|^{\gamma}} \, \ud s \ud t .
  \end{align}
  Then, provided $\mathscr{S}^{\gamma} (f) < \infty$,
  there exists a function
  $\widetilde f \in C^{\frac{1}{2} - \frac{1}{\gamma}}([0,1])$ such that
  $\widetilde f = f$ Lebesgue-almost everywhere, and there exists a strictly
  positive constant $C_{\gamma}$ such that
  \[
    \big\|\widetilde f \big\|_{C^{1/2 - 1/\gamma}}
    \le
    C_{\gamma} \left[\mathscr{S}^{\gamma} (f)\right]^{\frac{1}{2 \gamma}} .
  \]
  In particular, if $f$ is continuous, then $\widetilde f = f$ on $[0,1]$.
\end{lemma}

We now turn to the proof of Theorem~\ref{T:polymer-holder-regularity}.

\begin{proof}[Proof of Theorem~\ref{T:polymer-holder-regularity}]
  It suffices to consider the case $\beta > 0$, because the case $\beta = 0$
  reduces to the classical Brownian motion, whose H\"{o}lder continuity is
  well-understood. Fix $\ep\in(0,\frac12)$ and choose $\gamma >
  \max(2,\ep^{-1})$. Now for $\beta>0$, consider the canonical process $X$ under
  the measure $\bP_\beta^W$ (as introduced in~\eqref{E:fdd}). According to
  Lemma~\ref{L:sobolev-holder}, it suffices to show that, for this choice of $\gamma$, we
  have
  \begin{align}\label{E:akq-4.3-1}
    \PP \left(\bE_\beta^W \left[ \mathscr{S}^{\gamma} (X) \right] < \infty \right) = 1 \, .
  \end{align}
  Indeed, on the event in~\eqref{E:akq-4.3-1}, we have $\mathscr{S}^{\gamma}(X)
  < \infty$ for $\bP_\beta^W$-almost every path. Therefore, Lemma~\ref{L:sobolev-holder}
  yields a H\"older-continuous representative $\widetilde X$ of $X$ (agreeing
  with $X$ Lebesgue-almost everywhere). Since $\gamma > \ep^{-1}$, we have
  $\frac12-\frac{1}{\gamma}>\frac12-\ep$, and thus
  $C^{\frac12-\frac{1}{\gamma}}([0,1];\R^d)\subset C^{\frac12-\ep}([0,1];\R^d)$.
  Consequently, $\widetilde X\in C^{\frac12-\ep}([0,1];\R^d)$. Finally, since by
  Definition~\ref{D:PolyMeasure} the canonical path $X$ is continuous,
  Lemma~\ref{L:sobolev-holder} implies that $\widetilde X = X$ on $[0,1]$, and the claimed
  H\"older regularity follows.

  In order to prove~\eqref{E:akq-4.3-1}, recall from
  Theorem~\ref{T:transition-properties}-\ref{T:transition-properties:positive} that
  $\cz (0,0; \, 1,* ; \, \beta) > 0$ almost surely. Therefore,
  \eqref{E:akq-4.3-1} holds true, provided that the following inequality is
  satisfied:
  \begin{align}\label{E:akq-4.3-2}
    I \coloneqq \E \left[\cz (0,0; \, 1,* ; \, \beta) \times  \bE_\beta^W \left[ \mathscr{S}^{\gamma} (X)  \right] \right]< \infty.
  \end{align}

  Writing the definition~\eqref{E:Sgamma-f} of $\mathscr{S}^{\gamma}$, applying
  Fubini's theorem, and reducing our time integral to an integral on the simplex
  $0 < s < t < 1$, we get
  \begin{align}\label{E:akq-4.3-3-a}
    I = \int_0^1 \ud t \, J_{t}^{(1)} +  2 \int_0^1 \ud t \int_0^t \ud s \, |t - s|^{- \gamma} J_{s,t}^{(2)},
  \end{align}
  where we have set
  \[
     J_{t}^{(1)}   \coloneqq \mathbb{E} \left[\cz (0,0; \, 1,* ; \, \beta) \bE_\beta^W \left[ \left| X_t \right|^{2\gamma} \right] \right] \quad \text{and} \quad
     J_{s,t}^{(2)} \coloneqq \mathbb{E} \left[\cz (0,0; \, 1,* ; \, \beta)   \bE_\beta^W \left[ \left| X_t - X_s \right|^{2\gamma} \right] \right].
  \]

  Applying~\eqref{E:fdd-test} and Fubini's theorem again, we obtain
  \begin{align*}
    J_{s,t}^{(2)} = \iint_{\R^{2d}} \ud x \ud y \, |x - y|^{2 \gamma} \,
    \E \left[  \cz (0,0; \, s,x ; \, \beta)  \cz (s,x; \, t,y ; \, \beta) \cz (t,y; \, 1,* ; \, \beta) \right].
  \end{align*}
  In addition, invoking Theorem~\ref{T:transition-properties}-\ref{T:transition-properties:center}
  and~\ref{T:transition-properties:independence}\footnote{Theorem~\ref{T:transition-properties}~\ref{T:transition-properties:independence}
  also applies to the case when $x = *$; see~\eqref{E:Convention-t} for the
  notation.}, noting that
  $\E\left(\cz (t,y; \, 1,* ; \, \beta)\right) = 1$ (constant-one initial data),
  and using classical identities for the Wiener measure $\bP_0$, we obtain
  \begin{align*}
    J_{s, t}^{(2)}
    = \iint_{\R^{2d}} \ud x \ud y \, |x - y|^{2 \gamma} p_s (x) p_{t - s} (x - y)
    = \bE_0 \left[|X_t - X_s|^{2\gamma}\right].
  \end{align*}
  By the same reasoning, we also have
  \[
    J_{t}^{(1)}
      = \bE_0 \left[|X_t|^{2\gamma}\right].
  \]

  Plugging the above two identities into~\eqref{E:akq-4.3-3-a}, we end up with
  \begin{align} \label{E:akq-4.3-3-c}
    I = & \int_0^1 \ud t \,\bE_0 \left[|X_t|^{2\gamma}\right] + 2 \int_0^1  \ud t \int_0^t \ud s \; |t - s|^{-\gamma}  \bE_0 \left[|X_t - X_s|^{2\gamma}\right] \nonumber \\
    =   & \bE_0 \left[ \int_0^1 \ud t \, |X_t|^{2\gamma}  + 2 \int_0^1 \ud t \int_0^t \ud s \; \frac{|X_t - X_s|^{2\gamma}}{|t - s|^{\gamma}}\right]
    = \bE_0 \left[ \mathscr{S}^{\gamma}  (X)\right]
    < \infty.
  \end{align}
  This establishes~\eqref{E:akq-4.3-2} and hence~\eqref{E:akq-4.3-1}, thereby
  completing the proof of Theorem~\ref{T:polymer-holder-regularity}.
\end{proof}

\begin{remark}[A Kolmogorov--Chentsov approach]\label{R:KC-polymer}
  In the one-dimensional space-time white noise setting, Alberts et
  al.~\cite{alberts.janjigian.ea:22:greens} established the H\"older support of
  the point-to-point polymer measures, and constructed those measures on a
  continuous path space, by proving suitable moment bounds on the increments of
  polymer bridges and then applying the Kolmogorov--Chentsov continuity theorem.
  In the present setting of space-correlated noise in $\R^d$, it would be
  interesting to develop analogous quenched increment estimates for the
  canonical process under $\bP_\beta^W$, which would lead to an alternative
  proof of Theorem~\ref{T:polymer-holder-regularity} based on Kolmogorov continuity.
\end{remark}

Below we state another pathwise property of the canonical process $X$, which
tends to show that it is (locally) indistinguishable from a Brownian motion.
Specifically, we show that the quadratic variation of $X$ is the same as the
quadratic variation of a Brownian motion.

\begin{theorem}\label{T:polymer-quadratic-variation}
  Recall that $\PP_{\beta}$ defines the annealed measure given in
  Definition~\ref{D:PolyMeasure}. With $\PP_{\beta}$-probability one, it holds
  for all $0 \le t \le 1$,
  \begin{align}\label{E:quadratic-variation}
    \langle X \rangle_t = &
    \begin{pmatrix}
      \langle X^{(1)},X^{(1)} \rangle_t & \cdots & \langle X^{(1)},X^{(d)} \rangle_t \\
      \vdots                            & \ddots & \vdots                            \\
      \langle X^{(d)},X^{(1)} \rangle_t & \cdots & \langle X^{(d)},X^{(d)} \rangle_t
    \end{pmatrix}
    = t  I_d ,
  \end{align}
  where for each $i, j \in \{1, \dots, d\}$,
  \[
    \langle X^{(i)},X^{(j)} \rangle_t \coloneqq
    \lim_{n \to \infty} \sum_{k = 1}^{2^n} \left(X_{\frac{k}{2^n}}^{(i)} - X_{\frac{k - 1}{2^n}}^{(i)}\right) \left(X_{\frac{k}{2^n}}^{(j)} - X_{\frac{k - 1}{2^n}}^{(j)}\right) .
  \]
\end{theorem}

\begin{proof}
  We will restrict our analysis of~\eqref{E:quadratic-variation} to the case $t = 1$. The proof
  readily extends to any general $t \in [0,1]$.  Let us also introduce some
  extra notation. For $n \ge 1$, $i, j \in \{1, \dots, d\}$, and $k = 1, \dots,
  2^n$, we set
  \begin{align}\label{E:def-IY-n}
    I_k^{(i,j)} \coloneqq \left(X_{\frac{k}{2^n}}^{(i)} - X_{\frac{k - 1}{2^n}}^{(i)}\right) \left(X_{\frac{k}{2^n}}^{(j)} - X_{\frac{k - 1}{2^n}}^{(j)}\right) - \frac{1}{2^n} \1_{(i = j)} , \quad \text{and} \quad
    Y_n^{(i,j)} \coloneqq \sum_{k = 1}^{2^n} I_k^{(i,j)}.
  \end{align}
  Then observe that we also have
  \begin{align*}
    Y_n^{(i,j)} = \sum_{k = 1}^{2^n} \left(X_{\frac{k}{2^n}}^{(i)} - X_{\frac{k - 1}{2^n}}^{(i)}\right) \left(X_{\frac{k}{2^n}}^{(j)} - X_{\frac{k - 1}{2^n}}^{(j)}\right) - \1_{(i = j)} .
  \end{align*}
  Hence, our claim~\eqref{E:quadratic-variation} for $t = 1$ can be stated as
  \begin{align}\label{E:akq-4.4-1}
    \lim_{n \to \infty}  Y_n^{(i,j)}  = 0, \quad \PP_{\beta}\text{-almost surely}.
  \end{align}

  In order to show~\eqref{E:akq-4.4-1}, we proceed similarly to what we did in
  the proof of Theorem~\ref{T:polymer-holder-regularity}. That is along the same lines as
  for~\eqref{E:akq-4.3-2}, we resort to the strict positivity of $\cz (0,0; \,
  1, * ; \, \beta)$ to assert that~\eqref{E:akq-4.4-1} is equivalent to
  \begin{align}\label{E:akq-4.4-2}
    \lim_{n \to \infty} \cz (0,0; \, 1, * ; \, \beta) \left|Y_n^{(i,j)}\right|^2  = 0,   \quad \PP_{\beta}\text{-almost surely}.
  \end{align}
  Relation~\eqref{E:akq-4.4-2} can then be shown by a classical argument:
  applying a Borel--Cantelli type lemma
  (see, e.g.,~\cite{grimmett.stirzaker:20:probability}*{Lemma~10 in
  Section~7.2}), it is easily seen that~\eqref{E:akq-4.4-2} is true if one can
  prove that for all $\epsilon > 0$, we have
  \begin{align}\label{E:pre-bc}
    \sum_{n = 1}^{\infty} \PP_{\beta} \left(A_n^{(i,j)} (\epsilon) \right) < \infty, \quad \text{where }
    A_n^{(i,j)} (\epsilon) \coloneqq  \left\{\cz (0,0; \, 1, * ; \, \beta) \left|Y_n^{(i,j)}\right|^2   > \epsilon \right\}.
  \end{align}
  Let us now bound each term $\PP_{\beta} (A_n^{(i,j)} (\epsilon))$
  in~\eqref{E:pre-bc} thanks to Markov's inequality. To this aim, we compute
  \begin{align}\label{E:akq-4.4-d}
    \EE_{\beta} \left[\cz (0,0; \, 1, * ; \, \beta) \left|Y_n^{(i,j)}\right|^2 \right]
    = & \sum_{k, k' = 1}^{2^n} \EE_{\beta} \left[\cz (0,0; \, 1, * ; \, \beta) I_k^{(i,j)} I_{k'}^{(i,j)} \right] \nonumber               \\
    = & \sum_{k, k' = 1}^{2^n} \EE \left[\cz (0,0; \, 1, * ; \, \beta) \bE_\beta^W \left[I_k^{(i,j)} I_{k'}^{(i,j)} \right] \right] .
  \end{align}
  We can thus argue like in~\eqref{E:akq-4.3-3-c}, resorting
  to~\eqref{E:fdd-test},~\eqref{E:mean-cz-heat}. We end up with
  \begin{align*}
    \E \left( \cz (0,0; \, 1,* ; \, \beta)   \bE_\beta^W \left[I_k^{(i,j)} I_{k'}^{(i,j)}\right]  \right )
    = & \bE_0 \Bigg[\left(\left( X_{\frac{k}{2^n}}^{(i)} - X_{\frac{k - 1}{2^n}}^{(i)} \right) \left( X_{\frac{k}{2^n}}^{(j)} - X_{\frac{k - 1}{2^n}}^{(j)} \right) - \frac{1}{2^n} \1_{(i = j)}\right) \\
      & \quad\ \times \left(\Big( X_{\frac{k'}{2^n}}^{(i)} - X_{\frac{k' - 1}{2^n}}^{(i)} \Big) \Big( X_{\frac{k'}{2^n}}^{(j)} - X_{\frac{k' - 1}{2^n}}^{(j)} \Big) - \frac{1}{2^n} \1_{(i = j)}\right)\Bigg] \\
    = & \frac{1}{2^{2n - 1}} \1_{(i = j)} \1_{(k = k')}.
  \end{align*}
  Reporting this identity into~\eqref{E:akq-4.4-d}, it is readily checked that
  \begin{align*}
    \E_{\beta} \left[\cz (0,0; \, 1, * ; \, \beta) \left|Y_n^{(i,j)}\right|^2 \right]
    =  \frac{1}{2^{n - 1}} \1_{(i = j)}.
  \end{align*}
  From this, applying Markov's inequality, we get for any $\epsilon > 0$,
  \[
    \PP_{\beta} \left( \cz (0,0; \, 1, * ; \, \beta) \left|Y_n^{(i,j)}\right|^2  > \epsilon \right) \le \frac{\1_{(i = j)}}{2^{n - 1}\,\epsilon} \, .
  \]
  This implies~\eqref{E:pre-bc}, which in turn yields~\eqref{E:akq-4.4-2} and
  then~\eqref{E:akq-4.4-1}, completing the proof of Theorem~\ref{T:polymer-quadratic-variation}.
\end{proof}

\subsection{Singularity of the polymer measure}

The results in Section~\ref{S:pathwise} tend to show that on $[0,1]$, the
behavior of $X$ is similar to the behavior of a standard Brownian motion. With
this information in mind, our next result is somewhat surprising. It states that
on any finite interval, the polymer measure $\bP_\beta^W$ is singular with
respect to the Wiener measure. Our main finding along that lines
(Theorem~\ref{T:singularity-criterion} below) corresponds to the results
of~\cite{alberts.khanin.ea:14:continuum}*{Section 4.4}. However, the adaptation
to our setting requires some substantial extra effort. One reason for this added
difficulty is that unlike in the white noise case, we do not assume any scaling
property of the noise. Instead, we will establish our main theorem under the
following condition.

\begin{condition}\label{Cond:k-t-to-infty}
  Assume that
  \begin{equation}\label{E:k-t-to-infty}
    \lim_{t \downarrow 0} k(t) = \infty \quad\text{or equivalently}\quad
    \widehat{f}\left(\R^d\right) = \infty,
  \end{equation}
  where recall that the covariance function $\Gamma$, the density $f$, and the
  spectral measure $\widehat{f}$ are introduced in Definition~\ref{D:Noise},
  and the function $k(\cdot)$ is defined in~\eqref{E:k_t}.
\end{condition}

Notice that the property of $f$ being nonnegative definite allows us to apply
the monotone convergence theorem to see that
\begin{align*}
  k (t) = \int_{\R^d} e^{-t|\xi|^2/2} \widehat{f}(\ud \xi) \to \widehat{f}\left(\R^d\right) \quad \text{as $t\downarrow 0$,}
\end{align*}
which explains the equivalence relation in~\eqref{E:k-t-to-infty}.

Our assumption~\eqref{E:k-t-to-infty} is expressed in terms of the Fourier
transform $\widehat{f}$. If one wishes to have a condition expressed in direct
modes, the following lemma provides equivalent criteria in the case where the
density $f$ is radial and nonincreasing in $|x|$.

\begin{lemma}\label{L:k-limit-radial-density}
  Suppose that the covariance function $\Gamma$ from Definition~\ref{D:Noise}
  has a density $f \ge 0$ that is radial; that is, there exists a function
  $\phi: [0,\infty) \to [0,\infty]$ such that $f(x) = \phi(|x|)$ for all $x \in
  \R^d$. Moreover, assume that $\phi$ is a nonincreasing function. Then
  \begin{equation}\label{E:kt-to-phi0}
    \lim_{t \downarrow 0} k(t) = \phi(0+) \in [0,\infty],
  \end{equation}
  where $\phi(0+) \coloneqq \lim_{r\downarrow 0}\phi(r)$ is allowed to be
  $+\infty$. In particular,
  \[
    \lim_{t \downarrow 0} k(t) = \infty
    \quad \Longleftrightarrow \quad
    \phi(0+) = \infty.
  \]
\end{lemma}
\begin{proof}
  A change of variable $y = x/\sqrt{t}$ and an application of the monotone
  convergence theorem yield
  \[
    k(t) = \int_{\R^d} f(x)\,p_t(x) \ud x = \E\left[\phi\left(\sqrt{t}\;|Z|\right)\right],
  \]
  where $Z \sim N(0,I_d)$. Notice that $\lim_{|x|\downarrow0}f(x) = \phi(0+)$.
  Since $t\downarrow 0$ implies $\sqrt{t}\,|Z|\downarrow 0$ almost surely and
  $\phi$ is nonincreasing, the random variables $\phi\left(\sqrt{t}\;|Z|\right)$
  increase to $\phi(0+)$ almost surely, and therefore the monotone convergence
  theorem gives~\eqref{E:kt-to-phi0}.
\end{proof}

\begin{example}
  Elaborating on Example~\ref{Ex:riesz-kernel}, consider the Riesz kernel $\Gamma (\ud
    x) = |x|^{- \gamma} \ud x$ with $\gamma \in (0, 2 \wedge d)$. We have seen
  that $k (t) = C_{d,\gamma} t^{- \gamma/2}$ in that case. Hence
  Condition~\ref{Cond:k-t-to-infty} is satisfied regardless of the value of
  $\gamma \in (0, 2 \wedge d)$.
\end{example}

We now turn to the singularity result.

\begin{theorem}\label{T:singularity-criterion}
  For $\beta > 0$, consider the quenched probability $\bP_\beta^W$ given in
  Definition~\ref{D:PolyMeasure}. Also recall that the annealed probability is
  $\PP_{\beta} = \PP \otimes \bP_\beta^W$, where $\PP$ is the probability from
  Definition~\ref{D:Noise}. Furthermore, the probabilities $\PP_0$ and $\bP_0$
  have been introduced in
  Definition~\ref{D:PolyMeasure}-\ref{I:PolyMeasure:i}-\ref{I:PolyMeasure:ii}. Then
  \begin{enumerate}[label=\textnormal{(\roman*)}]
    \item If Condition~\ref{Cond:k-t-to-infty} holds (equivalently,
      $\widehat{f}\left(\R^d\right) = \infty$), we have $\PP_{\beta} \perp
      \PP_0$ and, for $\PP$-almost every realization of $W$, $\bP_\beta^W \perp
      \bP_0$.
    \item If $\widehat{f}\left(\R^d\right) < \infty$ (equivalently,
      $\lim_{t\downarrow 0} k(t) < \infty$), then $\PP_{\beta} \sim \PP_0$ and,
      for $\PP$-almost every $W$, $\bP_\beta^W \sim \bP_0$.
  \end{enumerate}
  In particular, $\PP_{\beta} \perp \PP_0$ holds if and only if
  $\widehat{f}\left(\R^d\right) = \infty$.
\end{theorem}

\begin{proof}
  We work on the product measurable space $(\Omega \times C([0,1]; \R^d),
  \mathcal{F} \otimes \mathcal{G})$, where $(\Omega,\mathcal{F},\PP)$ is the
  probability space carrying the noise $W$ from Definition~\ref{D:Noise}, and
  $(C([0,1]; \R^d),\mathcal{G})$ is the canonical path space from
  Definition~\ref{D:PolyMeasure}. For each $n \ge 1$, we set
  \[
    \cg_n \coloneqq \mathcal{F} \otimes \cg_n^X, \quad \text{where}\quad \cg_n^X \coloneqq \sigma \left( X \left(\frac{k}{2^n}\right) : 0 \le k \le 2^n - 1 \right).
  \]
  Then $\{\cg_n\}_{n \ge 1}$ is an increasing family of sub-$\sigma$-fields of
  $\mathcal{F} \otimes \mathcal{G}$. Moreover, since $X$ has continuous paths
  and the dyadic times are dense in $[0,1]$, for every $t \in [0,1]$ there
  exists a sequence of dyadic times $(t_m)_{m \ge 1} \subset [0,1)$ such that
  $t_m \to t$ and hence $X_t = \lim_{m \to \infty} X_{t_m}$. This yields that
  \begin{align}\label{E:prf_Singular-a}
  \bigvee_{n \ge 1} \cg_n^X = \mathcal{G} , \quad \text{and} \quad \bigvee_{n \ge 1} \cg_n = \mathcal{F} \otimes \mathcal{G} .
    \end{align}  We now argue that the
  restrictions of the two measures $\PP_{\beta}$ and $\PP_0$ to $\cg_n$ are
  mutually absolutely continuous, namely,
  \begin{equation}\label{E:abs-cont}
    \PP_{\beta} \big|_{\cg_n} \sim \PP_0 \big|_{\cg_n}, \quad \text{for all $n \ge 1$.}
  \end{equation}
  In order to get a better understanding of this relation, applying~\eqref{E:polymer-one-time-density}
  and~\eqref{E:fdd-test} for $t \in (0,1)$ and $F \in C_b (\R^d)$ with $k = 1$,
  we can write
  \begin{align*}
    \bE_\beta^W [F(X_t)]
    = & \frac{1}{\cz(0,0; \, 1, *;\, \beta)} \, \int_{\R^{d}} \ud x \; F(x)  \cz \left( 0, 0; \, t, x ; \, \beta\right) \cz \left( t, x; \, 1, * ; \, \beta\right)                         \\
    = & \frac{1}{\cz(0,0; \, 1, *;\, \beta)} \, \int_{\R^{d}} \ud x \; F(x) p_t(x) \frac{\cz \left( 0, 0; \, t, x ; \, \beta\right) \cz \left( t, x; \, 1, * ; \, \beta\right) }{p_t (x)}.
  \end{align*}
  Since $p_t(x)$ is the density of $X_t$ under the measure $\bP_0$, we get
  \begin{align}\label{E:polymer-RN-t}
    \bE_\beta^W [F(X_t)] =
    \bE_0 \left[F(X_t) \, \frac{\cz \left( 0, 0; \, t, X_t ; \, \beta\right) \cz \left( t, X_t; \, 1, * ; \, \beta\right)}{\cz(0,0; \, 1, *;\, \beta) \, p_t(X_t)}\right],
  \end{align}
  and expression~\eqref{E:polymer-RN-t} also yields a formula for the density of
  $\mathcal{L} \left(X_t | \bP_\beta^W\right)$ with respect to $\bP_0$. We now
  carry out the analogous computation on the dyadic skeleton. Fix $n \ge 1$ and
  set $t_k \coloneqq k/2^n$ for $k = 0,1,\dots,2^n$. Let $F \in
  C_b\left(\R^{d(2^n - 1)}\right)$. Applying~\eqref{E:fdd-test} with $k = 2^n -
  1$, we obtain
  \begin{multline*}
    \bE_\beta^W \left[F\left(X_{t_1},\dots,X_{t_{2^n - 1}}\right)\right]
    =  \frac{1}{\cz(0,0; \, 1, *;\, \beta)} \int_{\R^{d(2^n - 1)}} \ud \mathbf{x}
    \; F(x_1,\dots,x_{2^n - 1})                                                                     \\
       \times \prod_{k = 0}^{2^n - 2} \cz \left(t_k, x_k; \, t_{k + 1}, x_{k + 1}; \, \beta\right)
    \cz \left(t_{2^n - 1}, x_{2^n - 1}; \, 1, *; \, \beta\right),
  \end{multline*}
  where we use the convention $x_0 \coloneqq 0$. On the other hand, under
  $\bP_0$ we have
  \[
    \cz(s, y \,  ; t, x \, ; 0) = p_{t - s} (x - y)  ,
  \]
  and $\cz(t,x;\,1,*;\,0)=1$. Hence
  \begin{align*}
    \bE_0 \left[F\left(X_{t_1},\dots,X_{t_{2^n - 1}}\right)\right]
    = \int_{\R^{d(2^n - 1)}} \ud \mathbf{x} \; F(x_1,\dots,x_{2^n - 1})
    \prod_{k = 0}^{2^n - 2} p_{2^{-n}}(x_{k + 1} - x_k).
  \end{align*}
  Therefore,
  \begin{align*}
    \bE_\beta^W \left[F\left(X_{t_1},\dots,X_{t_{2^n - 1}}\right)\right]
    = \bE_0 \left[F\left(X_{t_1},\dots,X_{t_{2^n - 1}}\right) \, Y_n(W,X)\right],
  \end{align*}
  where $Y_n(W,X)$ is given by~\eqref{E:Y_n-RN} below. This
  implies~\eqref{E:abs-cont} with the following Radon--Nikodym derivative:
  \begin{align}\label{E:Y_n-RN}
    Y_n (W, X)
    \coloneqq
    \left.\frac{\ud \PP_{\beta}}{\ud \PP_0 } \right|_{\cg_n} \!\!\left( W, X \right)
    = \frac{M_n (W, X ; \, \beta) }{\cz (0,0; \, 1,*; \, \beta)}.
  \end{align}
  where the random variables $M_n (W, X ; \, \beta)$ are given by
  \begin{align}\label{E:def-m_n}
    M_n (W, X ; \, \beta) \coloneqq \prod_{k = 1}^{2^n} Q_k , \quad \text{with}
    \  Q_k \coloneqq
    \begin{dcases}
      \frac{\cz \left( \alpha_{n,k - 1} (X); \, \alpha_{n, k} (X); \, \beta \right)}{p \left(\alpha_{n,k} (X) - \alpha_{n, k - 1} (X)\right)}, & 1 \leq k < 2^{n}, \\
      \cz \left( 1 - 2^{-n}, X_{1 - 2^{- n}}; \, 1, * ; \, \beta\right),                                                                       & k = 2^n.
    \end{dcases}
  \end{align}
  In order to interpret~\eqref{E:Y_n-RN}, set
  \[
    \mathcal{G}_n^X \coloneqq \sigma \left( X \left(\frac{k}{2^n}\right) : 0 \le k \le 2^n - 1 \right),
    \quad \text{so that } \cg_n = \mathcal{F} \otimes \mathcal{G}_n^X.
  \]
  Recall from Definition~\ref{D:PolyMeasure} (see also Remark~\ref{R:beta-0} for
  $\beta = 0$) that the annealed measures are product measures, namely,
  \begin{equation}\label{E:annealed-factorization}
    \PP_\beta(\ud W, \ud X) = \PP(\ud W)\,\bP_\beta^W(\ud X)
    \quad \text{and} \quad
    \PP_0(\ud W, \ud X) = \PP(\ud W)\,\bP_0(\ud X).
  \end{equation}
  By~\eqref{E:annealed-factorization}, it follows that for $\PP$-almost every
  $W$, the map $X \mapsto Y_n(W,X)$ is a version of the Radon--Nikodym
  derivative
  \[
    \left.\frac{\ud \bP_\beta^W}{\ud \bP_0 } \right|_{\mathcal{G}_n^X} \!\!(X).
  \]
  Moreover, in the expression~\eqref{E:def-m_n}, $\alpha_{n,k} \colon C ([0,1];
  \R^d) \to [0,1]\times \R^d$ is defined by
  \begin{equation}\label{E:alpha-nk}
    \alpha_{n,k}(X)
    = \left(\frac{k}{2^n}, X\left(\frac{k}{2^n}\right)\right),
    \quad \text{for integers $n \ge 1$ and $0 \le k \le 2^n$.}
  \end{equation}
  According to a general convergence result taken
  from~\cite{morters.peres:10:brownian}*{Theorem~12.32 on p.~357}, the following
  facts hold:
  \begin{enumerate}[label=\textnormal{(\roman*)}]

    \item The sequence $\{Y_n (W, X); n \ge 1\}$ in~\eqref{E:Y_n-RN} is a
      positive martingale with respect to the filtration $\{\cg_n ; n \ge 1\}$.
      Therefore, it converges $\PP_0$-a.s. to a random variable:
      \begin{equation}\label{E:Y-limit}
        Y (W, X)= \PP_0\text{ a.s}-\lim_{n\to\infty} Y_n  (W, X)
      \end{equation}

    \item \label{I:abs-cont:2} If $\PP_0 (Y (W, X) = 0) = 1$, then $\PP_\beta$
      and $\PP_0$ are singular.

    \item \label{I:abs-cont:3} If $\PP_0 (Y (W, X) > 0) = 1$, then $\PP_0 \ll
      \PP_\beta$.

    \item \label{I:abs-cont:4} If $\PP_{\beta} (Y (W, X) = \infty) = 0$, then
      $\PP_{\beta} \ll \PP_0$.

  \end{enumerate}
  \bigskip

  We now divide our study in two cases:

  \noindent\textit{Case 1: $\widehat{f}\left(\R^d\right) = \infty$.} Assume
  Condition~\ref{Cond:k-t-to-infty}. Recalling that $\cz (0,0; \, 1,*;\, \beta) >
  0$ by Theorem~\ref{T:transition-properties}-\ref{T:transition-properties:positive} and invoking
  Proposition~\ref{P:martingale-vanishing} below, we obtain that $M_n(W,X;\beta)\to 0$ and
  hence $Y (W, X) = 0$, $\PP_0$-a.s. Therefore, $\PP_\beta \perp \PP_0$
  by~\ref{I:abs-cont:2} above.

  It remains to prove the statement about the quenched singularity. Since
  $\PP_\beta \perp \PP_0$, there exists a set $A \in \mathcal{F} \otimes
  \mathcal{G}$ such that $\PP_\beta(A) = 1$ and $\PP_0(A) = 0$. For $W \in
  \Omega$, define the section $A_W \coloneqq \{X \in C([0,1]; \R^d) : (W,X) \in
  A\} \in \mathcal{G}$. Using~\eqref{E:annealed-factorization}, Fubini's
  theorem yields
  \[
    0 = \PP_0(A)     = \int_\Omega \bP_0(A_W)\,\PP(\ud       W), \qquad
    1 = \PP_\beta(A) = \int_\Omega \bP_\beta^W(A_W)\,\PP(\ud W).
  \]
  Therefore, there exists a set $\Omega_0 \subset \Omega$ with $\PP(\Omega_0) =
  1$ such that for every $W \in \Omega_0$ we have $\bP_0(A_W) = 0$ and
  $\bP_\beta^W(A_W) = 1$, which implies $\bP_\beta^W \perp \bP_0$. \bigskip

  \noindent\textit{Case 2: $\widehat{f}\left(\R^d\right) < \infty$.} In this
  case, $\widehat{f}\left(\R^d\right)=\lim_{t\downarrow 0}k(t)$ is finite (see
  Remark~\ref{R:f0-convention}).  In Proposition~\ref{P:Y-finite-finite-mass} below we will
  prove that in this context we have
  \begin{equation*}
    \PP_{0}\left(     0<Y(W,X) < \infty\right) =1 , \quad\text{and}\quad
    \PP_{\beta}\left( Y(W,X)   < \infty\right) =1 .
  \end{equation*}
  Hence resorting to properties~\ref{I:abs-cont:3}-\ref{I:abs-cont:4} above, we
  conclude that $\mathbb{P}_{\beta} \ll \mathbb{P}_0$, and therefore
  $\mathbb{P}_{\beta} \sim \mathbb{P}_0$. Going back
  to~\eqref{E:annealed-factorization}, this implies that for $\PP$-almost every
  $W$, the quenched measures satisfy $\bP_\beta^W \sim \bP_0$. This completes the
  proof of Theorem~\ref{T:singularity-criterion}.
\end{proof}

Our next proposition handles the martingale convergence necessary to establish
the singularity of $\bP_\beta^W$ and $\bP_0$.

\begin{proposition}\label{P:martingale-vanishing}
  Under Condition~\ref{Cond:k-t-to-infty}, for all $\beta > 0$, $M_n (W, X ; \,
  \beta)$, given by~\eqref{E:def-m_n}, satisfies
  \begin{align}\label{E:Mn-to-0}
    \lim_{n \to \infty} M_n (W, X; \, \beta) = 0, \quad \text{$\PP_0$-almost surely.}
  \end{align}
\end{proposition}
\begin{proof}
  This proposition generalizes
  \cite{alberts.khanin.ea:14:continuum}*{Theorem~4.5}, and its proof follows the same
  strategy to that employed therein. In the sequel, we thus focus on the
  nontrivial adaptation to our current noise model. We now divide this proof in
  several steps. Throughout this proof, all probabilities and expectations are
  taken with respect to $\PP_0$ on $\left(\Omega \times C([0,1];\R^d),
  \mathcal{F} \otimes \mathcal{G}\right)$.

  \noindent\textit{Step~1: Reduction to a series convergence.} Our first step is
  to observe (using the same argument as in the proof of
  Theorem~\ref{T:singularity-criterion}) that $M_n$ is a positive martingale. Therefore, it
  converges almost surely. In order to prove~\eqref{E:Mn-to-0}, it is thus
  sufficient to analyze the convergence in probability. Otherwise stated,
  integrating now over the randomness in $W$, we are reduced to prove that for
  all $\delta \in (0, 1)$, we have
  \begin{align*}
    \lim_{n \to \infty} \PP_0 \left(M_n (W, X; \, \beta) > \delta\right)  = 0,
  \end{align*}
  or equivalently,
  \begin{align}\label{E:lim-m_n}
    \lim_{n \to \infty} \PP_0 \bigg(\log \left(M_n (W, X; \, \beta)\right) > \log (\delta)\bigg)  = 0 .
  \end{align}
  In order to prove~\eqref{E:lim-m_n}, let us consider the blocks $Q_k$
  in~\eqref{E:def-m_n}, and recall the notation $\alpha_{n,k}$
  from~\eqref{E:alpha-nk}. Using Theorem~\ref{T:transition-properties}-\ref{T:transition-properties:independence},
  and conditioning now on $X$, we see that $\left\{ Q_k; 1 \le k \le 2^n
  \right\}$ is a family of independent random variables. Furthermore, owing to
  Theorem~\ref{T:transition-properties}-\ref{T:transition-properties:stationarity} and~\ref{T:transition-properties:homogeneity}, we
  have
  \begin{subequations}\label{E:Q-distr}
    \begin{equation}\label{E:Qk-distr}
      Q_k
      = \frac{\cz \left( \alpha_{n,k - 1} (X); \, \alpha_{n, k} (X); \, \beta \right)}{p \left(\alpha_{n,k} (X) - \alpha_{n, k - 1} (X)\right)}
      \overset{(d)}{ = } \frac{\cz \big(0 , 0; \, 2^{- n}, 0; \, \beta\big)}{p \big(2^{- n}, 0 \big)},
      \quad \text{for $1 \le k < 2^n$},
    \end{equation}
    while for $k = 2^n$ we get
    \begin{equation}\label{E:Qn-distr}
      Q_{2^n}
      = \cz \left( 1 - 2^{-n}, X_{1 - 2^{- n}}; \, 1, * ; \, \beta\right)
      \overset{(d)}{ = } \cz \big(0 , 0; \, 2^{- n}, *; \, \beta\big) .
    \end{equation}
  \end{subequations}

  Plugging~\eqref{E:Q-distr} into~\eqref{E:def-m_n}, then~\eqref{E:lim-m_n}, we
  obtain
  \begin{align}\label{E:logMn-sum}
    \log \left(M_n (W, X; \, \beta)\right) \overset{(d)}{=} \sum_{i = 1}^{2^n} \log \left( 1 +A_{n, i}\right),
  \end{align}
  where, for each fixed $n$, the variables $\{A_{n, i}\}_{1 \le i < 2^n}$ are
  i.i.d. copies of the random variable
  \begin{align}\label{E:chaos-an}
    A_n \coloneqq \frac{\cz \big(0 , 0; \, 2^{- n}, 0; \, \beta\big)}{p_{2^{- n}} (0)} - 1
    =  \sum_{k=1}^{\infty}\frac{ \beta^k I_k(g_k(\cdot \, ; 0 , 0; \, 2^{- n}, 0)) }{p_{2^{- n}} (0)}  .
  \end{align}
  Note that in~\eqref{E:chaos-an} above, the last identity is a direct
  consequence of the chaos expansion~\eqref{E:Chaos-Exp-Z} and
  equation~\eqref{E:mean-cz-heat}. Moreover, for $i = 2^{n}$
  in~\eqref{E:logMn-sum}, we get a slightly different expression. Namely in this
  case $A_{n,2^n}$ is independent of $\{A_{n, i}\}_{1 \le i < 2^n}$ and is
  identical in distribution to
  \begin{align}\label{E:chaos-Atilde}
    \widetilde{A}_n \coloneqq \cz \big(0 , 0; \, 2^{- n}, *; \, \beta\big) - 1 = \sum_{k=1}^{\infty}  \beta^k I_k(g_k(\cdot \, ; 0 , 0; \, 2^{- n}, *))  ,
  \end{align}
  where the functions $g_k$ with a $*$ as the last argument are given by
  (see~\eqref{E:gk}):
  \begin{align}\label{E:gk-star}
    \begin{aligned}
      g_k (\mathbf{s}_k, \mathbf{y}_k \, ; s , y; \, t, *)
      \coloneqq & \int_{\R^d} g_k(\cdot \, ; s , y; \, t, x) \ud x                                                                                            \\
      =         & p_{s_{k} - s_{k - 1}} (y_k - y_{k - 1}) \cdots p_{s_2 - s_1} (y_2 - y_1)  p_{s_1 - s} (y_1 - y) \mathbf{1}_{\cs_k ((s,t])} (\mathbf{s}_k) .
    \end{aligned}
  \end{align}
  As the reader might have observed, the expressions for $i = 2^n$ are similar
  to those for $i=1,\ldots,2^{n}-1$, but with the boundary variable
  $\widetilde{A}_n$. In the sequel, we keep this term explicitly by setting
  $A_{n,2^n}\coloneqq \widetilde{A}_n$ (and later
  $B_{n,2^n}\coloneqq \widetilde{B}_n$), while
  $\{A_{n,i}\}_{1\le i<2^n}$ remain i.i.d. copies of $A_n$. All estimates below
  are then written with the decomposition into the $(2^n-1)$ interior terms plus
  this boundary term; see, for example,~\eqref{E:prob-logMn-cheb}, where both
  $\widetilde{B}_n$ and $B_n$ appear explicitly. \smallskip

  \noindent\textit{Step~2: Strategy.} Having~\eqref{E:lim-m_n}
  and~\eqref{E:logMn-sum} in mind, we want to prove that, as $n \to \infty$, it
  holds that
  \begin{align}\label{E:lim-R_n}
    R_n \coloneqq \sum_{i = 1}^{2^n} \log \left(1 + A_{n,i}\right) \longrightarrow - \infty, \quad \text{in probability.}
  \end{align}
  To this aim, we shall approximate the logarithmic function and use the
  following relation (made rigorous in Step~4) for $R_n$:
  \begin{align}\label{E:approx-R_n}
    R_n \approx \sum_{i = 1}^{2^n} \left(A_{n, i} - \frac{1}{2} A_{n, i}^2\right)~\eqqcolon  \sum_{i = 1}^{2^n} \widehat{B}_{n, i} .
  \end{align}
  We now want to prove that each term $\widehat{B}_{n, i}$
  in~\eqref{E:approx-R_n} contributes to a small negative drift. For this, a
  first order analysis consists in looking at $\E [\widehat{B}_{n, i}]$ and
  $\Var (\widehat{B}_{n, i})$. Let us discuss $\E [\widehat{B}_{n, i}]$ first.
  Specifically, recall from~\eqref{E:chaos-an} and~\eqref{E:chaos-Atilde} that
  we have
  \begin{align}\label{E:A-n-i-chaos}
    \begin{aligned}
      A_{n, 2^n} & \overset{(d)}{=} \sum_{k = 1}^{\infty} \beta^k I_k \left( g_k (\cdot \, ; 0 , 0; \, 2^{- n}, *)\right)  , \quad \text{and},                                         \\
      A_{n, i}   & \overset{(d)}{=} \sum_{k = 1}^{\infty} \beta^k \frac{I_k \left( g_k (\cdot \, ; 0 , 0; \, 2^{- n}, 0)\right) }{p_{2^{- n}} (0)} , \quad \text{for $1 \le i < 2^n$}.
    \end{aligned}
  \end{align}
  It is thus clear, from basic properties of chaos expansions, that $\E
  [A_{n,i}] = 0$. Therefore, we get
  \begin{align*}
    \E \left[\widehat{B}_{n, i}\right] = \E [A_{n, i}] - \frac{1}{2} \E \left[A_{n,i}^2\right] = - \frac{1}{2} \E \left[A_{n,i}^2\right].
  \end{align*}
  Now an important step is that on small scales (here meaning $2^{-n}$ small),
  the dominant term on the right-hand side of~\eqref{E:A-n-i-chaos} is the
  Gaussian term (which corresponds to $k = 1$). We will thus see in Step~4 that
  \begin{align}
    - \E \left[A_{n,i}^2\right] \approx - \frac{\beta^{2}}{p_{2^{- n}} (0)}
    \E \left\{\left[I_1 \left( g_1 (\cdot \, ; 0 , 0; \, 2^{- n}, 0)\right)\right]^2\right\} .
  \end{align}
  Then, some considerations similar to those of Lemma~\ref{L:holder} yield
  \[
    - \E \left[A_{n,i}^2\right] \lesssim - \beta^2 h_* (2^{-n}),
  \]
  where the function $h_*$ is introduced in~\eqref{E:h-star} below. The growth
  of $h_*$ (more precisely,~\eqref{E:hstar-over-t-infty}) is then enough to
  ensure that the cumulative drift diverges to $-\infty$:
  \begin{align}\label{E:EAni2-to-infty}
    - \sum_{i = 1}^{2^n} \E \left[A_{n,i}^2\right] \to - \infty.
  \end{align}
  This will be made rigorous in Steps~3--5, via~\eqref{E:B_n-mean}
  and~\eqref{E:hstar-over-t-infty}. Moreover, thanks to an extra effort
  analyzing variances, one can plug~\eqref{E:EAni2-to-infty}
  into~\eqref{E:approx-R_n} and obtain relation~\eqref{E:lim-R_n}. This will
  yield our claim. We now proceed to a rigorous version of this proof. \medskip

  \noindent\textit{Step~3: Reduction to a moment computation.} The logarithmic
  function in~\eqref{E:logMn-sum} is not easy to handle. We will thus replace it
  by a piecewise polynomial function. Specifically, one can check the following
  inequality by means of elementary computations:
  \begin{align}\label{E:ele-log}
    \log (1 + x) \le x - \frac{1}{2} x^2 \1_{(-1, 0)} (x), \quad \text{for all } x \in (-1, \infty).
  \end{align}
  With this inequality in mind, we now set
  \begin{align}\label{E:def_bn-an}
    B_n \coloneqq A_n - \frac{1}{2} A_n^2  \1_{(-1, 0)} \left( A_n \right), \quad
    \widetilde{B}_n \coloneqq \widetilde{A}_n - \frac{1}{2} \widetilde{A}_n^2 \1_{(-1, 0)} \big( \widetilde{A}_n \big),
  \end{align}
  and similarly, $B_{n, i}$ are defined with $A_n$ replaced by $A_{n, i}$.
  Resorting to~\eqref{E:ele-log} and the fact (taken from
  Theorem~\ref{T:transition-properties}-\ref{T:transition-properties:positive}) that $A_{n, i} > -1$ almost
  surely, we have $B_{n, i} \ge \log (1 + A_{n,i})$. This yields
  \begin{align}\label{E:prob-logMn-bound}
    \PP_0 \bigg(\log \left(M_n (W, X; \, \beta)\right) > \log (\delta) \bigg)
    =   & \PP_0 \left(\sum_{i = 1}^{2^n} \log \left( 1 +A_{n, i}\right) > \log (\delta)\right) \nonumber        \\
    \le & \PP_0 \left(\sum_{i = 1}^{2^n} B_{n,i} > \log (\delta)\right) \quad \text{for all } \delta \in (0,1).
  \end{align}
  Now fix an arbitrary $\delta \in (0,1)$. With
  inequality~\eqref{E:prob-logMn-bound} in hand, we now claim that the limit
  expression in~\eqref{E:lim-m_n} holds, if there exist universal positive
  constants $N$ and $C_{\beta}$ independent of $n$, such that for all $n \ge N$,
  \begin{subequations}\label{E:B_n}
    \begin{align}\label{E:B_n-mean}
      \E \left[ B_{n} \right] \le  -  C_\beta h_* (2^{-n}) , \qquad
      \E \big[ \widetilde{B}_n \big] \le  -  C_\beta h_* (2^{-n}),
    \end{align}
    and
    \begin{align}\label{E:B_n-var}
      \Var \left( B_{n} \right) \le C_\beta h_* (2^{-n}), \qquad
      \Var \big( \widetilde{B}_n \big) \le   C_\beta h_* (2^{-n}),
    \end{align}
  \end{subequations}
  where, with $k$ defined in~\eqref{E:k_t}, the function $h_* \colon \R_+ \to
  \R$ is given by
  \begin{align}\label{E:h-star}
    h_* (t) \coloneqq \int_0^{t} k \left(\frac{2 s (t - s)}{t}\right) \ud s.
  \end{align}
  Indeed, combining~\eqref{E:prob-logMn-bound}, Chebyshev's inequality and the
  independence of the random variables $B_{n,i}$, we conclude that
  \begin{align*}
    \MoveEqLeft
    \PP_0 \bigg(\log \left(M_n (W, X; \, \beta)\right) > \log (\delta) \bigg) \le  \PP_0 \left(\sum_{i = 1}^{2^n} B_{n,i} > \log (\delta)\right)                        \\
    =   & \PP_0 \left(\sum_{i = 1}^{2^n} B_{n,i} - \sum_{i = 1}^{2^n}  \E \left[B_{n,i} \right] > \log (\delta) -  \sum_{i = 1}^{2^n}  \E \left[B_{n,i} \right] \right) \\
    \le & \frac{(2^n - 1) \Var (B_n) + \Var (\widetilde{B}_n)}{\left(\log (\delta) -  (2^n - 1)  \E \left[B_{n} \right] - \E [\widetilde{B}_n ]\right)^2} .
  \end{align*}
  Therefore, invoking~\eqref{E:B_n-mean}-\eqref{E:B_n-var}, we get
  \begin{align}\label{E:prob-logMn-cheb}
    \PP_0 \bigg(\log \left(M_n (W, X; \, \beta)\right) > \log (\delta) \bigg)
    \le \frac{C_{\beta} 2^{n}  h_* (2^{-n})}{\left( C_{\beta} 2^{n}  h_* (2^{-n}) - |\log (\delta)|\right)^2} .
  \end{align}
  In addition, we will see that under Condition~\ref{Cond:k-t-to-infty}, we have
  \begin{align}\label{E:hstar-over-t-infty}
    \lim_{t \downarrow 0} \frac{1}{t} h_* (t) = \lim_{t \downarrow 0} \frac{1}{t} h_1 (t) = \infty.
  \end{align}
  Reporting this information into~\eqref{E:prob-logMn-cheb}, we get
  \[
    \PP_0 \bigg(\log \left(M_n (W, X; \, \beta)\right) > \log (\delta) \bigg)
    = O\left( \frac{1}{2^n h_*(2^{-n})} \right).
  \]
  The above identity easily entails~\eqref{E:lim-m_n}, and therefore our main
  claim~\eqref{E:Mn-to-0}.

  Let us now handle our assertion~\eqref{E:hstar-over-t-infty}. Starting from the
  definition~\eqref{E:h-star} of $h_*$, by symmetry we have
  \begin{align}\label{E:hstar-over-t}
    \frac{h_* (t)}{t} = \frac{2}{t} \int_0^{\frac{t}{2}} k \left( \frac{2 s (t - s)}{t}\right) \ud s.
  \end{align}
  Next, performing the change of variable $s / t \mapsto u$, we get
  \begin{align}\label{E:hstar-over-t-lower}
    \frac{h_* (t)}{t} = 2 \int_0^{\frac{1}{2}} k \left( t u (1 - u) \right) \ud u \ge \inf_{s \in (0,t)} k (s).
  \end{align}
  In addition, under Condition~\ref{Cond:k-t-to-infty}, we have
  \[
    \lim_{t \downarrow 0} \inf_{s \in (0,t)} k (s) = \infty.
  \]
  Going back to~\eqref{E:hstar-over-t-lower}, this proves
  relation~\eqref{E:hstar-over-t-infty}, thereby completing the proof of
  Proposition~\ref{P:martingale-vanishing}. It remains to verify
  inequalities~\eqref{E:B_n-mean} and~\eqref{E:B_n-var}. \smallskip

  \noindent\textit{Step~4: Proof of inequality~\eqref{E:B_n-mean} for $B_n$.}
  Recall that $A_n$ and $B_n$ are defined by~\eqref{E:chaos-an}
  and~\eqref{E:def_bn-an}, respectively. Then, it is readily checked from
  Theorem~\ref{T:transition-properties}-\ref{T:transition-properties:center} that $A_n$ is a centered random
  variable. Thus,
  \begin{equation}\label{E:EBn}
    \E [B_n] = - \frac{1}{2} \E \left[A_n^2 \1_{(-1, 0)} (A_n)\right].
  \end{equation}
  Next we decompose $A_n$ into
  \begin{align}\label{E:A-decomp}
    A_n = A_n^{(1)} + A_n^{(2)},
  \end{align}
  where $A_n^{(1)}$ is the Gaussian part of the expansion, and $A_n^{(2)}$ has
  to be thought of as a remainder:
  \[
    A_n^{(1)} \coloneqq  \beta \frac{I_1 (g_1 (\cdot \, ; 0 , 0; \, 2^{- n}, 0))}{p_{2^{-n}} (0)} ,
    \quad \text{and} \quad
    A_n^{(2)} \coloneqq  \frac{ 1}{p_{2^{- n}} (0)} \sum_{k=2}^{\infty}  \beta^k I_k \left( g_k (\cdot \, ; 0 , 0; \, 2^{- n}, 0)\right).
  \]
  Let us first analyze the Gaussian part $A_n^{(1)}$. One can argue as
  in~\eqref{E:I1-over-p1} and~\eqref{E:chaos-1-bridge} with the unit time
  replaced by $2^{-n}$ to deduce that
  \[
    A_n^{(1)} \overset{(d)}{=} \beta \int_0^{2^{-n}} \int_{\R^d} p_{s (1 - 2^n s)} (y) W(\ud s, \ud y).
  \]
  Hence, computing the variance of this Wiener integral according
  to~\eqref{E:W-Density}, we conclude that
  \begin{align}\label{E:A1-gaus}
    A_n^{(1)} \sim \mathcal{N} \left(0, \sigma_n^2\right) \quad \text{with}\quad
    \sigma_n^2 = \beta^2 \int_0^{2^{-n}}  k \left( 2 s \left( 1 - 2^n s\right) \right) \ud s = \beta^2 h_*(2^{-n}).
  \end{align}
  Moreover, by Theorem~\ref{T:transition-properties}-\ref{T:transition-properties:positive}, we have
  $A_n > -1$ almost surely, and therefore
  \[
    \1_{(-1,0)}(A_n)=\1_{(-\infty,0)}(A_n)\quad\text{almost surely.}
  \]
  Consequently, inequality~\eqref{E:B_n-mean} follows directly from
  Lemma~\ref{L:gauss-perturb} below, provided that the following limit result is
  established:
  \begin{align}\label{E:varA2-over-varA1}
    \lim_{n \to \infty} \frac{ {\rm Var } \big(A_n^{(2)}\big) }{{\rm Var } \big( A_n^{(1)} \big) } = 0.
  \end{align}

  In order to prove~\eqref{E:varA2-over-varA1}, we will extend the definition of
  $A_n^{(2)}$ from the specific time $t = 2^{-n}$ to an arbitrary time $t > 0$.
  That is, we set
  \begin{align}\label{E:g_beta}
    g_{\beta} (t)
    \coloneqq \Var \left( \sum_{k=2}^{\infty} \frac{ \beta^k I_k(g_k(\cdot \, ; 0 , 0; \, t, 0)) }{p_t(0)} \right)
    = \frac{1}{p_t(0)^2} \sum_{k=2}^{\infty} \beta^{2k}  \Var \left( I_k(g_k(\cdot \, ; 0 , 0; \, t, 0)) \right) ,
  \end{align}
  where the equality is due to the orthogonality of multiple integrals. Thus, to
  prove inequality~\eqref{E:varA2-over-varA1}, it suffices to show that
  \begin{align}\label{E:gbeta-over-hstar}
    \lim_{t \downarrow 0} \frac{g_{\beta} (t)}{ \beta^2 h_* (t)} = 0.
  \end{align}
  First, by employing the first display in the proof
  of~\cite{balan.chen:18:parabolic}*{Theorem~1.1 on page~2238}, we can write
  \begin{align}\label{E:vargk-bound}
    \frac{1}{p_t(0)^2} \Var \left( I_k(g_k(\cdot \, ; 0 , 0; \, t, 0)) \right) \le 2^k h_k (t),
  \end{align}
  where $h_k$ is given by~\eqref{E:hn-recursion}. Injecting this relation
  into~\eqref{E:g_beta} and invoking Corollary~\ref{C:hn-via-h1}, there exists a
  constant $C_{\eta} > 0$ such that
  \begin{equation}\label{E:gbeta-bound}
    g_{\beta} (t)
    \le h_1 (t) \times \sum_{k = 2}^{\infty} \beta^{2k} \frac{C_{\eta}^{k - 1} t^{(k - 1) \eta} }{\Gamma ((k - 1) \eta + 1 )}
    = h_1 (t) \beta^4 C_{\eta}\, t^{\eta} \times E_{\eta, 1-\eta} \left( \beta^2 C_{\eta}^{1/\eta} t^{\eta} \right),
  \end{equation}
  where $E_{\eta, \gamma} (\cdot)$ is the two-parameter Mittag-Leffler function
  defined by~\eqref{E:mittag-leffler}. On the other hand, recall
  from~\eqref{E:hstar-over-t} that
  \[
    h_* (t) = 2 \int_0^{t/2} k \left(\frac{2 s (t - s)}{t}\right) \ud s.
  \]
  Define $\psi_t(s)\coloneqq 2 s (t-s)/t$ on $[0,t/2]$. Since $\psi_t$ is
  strictly increasing from $0$ to $t/2$, we can set $u=\psi_t(s)$ and use the
  inverse map
  \[
    \psi_t^{-1}(u)= \frac{t-\sqrt{t(t-2u)}}{2}, \qquad
    \left(\psi_t^{-1}\right)'(u)= \frac{\sqrt{t}}{2\sqrt{t-2u}}, \quad
    u\in(0,t/2).
  \]
  Therefore,
  \begin{align}\label{E:hstar-ge-h1}
    h_* (t) = t^{\frac{1}{2}} \int_0^{\frac{t}{2}} \frac{k (u)}{\sqrt{t - 2 u}} \ud u \ge \int_0^{\frac{t}{2}} k (u) \ud u \ge \frac{1}{2}\int_0^{t} k (u) \ud u = \frac{h_1(t)}{2},
  \end{align}
  where the last inequality uses that $k(\cdot)$ is nonincreasing. Therefore,
  \eqref{E:gbeta-over-hstar} follows from~\eqref{E:gbeta-bound}
  and~\eqref{E:hstar-ge-h1}. This completes the proof of
  inequality~\eqref{E:B_n-mean} for $B_n$. \smallskip

  \noindent\textit{Step~5: Proof of inequality~\eqref{E:B_n-var} for $B_n$.} We
  now turn to the proof of~\eqref{E:B_n-var} for $B_n$. Note that from
  expression~\eqref{E:def_bn-an}, we have
  \begin{align}\label{E:var-B}
    \Var (B_n) \le 2 \left( \Var (A_n) + \Var \big( A_n^2 \big) \right).
  \end{align}
  We shall estimate those two terms separately. Now for $\Var (A_n)$, we use
  decomposition~\eqref{E:A-decomp}, which allows to write
  \[
    \Var \left(A_n\right)
    = \Var \left(A_n^{(1)}\right) + \Var \left(A_n^{(2)}\right)
    = \beta^2 h_*(2^{-n}) + g_{\beta} (2^{-n}) ,
  \]
  where we have resorted to~\eqref{E:A1-gaus} and~\eqref{E:g_beta} for the
  second identity. In addition, relation~\eqref{E:gbeta-over-hstar} entails that
  for $n$ large enough, we have
  \begin{align}\label{E:varA-bound}
    \Var \left(A_n\right) \le 2 \beta^2 h_*(2^{-n}) .
  \end{align}
  Let us now upper bound $\Var (A_n^2)$ in~\eqref{E:var-B}. To this aim, we rely
  on the chaos decomposition~\eqref{E:chaos-an} and the hypercontractivity in
  fixed Wiener chaos (see, e.g., \cite{hu:17:analysis}*{Formula (7.1.41)}). This
  yields
  \begin{align*}
    \left[\Var \left(A_n^2\right)\right]^{1/4}
    \le  \Norm{A_n}_4 \le & \frac{ 1 }{p_{2^{- n}}(0)} \sum_{k=1}^{\infty} \beta^k \Norm{ I_k \left( g_k(\cdot \, ; 0 , 0; \, 2^{- n}, 0) \right)}_4      \\
    \le                   & \frac{1}{p_{2^{- n}} (0)}  \sum_{k=1}^{\infty} (3\beta)^k \Norm{ I_k \left( g_k(\cdot \, ; 0 , 0; \, 2^{- n}, 0) \right)}_2 .
  \end{align*}
  Therefore, owing to~\eqref{E:vargk-bound} and Corollary~\ref{C:hn-via-h1}, for some
  $C_\eta > 0$, we have that
  \begin{align*}
    \left[\Var \left(A_n^2\right)\right]^{1/4}
     & \le \sum_{k=1}^{\infty} (3\beta)^k 2^{\frac{k}{2}} h_k \left(2^{-n}\right)^{1/2}
    \le \sum_{k=1}^{\infty} \left(3\beta \sqrt{2}\right)^k  \left[\frac{C_\eta^{k-1} 2^{-(k-1)\eta}  h_1 (2^{-n})}{\Gamma\left((k-1) \eta + 1\right)}\right]^{1/2}                                                                    \\
     & = 3 \beta \sqrt{2}\; \Theta \times \sqrt{h_1 (2^{-n})}, \quad \text{where }  \Theta \coloneqq \sum_{k=0}^{\infty} \frac{\left(3\beta \sqrt{C_\eta}\, 2^{(1-\eta)/2} \right)^k}{\Gamma\left(k \eta + 1\right)^{1/2} } < \infty.
  \end{align*}
  Hence, we can conclude that
  \begin{align}\label{E:varA2-bound}
    \Var \left(A_n^2\right)
    \le   C_{\beta, \eta}\, h_1^2 \left(2^{-n}\right)
    \le   C_{\beta, \eta}\, h_1(1)\, h_1 \left(2^{-n}\right)
    \le 2 C_{\beta, \eta}\, h_1(1)\, h_* \left(2^{-n}\right),
  \end{align}
  where the second inequality uses the fact that
  \[
    h_1(t)
    = \int_0^t k(t-s) \ud s
    = 2 (2 \pi )^{-d} \int_{\R^d} \frac{1-\exp(-2^{-1}t |\xi|^2)}{|\xi|^2} \widehat{f}(\ud \xi)
  \]
  is a non-decreasing function (see also Lemma 2.6
  of~\cite{chen.kim:19:nonlinear} for a more general statement), and the last
  inequality in~\eqref{E:varA2-bound} is due to~\eqref{E:hstar-ge-h1}.
  Gathering~\eqref{E:varA-bound} and~\eqref{E:varA2-bound} into~\eqref{E:var-B},
  this completes the proof of~\eqref{E:B_n-var} for the term $B_{n}$.

  \noindent\textit{Step~6: Proof of inequalities in~\eqref{E:B_n} for
  $\widetilde{B}_n$.} Recall that $\widetilde{A}_n$ and $\widetilde{B}_n$ are
  defined by~\eqref{E:chaos-Atilde} and~\eqref{E:def_bn-an}, respectively.
  Proceed exactly as in~\eqref{E:EBn} and write
  \[
    \E [\widetilde{B}_n] = - \frac{1}{2} \E \left[\widetilde{A}_n^2 \1_{(-1, 0)} (\widetilde{A}_n)\right].
  \]
  Next we decompose $\widetilde{A}_n$ into
  \begin{align}\label{E:Atilde-decomp}
    \widetilde{A}_n = \widetilde{A}_n^{(1)} + \widetilde{A}_n^{(2)},
  \end{align}
  where $\widetilde{A}_n^{(1)}$ is the Gaussian part of the expansion, and
  $\widetilde{A}_n^{(2)}$ has to be thought of as a remainder:
  \[
    \widetilde{A}_{n}^{(1)} \coloneqq \beta I_1 (g_1 (\cdot \, ; 0 , 0; \, 2^{- n}, *)) \quad \text{and} \quad
    \widetilde{A}_{n}^{(2)} \coloneqq \sum_{k=2}^{\infty}  \beta^k I_k \left( g_k (\cdot \, ; 0 , 0; \, 2^{- n}, *)\right).
  \]
  Let us first analyze the Gaussian part $\widetilde{A}_n^{(1)}$.
  From~\eqref{E:gk-star}, we see that
  \[
    \widetilde{A}_n^{(1)} \overset{(d)}{=} \beta\int_0^{2^{-n}} \int_{\R^d} p_{s} (y) W(\ud s, \ud y).
  \]
  Hence, computing the variance of this Wiener integral according
  to~\eqref{E:W-Density}, we conclude that
  \begin{align}\label{E:A1-gaus-t}
    \widetilde{A}_n^{(1)} \sim \mathcal{N} \left(0, \widetilde{\sigma}_n^2\right), \quad \text{with}\quad
    \widetilde{\sigma}_n^2 = \beta^2 \int_0^{2^{-n}}  k \left( 2 s \right) \ud s = \frac{\beta^2}{2} h_1 (2^{- n + 1}).
  \end{align}
  We now make explicit the two estimates in~\eqref{E:B_n} for
  $\widetilde{B}_n$.
  First, define for $t>0$
  \[
    \widetilde{g}_{\beta}(t)\coloneqq
    \Var\left(\sum_{k=2}^{\infty}\beta^k I_k\left(g_k(\cdot \, ; 0,0; \, t,*)\right)\right).
  \]
  Also, using the same argument as for~\eqref{E:vargk-bound}, one gets for every
  $k\ge 1$:
  \begin{align}\label{E:vargk-bound-star}
    \Var \left( I_k(g_k(\cdot \, ; 0 , 0; \, t, *)) \right) \le 2^k h_k (t).
  \end{align}
  Therefore, invoking Corollary~\ref{C:hn-via-h1}, we obtain
  \[
    \widetilde{g}_{\beta}(t)
    \le h_1(t)\sum_{k=2}^{\infty}\beta^{2k}
    \frac{C_{\eta}^{k-1}t^{(k-1)\eta}}{\Gamma((k-1)\eta+1)}
    \le C_{\beta,\eta}\, h_1(t)\, t^{\eta}.
  \]
  Combined with~\eqref{E:A1-gaus-t}, this gives
  \[
    \frac{\Var\left(\widetilde{A}_n^{(2)}\right)}
    {\Var\left(\widetilde{A}_n^{(1)}\right)}
    = \frac{\widetilde{g}_{\beta}(2^{-n})}
    {\frac{\beta^2}{2}h_1(2^{-n+1})}
    \xrightarrow[n\to\infty]{}0.
  \]
  In addition, since $\widetilde{A}_n = \cz(0,0;\,2^{-n},*;\,\beta)-1$ and
  $\cz(0,0;\,2^{-n},*;\,\beta)>0$ almost surely, we have
  $\widetilde{A}_n>-1$ almost surely, hence
  \[
    \1_{(-1,0)}(\widetilde{A}_n)=\1_{(-\infty,0)}(\widetilde{A}_n)\quad
    \text{almost surely.}
  \]
  Hence, applying Lemma~\ref{L:gauss-perturb} exactly as in Step~4, for $n$
  large enough we obtain
  \[
    \E\left[\widetilde{B}_n\right]\le -C_{\beta}\,h_1(2^{-n}).
  \]
  Next, since $k(\cdot)$ is nonincreasing and
  $2s(t-s)/t\ge \min\{s,t-s\}$ on $[0,t]$, we have
  \[
    h_*(t)=2\int_0^{t/2}k\!\left(\frac{2s(t-s)}{t}\right)\ud s
    \le 2\int_0^{t/2}k(s)\ud s
    = h_1(t).
  \]
  Therefore, after possibly reducing $C_{\beta}$, we conclude that
  \[
    \E\left[\widetilde{B}_n\right]\le -C_{\beta}\,h_*(2^{-n}).
  \]

  Second, we now prove the variance bound in~\eqref{E:B_n-var}. As
  in~\eqref{E:var-B},
  \[
    \Var\!\left(\widetilde{B}_n\right)
    \le 2\!\left(\Var\!\left(\widetilde{A}_n\right)
    + \Var\!\left(\widetilde{A}_n^2\right)\right).
  \]
  For the first term, by~\eqref{E:Atilde-decomp},~\eqref{E:A1-gaus-t}
  and~\eqref{E:vargk-bound-star},
  \begin{align*}
    \Var\!\left(\widetilde{A}_n\right)
    = \Var\!\left(\widetilde{A}_n^{(1)}\right)
    + \Var\!\left(\widetilde{A}_n^{(2)}\right)
    = \frac{\beta^2}{2}h_1(2^{-n+1}) + \widetilde{g}_{\beta}(2^{-n})
    \le C_{\beta}\,h_1(2^{-n}),
  \end{align*}
  for $n$ large, where we used $\widetilde{g}_{\beta}(2^{-n}) \le
  C_{\beta,\eta}h_1(2^{-n})2^{-n\eta}$ and $h_1(2^{-n+1})\le 2h_1(2^{-n})$. For
  the second term, by hypercontractivity in fixed Wiener chaos
  and~\eqref{E:vargk-bound-star},
  \begin{align*}
    \left[\Var\!\left(\widetilde{A}_n^2\right)\right]^{1/4}
    &\le \Norm{\widetilde{A}_n}_4
    \le \sum_{k=1}^{\infty}(3\beta)^k
    \Norm{I_k\!\left(g_k(\cdot \, ;0,0;\,2^{-n},*)\right)}_2 \\
    &\le \sum_{k=1}^{\infty}(3\beta)^k 2^{k/2} h_k(2^{-n})^{1/2}
    \le C_{\beta,\eta}\sqrt{h_1(2^{-n})},
  \end{align*}
  which implies
  \[
    \Var\!\left(\widetilde{A}_n^2\right)
    \le C_{\beta,\eta}\,h_1(2^{-n})^2
    \le C_{\beta,\eta}\,h_1(2^{-n}).
  \]
  Combining the two bounds gives $\Var\!\left(\widetilde{B}_n\right)\le
  C_{\beta}\,h_1(2^{-n})$. Using~\eqref{E:hstar-ge-h1}, namely $h_1(t)\le
  2h_*(t)$, we obtain (after enlarging $C_{\beta}$ if needed)
  $\Var\!\left(\widetilde{B}_n\right)\le C_{\beta}\,h_*(2^{-n})$. This
  proves~\eqref{E:B_n-mean}--\eqref{E:B_n-var} for $\widetilde{B}_n$, and
  concludes the proof of Proposition~\ref{P:martingale-vanishing}.
\end{proof}

The proof of Proposition~\ref{P:martingale-vanishing} relies on the
decomposition~\eqref{E:A-decomp} into a Gaussian part and a remainder. The lemma
below shows how to handle variance estimates for this type of sums. It can be
interpreted as a generalization of~\cite{alberts.khanin.ea:14:continuum}*{Lemma
4.6}. Indeed, under the specialization $\epsilon = \frac{1}{n}$, $a_n =
\frac{1}{n^2}$, and $b_n = O(\frac{1}{n^4})$, Lemma~\ref{L:gauss-perturb}
reduces precisely to~\cite{alberts.khanin.ea:14:continuum}*{Lemma 4.6}.

\begin{lemma}\label{L:gauss-perturb}
  Let $\{Z_n; n \ge 1\}$ be a sequence of centered Gaussian random variables
  with variance $a (n)$, and let $\{Y_n; n \ge 1\} $ be a sequence of centered
  random variables with variance $b(n)$. Define $X_n \coloneqq Z_n + Y_n$ for
  each $n \ge 1$. Then,
  \begin{align}\label{E:gauss-perturb-neg}
    \lim_{n \to \infty} \frac{b(n)}{a(n)} = 0 \quad \implies \quad
    \liminf_{n \to \infty} \frac{1}{a(n)}\E \left[X_n^2 \1_{(-\infty, 0)}\left(X_n\right) \right] \ge \frac{1}{2} \,.
  \end{align}
\end{lemma}
\begin{proof}
  The Gaussian contribution is explicit. By symmetry of the Gaussian
  distribution and the relation $a(n) = \Var (Z_n)$, it follows that
  \[
    \E \left[Z_n^2 \1_{(-\infty, 0)} (Z_n)\right] = \frac{a (n)}{2}.
  \]
  Therefore,
  \begin{align*}
    \E \left[X_n^2 \1_{(-\infty, 0)}\left(X_n\right) \right]
    = \frac{a(n)}{2} + I_1 + I_2 + I_3,
  \end{align*}
  where the terms $I_1$, $I_2$, and $I_3$ are respectively defined by
  \begin{align*}
    I_1 \coloneqq & \E \Big(Z_n^2 \left[ \1_{(-\infty, 0)} \left(X_n\right) - \1_{(-\infty, 0)} \left(Z_n\right) \right]\Big), \\
    I_2 \coloneqq & \E \left[Y_n^2 \1_{(-\infty, 0)} \left(X_n\right) \right] \quad \text{and} \quad
    I_3 \coloneqq  2\E \left[Z_n Y_n \1_{(-\infty, 0)} \left(X_n\right) \right].
  \end{align*}
  Hence,
  \begin{align}\label{E:Xn-neg-lower}
    \E \left[X_n^2 \1_{(-\infty, 0)}\left(X_n\right) \right]
    \ge & \frac{a(n)}{2} - |I_1| - |I_2| - |I_3|.
  \end{align}
  The estimates for $I_2$ and $I_3$ above are relatively easy and proceed as
  follows:
  \begin{subequations}\label{E:I2I3-bounds}
    \begin{gather}
      |I_2| \le \Var \left(Y_n\right) = b(n) = a(n) \times \frac{b(n)}{a(n)} \quad \text{and} \\
      |I_3|
      \le 2 \E \left[|Z_n Y_n|\right]
      \le 2 \sqrt{\Var \left(Z_n\right) \Var \left(Y_n\right)}
      =   2 \sqrt{a (n) b(n)}
      =   2 a (n)  \sqrt{\frac{b(n)}{a(n)}}.
    \end{gather}
  \end{subequations}

  As for $I_1$, let $L > 0$ be a parameter to be chosen later. The expectation
  in $I_1$ is then evaluated by dividing the integration region by two parts:
  \[
    A_1 \coloneqq \left\{|Y_n| < \sqrt{b(n)} L \right\} \quad \text{and} \quad
    A_2 \coloneqq \left\{|Y_n| \ge \sqrt{b(n)} L \right\}.
  \]
  Then, we can deduce that
  \begin{align}\label{E:gauss-pert-I1-split}
    |I_1| \le I_{1,1} + I_{1,2}
  \end{align}
  where
  \begin{eqnarray}\label{E:gauss-pert-I11-def}
    I_{1,1} & \coloneqq & \Big|\E \Big[Z_n^2 \left[\1_{(-\infty, 0)} \left(Z_n + Y_n\right) - \1_{(-\infty, 0)} \left(Z_n\right) \right] \1_{A_1}\Big]\Big|, \quad \text{and} \\
    I_{1,2} & \coloneqq & \Big|\E \Big[Z_n^2 \left[\1_{(-\infty, 0)} \left(Z_n + Y_n\right) - \1_{(-\infty, 0)} \left(Z_n\right) \right] \1_{A_2}\Big]\Big|. \label{E:gauss-pert-I12-def}
  \end{eqnarray}
  In order to bound $I_{1,1}$ in~\eqref{E:gauss-pert-I11-def}, we notice that
  \begin{align*}
    \1_{(-\infty, 0)} \left(Z_n + Y_n\right) - \1_{(-\infty, 0)} \left(Z_n\right) \ne 0
    \iff & \left\{Z_n + Y_n \ge 0, Z_n < 0 \right\} \cup \left\{Z_n + Y_n < 0, Z_n \ge 0 \right\}  \\
    \iff & \left\{Y_n \ge - Z_n,   Z_n < 0 \right\} \cup \left\{Y_n < - Z_n,   Z_n \ge 0 \right\}.
  \end{align*}
  In both cases, together with the condition from $A_1$, it holds that $|Z_n|
    \le |Y_n| < \sqrt{b(n)} L$. By an elementary change of variable and standard
  estimates for the Gaussian density, we thus have
  \begin{align}\label{E:gauss-pert-I11-bound}
    I_{1,1}
    \le & \E \left[Z_n^2 \1_{\left(-\sqrt{b (n)} L, \sqrt{b (n)} L\right)} \left(Z_n\right) \right]
    =    2\int_0^{\sqrt{b (n)} L} x^2 \frac{ e^{-\frac{x^2}{2 a (n)}}}{\sqrt{2 \pi a (n)}} \ud x \nonumber \\
    =   & \frac{2a (n)}{\sqrt{2 \pi}} \int_0^{\sqrt{\frac{b(n)}{a (n)}} L} z^2 e^{- \frac{z^2}{2}} \ud z
    \le   \frac{2a (n)}{\sqrt{2 \pi}} \left(\sqrt{\frac{b(n)}{a (n)}} L\right)^2 \int_0^{\infty} e^{-\frac{z^2}{2}} \ud z
    =    b(n) L^2.
  \end{align}
  The term $I_{1,2}$ in~\eqref{E:gauss-pert-I12-def} is easier to bound. We just
  apply the Cauchy--Schwarz and Chebyshev's inequalities in order to get
  \begin{align}\label{E:gauss-pert-I12-bound}
    I_{1,2}
    \le & \sqrt{\E \left[Z_n^4\right] \PP\left(A_2\right)}
    \le   \sqrt{3a(n)^2 \times \PP\left(|Y_n| \ge \sqrt{b(n)} L\right)}
    \le   \frac{\sqrt{3}}{L} a(n).
  \end{align}
  Hence, gathering~\eqref{E:gauss-pert-I11-bound}
  and~\eqref{E:gauss-pert-I12-bound} into~\eqref{E:gauss-pert-I1-split} we
  obtain
  \begin{align}\label{E:I1-bound}
    |I_1| \le a (n) \left(\frac{L^2 b(n)}{a(n)} + \frac{\sqrt{3}}{L}\right).
  \end{align}

  We can now summarize our estimates as follows: combining~\eqref{E:I2I3-bounds}
  and~\eqref{E:I1-bound} into~\eqref{E:Xn-neg-lower} and using the notation with
  $r_n \coloneqq b(n) / a(n)$ yield
  \begin{align*}
    \E \left[X_n^2 \1_{(-\infty, 0)}\left(X_n\right) \right]
    \ge a (n) \left(\frac{1}{2} - r_n - 2\sqrt{r_n} - L^2 r_n - \frac{\sqrt{3}}{L}\right)
    \eqqcolon  a (n) \times \Theta(n, L).
  \end{align*}
  If $r_n = 0$, then $b(n)=0$ and hence $Y_n = 0$ almost surely, so the claim is
  immediate. Otherwise, we take $L = r_n^{-1/3}$ and obtain
  \begin{align*}
    \Theta\left(n, r_n^{-1/3}\right)
    \ge \frac{1}{2} - \left(r_n + 2 \sqrt{r_n} + (1+\sqrt{3}) r_n^{1/3}\right)
    \to \frac{1}{2} \quad  \text{as $n \to \infty$}.
  \end{align*}
  This completes the proof of Lemma~\ref{L:gauss-perturb}.
\end{proof}

The proposition below proves the martingale convergence needed for
Theorem~\ref{T:singularity-criterion}, in the situation where the measures $\bP_\beta^W$ and
$\bP_0$ are equivalent.

\begin{proposition}\label{P:Y-finite-finite-mass}
  Let $k$ be the function defined by~\eqref{E:k_t}, and assume
  \begin{equation}\label{E:k-t-to-finite}
    \lim_{t \downarrow 0} k(t) < \infty , \quad\text{or equivalently}\quad
    \widehat{f}\left(\R^d\right) < \infty.
  \end{equation}
  Then the random variable $Y(W,X)$ defined by~\eqref{E:Y-limit} is such that
  \begin{equation}\label{E:Y-finite}
    \PP_{0}\left(  0< Y(W,X) <\infty\right) =1 ,
    \quad\text{and}\quad
    \PP_{\beta}\left(   Y(W,X) <\infty\right) =1 .
  \end{equation}
\end{proposition}

\begin{proof}
  Here again, we divide the proof in several steps.

  \noindent\textit{Step~1: Martingale representation of the partition function.}
  Whenever~\eqref{E:k-t-to-finite} is satisfied, one can evaluate the noise at
  spatial points. Specifically, for a continuous path $X$ we can set
  \begin{equation}\label{E:W-along-path}
    \int_0^1 W(\ud s, X_s)
    \coloneqq
    L^2(\Omega)\text{-}\lim_{\epsilon\downarrow 0}
    \int_0^1 \int_{\R^d} p_\epsilon(X_s-z)\,W(\ud s,\ud z).
  \end{equation}
  The above limit exists by It\^o's isometry and the monotone convergence
  theorem applied to $\int_{\R^d}e^{-\epsilon|\xi|^2}\widehat{f}(\ud\xi)\uparrow
  \widehat{f}\left(\R^d\right)$ as $\epsilon\downarrow 0$ (we refer
  to~\cite{hu.huang.ea:15:stochastic}*{Proposition~4.2} for details about this
  type of computation). Define
  \begin{align}\label{E:Singular-a}
    \mathcal{E}(W,X)
    \coloneqq
    \exp\left(\beta \int_0^1 W(\ud s, X_s) - \frac{\beta^2}{2} \widehat{f}\left(\R^d\right)\right) .
  \end{align}
  We now relate $\mathcal{E}(W,X)$ to the partition functions $\cz$ from
  Definition~\ref{D:PolyMeasure}. Fix $0 \le s < t \le 1$ and $x,y\in\R^d$, and
  let $\bE_{s,y;t,x}$ denote the expected value for a Brownian bridge $X$ going
  from $(s,y)$ to $(t,x)$. For $r\in[s,t]$, set
  \begin{align}\label{E:Singular-b}
    \mathcal{E}_{s,r}(W,X)
    \coloneqq
    \exp\left(\beta \int_s^r W(\ud u, X_u) - \frac{\beta^2}{2} (r-s)\widehat{f}\left(\R^d\right)\right) ,
  \end{align}
  so that the quantity $\mathcal{E} (W, X)$ in~\eqref{E:Singular-a} is also
  $\mathcal{E}_{0,1} (W, X)$ in~\eqref{E:Singular-b}. In addition, the
  filtration introduced in Remark~\ref{R:Filtration} can now be expressed as
  \begin{align*}
  \mathcal{F}_t \coloneqq \sigma \left\{ \int_0^s W(\ud u, x) ; 0 \leq s \leq t, \, x \in \R^d \right\}.
  \end{align*}
  By It\^o's isometry relative to $\mathcal{F}_t$, and for a fixed realization
  of $X$, we have
  \[
    \E\!\left[\left(\int_s^r W(\ud u, X_u)\right)^2 \right]
    =(r-s)\widehat{f}\left(\R^d\right).
  \]
  Hence, it is readily checked that $\{\mathcal{E}_{s,r}\}_{r \in [s,t]}$ is a
  square integrable $\mathcal{F}_r$-martingale. Moreover, it satisfies the
  linear equation:
  \begin{align}\label{E:Singular-d}
    \mathcal{E}_{s,t}(W,X)
    = 1 + \beta \int_s^t \mathcal{E}_{s,r}(W,X)\, W(\ud r, X_r).
  \end{align}
  Define the random field
  \begin{align}\label{E:Singular-c}
    \widehat{\cz}(s,y ; \, t,x; \, \beta)
    \coloneqq
    p_{t-s}(x-y)\,\bE_{s,y;t,x}\!\left[\mathcal{E}_{s,t}(W,X)\right].
  \end{align}
  It is known (see, e.g., \cite{hu.huang.ea:15:stochastic}*{Theorem~5.3}) that
  $\widehat{\cz} = \cz$. For the sake of completeness, we include a detailed
  proof below. Using the fact that~\eqref{E:Singular-d} is a linear equation and
  invoking a stochastic Fubini theorem to exchange the $\int$ and
  $\bE_{s,y;t,x}$ symbols, one can recast~\eqref{E:Singular-c} as
 \begin{equation*}
    \widehat{\cz}  (s,y; \, t,x; \, \beta)
    = p_{t-s}(x-y)
    + \beta \, p_{t-s}(x - y) \int_s^t \int_{\R^d}
      \bE_{s,y;t,x}\left[\mathcal{E}_{s,r} (W, X)  W(\ud r, X_{r}) \right].
  \end{equation*}
  Hence, conditioning on the values of $X_r$ in the right-hand side above, we
  end up with
  \begin{multline}\label{E:se-cZ-hat}
    \widehat{\cz}  (s,y; \, t,x; \, \beta) \\
    = p_{t-s}(x-y)
    + \beta \int_s^t \int_{\R^d}
    p_{t-s}(x - y) p^{(s,y;t,x)}_r (z) \bE_{s,y;r,z}\left[\mathcal{E}_{s,r} (W, X)\right]  W(\ud r,\ud z),
  \end{multline}
  where $p^{(s,y;t,x)}_r (z)$ denotes the probability density of a Brownian
  bridge $X_r$ ( for $s < r < t$), where we are conditioning to $X_s = y$ and
  $X_t = x$. Furthermore, by the density of the Brownian bridge (see
  Lemma~\ref{L:heat-kernel-factorization}), we have
  \begin{align*}
    p^{(s,y;t,x)}_r (z) = \frac{p_{t - r} (x - z) p_{r - s} (z - y)}{p_{t- s} (x - y)} .
  \end{align*}
  Applying the above identity and invoking the definition~\eqref{E:Singular-c}
  of $\widehat{\cz}$, we obtain
  \[
    p_{t-s}(x - y) p^{(s,y;t,x)}_r (z) \bE_{s,y;r,z}\left[\mathcal{E}_{s,r} (W, X) \right]
    = p_{t-r}(x-z)\,\widehat{\cz} (s,y;r,z;\beta).
  \]
  Plugging this relation back into~\eqref{E:se-cZ-hat}, one finds that
  \[
    \widehat{\cz} (s,y; \, t,x; \, \beta)
    = p_{t-s}(x-y)
    + \beta\int_s^t\int_{\R^d} p_{t-r}(x-u)\,\widehat{\cz} (s,y;r,u;\beta)\,W (\ud r,\ud u).
  \]
  Therefore, $\widehat{\cz}$ is a solution to~\eqref{E:PAM-Delta}, and hence
  $\widehat{\cz}=\cz$ by uniqueness; see Theorem~\ref{T:she-well-posedness}. We
  have thus obtained the martingale representation~\eqref{E:Singular-c} for
  $\cz$.

  \noindent \textit{Step 2: Representing $Y$ in terms of the martingale $\ce$.}
  For this step, we first consider ${X}$ as a Brownian bridge from $(0,0)$ to
  $(1,x)$. According to~\eqref{E:Singular-c}, the quenched average of the
  quantity $\mathcal{E}$ in~\eqref{E:Singular-b} corresponding to this process
  is
  \begin{align}\label{E:Singular-e}
    \bE_{0, 0; 1, x} \left[\mathcal{E}_{0, 1} \left(W, {X}\right)\right]
    = \frac{\widehat{\cz} (0,0; \, 1, x ; \, \beta)}{p_1 (x)}
    = \frac{\cz (0,0; \, 1, x ; \, \beta)}{p_1 (x)} .
  \end{align}
  Furthermore, due to the linear nature of equation~\eqref{E:Singular-d}, if one
  considers now $X$ as a Brownian motion on $[0,1]$, we have
  \begin{align*}
    \bE_0\!\left[\mathcal{E}(W,X)\right]
    = & \int_{\R^d} p_1(x)\,\bE_{0,0;1,x}\!\left[\mathcal{E}_{0, 1} \left(W, {X}\right)\right]\ud x ,
  \end{align*}
  where we recall that $\mathcal{E}(W,X)=\mathcal{E}_{0,1}(W,X)$.
  Inserting~\eqref{E:Singular-e} into the above equation, we get
  \begin{align}\label{E:bEcE=Z}
    \bE_0\!\left[\mathcal{E}(W,X)\right]
    = & \int_{\R^d} \cz(0,0;\,1,x;\,\beta)\ud x
    = \cz(0,0;\,1,*;\,\beta).
  \end{align}
  Moreover, by the Markov property of the Brownian motion under $\bP_0$, the
  conditional law of $X$ given $\mathcal{G}_n^X$ is that of independent Brownian
  bridges on each interval $\left[t_{k-1},t_k\right]$ (with $t_k=k/2^n$) and a
  final Brownian motion segment on $\left[1-2^{-n},1\right]$. Therefore, using
  the factorization
  \[
    \mathcal{E}(W,X)
    = \prod_{k=1}^{2^n-1}\mathcal{E}_{t_{k-1},t_k}(W,X)\times \mathcal{E}_{1-2^{-n},1}(W,X),
  \]
  which is trivial by~\eqref{E:Singular-b}, together with the identity (which
  stems from~\eqref{E:Singular-c})
  \[
    \bE_{t_{k-1},X_{t_{k-1}};\,t_k,X_{t_k}}\!\left[\mathcal{E}_{t_{k-1},t_k}(W,X)\right]
    = \frac{\cz\!\left(t_{k-1},X_{t_{k-1}};\,t_k,X_{t_k};\,\beta\right)}{p_{t_k-t_{k-1}}\!\left(X_{t_k}-X_{t_{k-1}}\right)},
  \]
  and its integrated counterpart
  \[
    \bE_{1-2^{-n},X_{1-2^{-n}}}\!\left[\mathcal{E}_{1-2^{-n},1}(W,X)\right]
    = \cz\!\left(1-2^{-n},X_{1-2^{-n}};\,1,*;\,\beta\right),
  \]
  we obtain
  \[
    \bE_0\!\left[\mathcal{E}(W,X)\,\big|\,\mathcal{G}_n^X\right]
    = M_n(W,X;\beta),
  \]
  where $M_n$ is defined in~\eqref{E:def-m_n}, and hence going back
  to~\eqref{E:Y_n-RN} we get
  \begin{align}\label{E:prf_Singular-b}
    Y_n(W,X) = \frac{\bE_0\!\left[\mathcal{E}(W,X)\,\big|\,\mathcal{G}_n^X\right]}{\cz(0,0;\,1,*;\,\beta)}.
  \end{align}
  Now under condition~\eqref{E:k-t-to-finite}, the random variable $\ce (W,X)$
  defined by~\eqref{E:Singular-a} is in $L^2 (\Omega \times C ([0,1]))$;
  see~\cite{hu.huang.ea:15:stochastic}*{Theorem 4.6} for such an estimate. In
  addition, we have argued in~\eqref{E:prf_Singular-a} that $\bigvee_{n \ge 1}
  \cg_n^X = \mathcal{G}$. One can thus take the limit
  in~\eqref{E:prf_Singular-b} in order to get $Y_n(W,X)\to Y(W,X)$ $\PP_0$-a.s.
  with
  \begin{equation}\label{E:Y-radon-nikodym}
    Y(W,X)=\frac{\mathcal{E}(W,X)}{\cz(0,0;\,1,*;\,\beta)} .
  \end{equation}

  \noindent \textit{Step 3: Conclusion.} Starting from~\eqref{E:Y-radon-nikodym}, taking
  account of Theorem~\ref{T:transition-properties}-\ref{T:transition-properties:positive}
  and~\eqref{E:Singular-b}, it is readily checked that $0<Y(W,X)<\infty$
  $\PP_0$-a.s. Furthermore, applying~\eqref{E:bEcE=Z}, we have
  \[
    \E_0\!\left[Y(W,X)\right]
    = \E_0 \! \left[\frac{\bE_0 \!\left[\mathcal{E}(W,X)\right]}{\cz(0,0;\,1,*;\,\beta)}\right]
    = 1 .
  \]
  Since a positive random variable with finite mean must be finite almost
  surely, we get $Y (W,X) < \infty$, $\mathbb{P}_{\beta}$-a.s., which concludes
  the proof of Proposition~\ref{P:Y-finite-finite-mass}.
\end{proof}

\section{Diffusive behavior in high temperature}\label{S:Diffusive}

In this section we establish another basic property of the polymer. Namely we
will see that in dimension $d\ge3$, if the inverse temperature $\beta$ is small
enough, the behavior of the polymer $X$ is diffusive in large time. One thus
expects a phase transition from diffusive behavior (small $\beta$) to
superdiffusive behavior (large $\beta$). Our method of proof is based on
Feynman-Kac type representations, for which we need some preliminary results
included in Section~\ref{S:another-def-polymer}.

\subsection{Another formulation of the polymer measure}\label{S:another-def-polymer}

In this section, we give an alternative formulation of the polymer measure using
Feynman-Kac formula. Moreover, we will show that this new definition is
compatible with Definition~\ref{D:PolyMeasure}, under suitable conditions on the
covariance function of the noise. Our methodology will rely on an approximation
procedure, which is reminiscent of Proposition~\ref{P:Y-finite-finite-mass} and that we
detail now. We start with an approximation of the noise.

\begin{definition}\label{D:smooth-noise}
  Let $W$ be the noise introduced in Definition~\ref{D:Noise}, and recall that
  $p_t$ is the Gaussian kernel from equation~\eqref{E:heat}. For $\epsilon>0$ we
  define a centered Gaussian family $\{W^\epsilon(\phi);
  \phi\in\mathcal{S}(\mathbb{R}_+\times\mathbb{R}^d)\}$ by setting
  \begin{align}\label{E:smooth-noise}
    W^\epsilon(\phi) = W(p_\epsilon*\phi).
  \end{align}
  Otherwise stated, the noise $\dot{W}^\epsilon$ displayed, e.g.,
  in~\eqref{E:Noise-Convention} is defined as
  \[
    \dot{W}^\epsilon(t,x)=\left[p_\epsilon*\dot{W}(t,\cdot)\right](x).
  \]
\end{definition}

\begin{remark}\label{R:f_eps}
  Plugging relation~\eqref{E:smooth-noise} into~\eqref{E:W-Space}, we get the
  following relation for
  $\phi,\psi\in\mathcal{S}(\mathbb{R}_+\times\mathbb{R}^d):$
  \begin{equation*}
    \E\left[W^\epsilon(\phi)W^\epsilon(\psi)\right]
    =  \int_{\mathbb{R}_+}ds\int_{\mathbb{R}^{2d}}\phi(s,x)f_\epsilon(x-y)\psi(s,y)\ud x \ud y,
  \end{equation*}
  where $f_\epsilon=p_{2\epsilon}*f.$ In particular, for a fixed $\epsilon>0$
  the function $f_\epsilon$ is smooth and bounded, with $f_\epsilon(0)<\infty.$
  Moreover, the corresponding spectral measure satisfies
  \begin{equation}\label{E:fhat_eps}
    \widehat{f_\epsilon}(\ud\xi)=e^{-\epsilon|\xi|^2}\widehat{f}(\ud\xi).
  \end{equation}
\end{remark}

Similarly to what we did in Section~\ref{S:prop-heat-eq}, one can consider a
stochastic heat equation like~\eqref{E:SHE} driven by the smoothed noise
$W^\epsilon$. Below we derive some properties of this type of equation, starting
with the case of a flat initial condition.

\begin{proposition}\label{P:FK_smooth}
  Let $W^\epsilon$ be the noise introduced in Definition~\ref{D:smooth-noise}.
  We consider the following equation
  \begin{align}
    \left(\frac{\partial}{\partial t}-\frac{1}{2}\Delta\right)u^\epsilon(t,x)=\beta \, u^{\epsilon}(t,x) \, \dot{W}^\epsilon(t,x),
    \label{E:smoothed-equ}
  \end{align}
  interpreted in the mild sense~\eqref{E:mild}, with initial condition $u(0,x)
    = 1$. Then equation~\eqref{E:smoothed-equ} possesses a unique random field
  solution $u^\epsilon$. Moreover $u^\epsilon$ admits a Feynman-Kac
  representation:
  \begin{align}\label{E:FK_smooth}
    u^\epsilon(t,x)=\bE_0\left[\exp\left(\beta\int_0^t\dot{W}^\epsilon(u,x+X_{t-u})\ud u
      -\frac{\beta^{2}}{2}f_\epsilon(0)t\right)\right],
  \end{align}
  where $X$ is a $\mathbb{R}^d$-valued Brownian motion under the Wiener measure
  $\bP_0$ (see Remark~\ref{R:beta-0}), where $f_\epsilon$ is as in
  Remark~\ref{R:f_eps}.
\end{proposition}

\begin{remark}\label{R:FK_smooth}
  As mentioned in~\eqref{E:W-along-path}, the integral with respect to
  $\dot{W}^\epsilon$ in~\eqref{E:FK_smooth} can be interpreted as the limit of a
  sequence of Wiener integrals. That is one can write
  \begin{align*}
    \int_0^t\dot{W}^\epsilon(u,x+X_{t-u})\ud u=\int_0^t\int_{\mathbb{R}^d}\delta_{x+X_{t-u}}(y)W^\epsilon(\ud u,\ud y),
  \end{align*}
  where the right-hand side has to be understood as
  \begin{equation*}
    \int_0^t\int_{\mathbb{R}^d}\delta_{x+X_{t-u}}(y)W^\epsilon(\ud u,\ud y)\\
    = L^2(\Omega) - \lim_{\eta\to0}\int_0^t\int_{\mathbb{R}^d}p_\eta(y-(x+X_{t-u}))W^\epsilon(\ud u,\ud y).
  \end{equation*}
  See~\cite{hu.nualart:09:stochastic}*{Proposition 3.1}
  or~\cite{hu.huang.ea:15:stochastic}*{Proposition 4.2} for a rigorous version
  of this fact in very similar contexts.
\end{remark}

\begin{proof}[Proof of Proposition~\ref{P:FK_smooth}]
  Equation~\eqref{E:smoothed-equ} satisfies the condition of
  Theorem~\ref{T:she-well-posedness}, which can be applied directly. This proves
  existence and uniqueness of a random field solution. The Feynman-Kac
  representation is borrowed from~\cite{hu.huang.ea:15:stochastic}*{Theorem
  5.3}, with minor adaptations to our context.
\end{proof}

We now extend Proposition~\ref{P:FK_smooth} to a wedge initial
condition. This will provide the building blocks to our new representation of
the smoothed polymer measure.

\begin{proposition}\label{P:FK_smooth_p-p}
  We consider the stochastic heat equation~\eqref{E:PAM-Delta} with wedge
  initial condition $\delta_y$, starting at time $s\ge0$, and driven by the
  noise $W^\epsilon$ from Definition~\ref{D:smooth-noise}. Then this equation
  admits a unique random field solution denoted by
  $\cz^\epsilon(s,y; \, t,x; \, \beta)$. The solution admits a Feynman-Kac
  representation of the form
  \begin{align}\label{E:FK_P-P}
    \cz^\epsilon(s,y; \, t,x; \, \beta)
    =\bE_{s,y;t,x} \left[\ce^{\epsilon}_{s,t} (W, X) \right] p_{t-s}(x-y),
  \end{align}
  where, similarly to~\eqref{E:Singular-b}, we set
  \begin{align}\label{E:FK_smooth_p-p-b}
  \ce^{\epsilon}_{s,t} (W, X) \coloneqq \exp \left(\beta \int_s^t \dot{W}^{\epsilon} (u, X_u) \ud u - \frac{\beta^2}{2} f_{\epsilon} (0) (t - s) \right),
  \end{align}
  and $\bE_{s,y;t,x}$ denotes the expected value for a Brownian bridge $X$
  going from $(s,y)$ to $(t,x)$.
\end{proposition}

\begin{proof}
  The main claim~\eqref{E:FK_P-P} has already been proved in
  Proposition~\ref{P:Y-finite-finite-mass} - Step 1. Let us just mention the trivial
  relation $f_{\epsilon} (0) = \widehat{f}_{\epsilon} (\R^d)$ for a complete
  identification of relations~\eqref{E:Singular-b}
  and~\eqref{E:FK_smooth_p-p-b}.
\end{proof}

As a last preliminary step before giving our alternative definition of the
polymer measure, let us label a representation for a spatial integral of the
process $\cz^\epsilon$.

\begin{lemma}\label{L:Zeps-star}
  For $\epsilon > 0$, let $\cz^\epsilon$ be the field introduced in
  Proposition~\ref{P:FK_smooth_p-p}. We consider the renormalization variable
  $\cz^\epsilon(s,y;t,*;\beta)$ defined similarly to
  Definition~\ref{D:PolyMeasure}:
  \begin{align} \label{E:Zeps_star}
    \cz^\epsilon(s,y;t,*;\beta) = \int_{\mathbb{R}^d}\cz^\epsilon(s,y; \, t,x; \, \beta)\ud x.
  \end{align}
  For $t \ge 0$ we also define a random variable $Z_t^{\epsilon}$ as
  \begin{align}\label{E:Zeps_t}
    Z^\epsilon_t = \bE_0\left[\exp\left(\beta\int_0^t\dot{W}^\epsilon(u,X_u)\ud u\right)\right],
  \end{align}
  where we recall that $\bE_0$ stands for the expectation under the Wiener
  measure. Then we have
  \begin{align}\label{E:Zeps_star_by_Z}
    \cz^\epsilon(0,0;T,*;\beta)=Z^\epsilon_T \, \exp\left(-\frac{\beta^2}{2}f_\epsilon(0)T\right).
  \end{align}
\end{lemma}
\begin{proof}
  By~\eqref{E:FK_P-P} and the definition~\eqref{E:Zeps_star} of
  $\cz^\epsilon(0,0;T,*;\beta)$, we have
  \begin{align*}
    \cz^\epsilon(0,0;T,*;\beta)
    = & \int_{\mathbb{R}^d}\cz^\epsilon(0,0;T,x;\beta) \, \ud x \\
    = & \int_{\mathbb{R}^d}\bE_{0,0;T,x}\left[\ce^{\epsilon}_{0, T} (W, X)
    \right] p_T(x) \, \ud x ,
  \end{align*}
  where we recall that $\ce^{\epsilon}_{s,t} (W, X)$ is defined
  by~\eqref{E:FK_smooth_p-p-b}. Moreover, the integral with respect to $p_T(x)$
  above represents a de-conditioning for the relation $X_T = x$ defining the
  Brownian bridge. We thus get
  \begin{equation}\label{E:Zeps_star_FK}
    \cz^\epsilon(0,0;T,*;\beta)
    =  \bE_{0}\left[\exp\left( \beta\int_0^T\dot{W}^\epsilon(u,X_u) \ud u -\frac{1}{2}\beta^2f_\epsilon(0)T\right)\right],
  \end{equation}
  where we recall that $\bE_{0}$ represents the expected value for a Brownian
  motion starting from the origin. We now identify the right-hand side
  in~\eqref{E:Zeps_star_FK}, thanks to the definition~\eqref{E:Zeps_t}. This
  easily yields our claim~\eqref{E:Zeps_star_by_Z}, which completes the proof of
  Lemma~\ref{L:Zeps-star}.
\end{proof}

Having defined renormalization constants for the smoothed noise, we are ready to
give our alternative definition of the polymer measure.

\begin{definition}\label{D:Polymer-measure-alt}
  For a fixed realization $W^\epsilon$ of the noise in
  Definition~\ref{D:smooth-noise}, and a temperature parameter $\beta > 0$, we
  define a measure $\bP_T^{\beta, W^\epsilon}$ on $C ([0,T]; \R^d)$ equipped
  with the Borel $\sigma$-algebra $\mathcal{G}$, by specifying the finite
  dimensional distributions of the canonical process $X$. Those finite
  dimensional distributions are expressed by considering test functions $\phi
  \in C_{b}(\R^{d\times k})$, some times $0 \le t_1 < \cdots < t_k < T$, and
  setting
  \begin{equation}\label{E:fdd-eps}
    \bE_T^{\beta,W^\epsilon} \left[\phi(X_{t_1},\ldots ,X_{t_k})\right]
    = \frac{1}{Z^\epsilon_T}\bE_0\left[e^{\beta{\int_0^T\dot{W}^\epsilon(u,X_u) \ud u}}\phi(X_{t_1},\ldots ,X_{t_k})\right]
  \end{equation}
  where we recall that $Z^\epsilon_T$ is defined by~\eqref{E:Zeps_t} and
  $\bE_{0}$ is the expectation with respect to a standard Brownian motion $X$
  starting from the origin.
\end{definition}

\begin{remark}
  At first sight, definition~\eqref{E:fdd-eps} looks exactly like the smooth
  polymer measure introduced in~\cite{rovira.tindel:05:on}. However, starting
  from relation~\eqref{E:fdd-eps}, inserting the renormalization term
  $\frac{1}{2}\beta^2f_\epsilon(0)T$ and invoking~\eqref{E:Zeps_star_by_Z}, it
  is readily checked that
  \begin{align}\label{E:fdd-eps-renorm}
    \bE_T^{\beta,W^\epsilon} \left[\phi(X_{t_1},\ldots ,X_{t_k})\right]
    = & \frac{1}{Z^\epsilon_T\,e^{-\frac{1}{2}\beta^2f_\epsilon(0)T}}\bE_0\left[\ce_{0, T} (W, X) \phi(X_{t_1},\ldots ,X_{t_k})\right] \nonumber \\
    = & \frac{1}{\cz^\epsilon(0,0;T,*;\beta)} \bE_0\left[\ce_{0, T} (W, X) \phi(X_{t_1},\ldots ,X_{t_k})\right] .
  \end{align}
  This definition is closer to the one we had in~\eqref{E:fdd-test}.
\end{remark}

An important point for this section is that
Definition~\ref{D:Polymer-measure-alt} is compatible with our previous notion of
polymer, whenever one considers a smoothed noise. We summarize this property in
the proposition below.

\begin{proposition}\label{P:same-def-polymer}
  Let $W^\epsilon$ be the noise given in Definition~\ref{D:smooth-noise}. For a
  fixed realization of $W^\epsilon$, the polymer measure given in
  Definition~\ref{D:Polymer-measure-alt} coincides with that in
  Definition~\ref{D:PolyMeasure}. More specifically, for any test functions
  $\phi \in C_{b}(\R^{d\times k})$ and any tuple of times $0 \le t_1 < \cdots <
  t_k < T$, the natural extension of~\eqref{E:fdd-test} to the time interval
  $[0,T]$ is the same as~\eqref{E:fdd-eps-renorm}.
\end{proposition}

\begin{remark}
  In Section~\ref{S:first-properties-of-polymer}, we have exclusively worked
  with a polymer defined on the interval $[0,1]$. Here we assume that the
  trivial extension of this construction to an arbitrary interval $[0,T]$ is
  already carried out. The polymer measure on $[0,T]$ is denoted by
  $\bP^{\beta,W}_T$ for the singular noise $\dot{W}$ and
  $\bP^{\beta,W^\epsilon}_T$ for the smoothed noise $\dot{W}^{\epsilon}$.
\end{remark}

\begin{proof}[Proof of Proposition~\ref{P:same-def-polymer}]
  Let $\phi$ and $t_1,\dots,t_k$ be as in the proposition. We start from the
  definition~\eqref{E:fdd-eps}, written in the form~\eqref{E:fdd-eps-renorm}.
  Next use simple additivity properties in order to write the exponential
  factor in~\eqref{E:fdd-eps-renorm} as
  \begin{align*}
    \beta\int_0^T\dot{W}^\epsilon(u,X_u)\ud u-\frac{\beta^2}{2}f_\epsilon(0)T
    =\sum_{i=0}^k\beta\int_{t_i}^{t_{i+1}}\dot{W}^\epsilon(u,X_u)\ud u-\frac{\beta^2}{2}f_\epsilon(0)(t_{i+1}-t_i),
  \end{align*}
  where we extend the definition of the $t_i$'s to $t_0 = 0, t_{k+1} = T$ and
  $x_0 = 0$. Plugging this identity into~\eqref{E:fdd-eps-renorm}, we have
  \begin{multline}\label{E:fdd-eps-split}
    \bE_T^{\beta,W^\epsilon} \left[ \phi\left(X_{t_1},\ldots ,X_{t_k}\right)\right]
    =  \frac{1}{\cz^\epsilon(0,0;T,*;\beta)}  \bE_0 \left[\prod_{i = 0}^k \ce^{\epsilon}_{t_i, t_{i + 1}} (W, X) \phi\left(X_{t_1},\ldots ,X_{t_k}\right)\right].
  \end{multline}
  Now consider a functional $\psi$ of the path $X_{[t_i,t_{i+1}]}$ on a generic
  interval $[t_i,t_{i+1}].$ Whenever $X_{t_i}=x_i$ we have
  \begin{align}\label{E:bridge-decond}
    \bE_{x_i}\left[\psi(X_{[t_i,t_{i+1}]})\right]
     & =\int_{\mathbb{R}^d}
    \bE_{x_i}\!\left[\psi(X_{[t_i,t_{i+1}]})| X_{t_{i+1}}=x_{i+1}\right]
    p_{t_{i+1}-t_i}(x_{i+1}-x_i)\ud x_{i+1} \notag\\
     & =\int_{\mathbb{R}^d}\bE_{t_i,x_i;t_{i+1},x_{i+1}}\!\left[\psi(X_{[t_i,t_{i+1}]})\right]
    p_{t_{i+1}-t_i}(x_{i+1}-x_i)\ud x_{i+1}.
  \end{align}
  Hence, conditioning the right-hand side of~\eqref{E:fdd-eps-split}
  successively on $X_{t_1},\dots,X_{t_k},X_T$ (with $W^\epsilon$ fixed) and
  invoking~\eqref{E:bridge-decond} for each conditioning, we get
  \begin{multline}\label{E:proof-same-poly}
    \begin{aligned}
      \bE_T^{\beta,W^\epsilon} \left[ \phi\left(X_{t_1},\ldots ,X_{t_k}\right)\right]
      = & \frac{1}{\cz^\epsilon(0,0;T,*;\beta)}\int_{\mathbb{R}^{d\times(k+1)}} \ud x_1\cdots \ud x_k \ud x_{k+1} \, \phi(x_1,\ldots ,x_k) \\
        & \times \prod_{i=0}^k \bE_{t_i,x_i;t_{i+1},x_{i+1}} \left[ \ce^{\epsilon}_{t_i, t_{i + 1}} (W, X)\right]  p_{t_{i+1}-t_i}(x_{i+1}-x_i),
    \end{aligned}
  \end{multline}
  where we use the convention $t_{k+1}=T$ and interpret $x_{k+1}$ as the
  integration variable corresponding to the endpoint $X_T$.
  Injecting~\eqref{E:FK_P-P} into~\eqref{E:proof-same-poly}, we end up with
  \begin{align}\label{E:equiv-finite-dist}
    \begin{aligned}
        & \bE_T^{\beta,W^\epsilon} \left[\phi(X_{t_1},\ldots ,X_{t_k})\right]                                                                                                                      \\
      = & \frac{1}{\cz^\epsilon(0,0;T,*;\beta)}\int_{\mathbb{R}^{d\times(k+1)}}\ud x_1\cdots \ud x_k \ud x_{k+1}\: \phi(x_1,\ldots ,x_k)\prod_{i=0}^k \cz^\epsilon(t_i,x_i;t_{i+1},x_{i+1};\beta),
    \end{aligned}
  \end{align}
  and the integration in $x_{k+1}$ yields $\cz^\epsilon(t_k,x_k;T,*;\beta)$.
  Therefore,~\eqref{E:equiv-finite-dist} coincides with the natural extension
  of~\eqref{E:fdd-test} to the time interval $[0,T]$. This proves
  Proposition~\ref{P:same-def-polymer}.
\end{proof}

In Lemma~\ref{L:holder}, we have established some H\"{o}lder continuity
properties for the field $\cz$ corresponding to the singular noise $\dot{W}$.
For our future considerations, we need to extend those estimates to uniform
estimates for the family $\{\cz^{\epsilon}; \epsilon > 0\}$. Those results are
collected below: Lemma~\ref{L:holder-eps-sy} deals with the $(s,y)$-increments,
Lemma~\ref{L:holder-eps-tx} takes care of the $(t,x)$-increments, and
Corollary~\ref{C:holder-eps-4} gives a global estimate.

\begin{lemma}[Uniform $(s,y)$-increments]\label{L:holder-eps-sy}
  Assume that Hypothesis~\ref{H:stren} holds. Fix $T_0>0$. Then for all $p\ge
  1$ and $\alpha\in(0,\eta)$ there exists $C=C(\alpha,\beta,p,\eta,T_0)>0$ such
  that for all $\epsilon>0$, all $0\le s\le s'<t\le T_0$ and all $x, y, y' \in
  \R^d$,
  \begin{multline}\label{E:holder-eps-sy}
    \Norm{\cz^\epsilon (s',y'; \, t,x ; \, \beta) - \cz^\epsilon(s, y; \, t,x ; \, \beta)}_{2p} \\
    \le C \times \frac{p_{4(t - s')} \left(x - y\right) + p_{4(t - s')} \left(x - y'\right)}{(t-s')^{\alpha/2}} \times
    \square_\alpha \left(s'-s, y'-y\right).
  \end{multline}
\end{lemma}
\begin{proof}
  The proof is identical to that of Lemma~\ref{L:holder}, replacing $f$ by
  $f_\epsilon=p_{2\epsilon}*f$ (see Remark~\ref{R:f_eps}). At the heart of our
  uniform results lies the following fact: due to relation~\eqref{E:fhat_eps},
  the strengthened Dalang condition~\eqref{E:stren} is satisfied uniformly in
  $\epsilon$. That is $f_\epsilon$ verifies:
  \begin{align}\label{E:Upsilon-eps-bound}
    \Upsilon_{\eta}^{\epsilon}
    \coloneqq  \int_{\R^d}\frac{\widehat{f_\epsilon}(\ud\xi)}{(1+|\xi|^2)^{1-\eta}}
    =   \int_{\R^d}\frac{e^{-\epsilon|\xi|^2}\widehat{f}(\ud\xi)}{(1+|\xi|^2)^{1-\eta}}
    \le \int_{\R^d}\frac{\widehat{f}(\ud\xi)}{(1+|\xi|^2)^{1-\eta}}
    =   \Upsilon_{\eta}.
  \end{align}
  We omit the remaining details, since they are identical to those in
  Lemma~\ref{L:holder} after replacing $f$ by $f_\epsilon$.
\end{proof}

\begin{lemma}[Uniform $(t,x)$-increments]\label{L:holder-eps-tx}
  Assume that Hypothesis~\ref{H:stren} holds. Fix $T_0>0$. Then for all $p\ge
  1$ and $\alpha\in(0,\eta)$ there exists $C=C(\alpha,\beta,p,\eta,T_0)>0$ such
  that for all $\epsilon>0$, all $0\le s<t\le t'\le T_0$ and all $x, x', y \in
  \R^d$,
  \begin{multline}\label{E:holder-eps-tx}
    \Norm{\cz^\epsilon (s,y; \, t',x' ; \, \beta) - \cz^\epsilon(s, y; \, t,x ; \, \beta)}_{2p} \\
    \le C \times \frac{p_{4(t - s)} \left(x - y\right) + p_{4(t' - s)} \left(x' - y\right)}{(t-s)^{\alpha/2}} \times
    \square_\alpha \left(t'-t, x'-x\right).
  \end{multline}
\end{lemma}
\begin{proof}
  This follows by specializing the arguments
  of~\cite{chen.huang:19:comparison}*{Proof of Theorem~1.8} to the linear
  equation~\eqref{E:smoothed-equ} (with the delta initial condition $\delta_y$)
  and spatial covariance~$f_\epsilon$. The constants in that proof depend on the
  covariance only through Fourier-side integrals of the form
  $\int_{\R^d}\frac{\widehat{f_\epsilon}(\ud\xi)}{(1+|\xi|^2)^{1-\alpha}}$.
  Since $\alpha<\eta$, these integrals are bounded by $\Upsilon^\epsilon_\eta\le
  \Upsilon_\eta$ thanks to~\eqref{E:Upsilon-eps-bound}, and therefore the
  constant $C$ in~\eqref{E:holder-eps-tx} is uniform in~$\epsilon$.
\end{proof}

We have now all the ingredients to state the following corollary, which won't be
used in this paper and is included for the sake of completeness and future
reference.

\begin{corollary}[Uniform four-parameter increments]\label{C:holder-eps-4}
  Under the assumptions of both Lemma~\ref{L:holder-eps-sy}
  and~\ref{L:holder-eps-tx}, for all $p\ge1$ and $\alpha\in(0,\eta)$ there
  exists $C=C(\alpha,\beta,p,\eta,T_0)>0$ such that for all $\epsilon>0$, all $0
  \le s \le s'<t\le t'\le T_0$ and all $x, x', y, y' \in \R^d$,
  \begin{multline}\label{E:holder-eps-4}
    \Norm{\cz^\epsilon (s',y'; \, t',x' ; \, \beta) - \cz^\epsilon(s, y; \, t,x ; \, \beta)}_{2p} \\
    \le C \times \frac{p_{4(t - s)} \left(x - y\right) + p_{4(t' - s)} \left(x' - y\right)}{(t-s)^{\alpha/2}} \times
    \square_\alpha \left(t'-t, x'-x\right)                                                        \\
    + C \times \frac{p_{4(t' - s')} \left(x' - y\right) + p_{4(t' - s')} \left(x' - y'\right)}{(t'-s')^{\alpha/2}} \times
    \square_\alpha \left(s'-s, y'-y\right).
  \end{multline}
\end{corollary}

\subsection{Diffusive behavior}

Let us go back to our original polymer measure $\bP_T^{\beta,W}$ introduced in
Definition~\ref{D:PolyMeasure}, but now considered on an arbitrary interval
$[0,T]$. The purpose of this section is to show that as $T$ tends to infinity,
in the so-called $L^2$-region (that is, when $\beta$ is small), $X_T$ presents
Gaussian behavior under the polymer measure $\bP_T^{\beta,W}$. Like for most
polymer models considered in the literature (see, e.g.,
\cite{comets:17:directed} for a review), note that the diffusive phenomenon only
occurs in dimension $d\ge3.$ This is related to some results stating boundedness
in $L^2(\Omega)$ for the solution of the stochastic heat equation. We quote this
result from \cite{chen.ouyang.ea:24:on}, although it can be traced back
to~\cite{chen.eisenberg:24:invariant}.

\begin{proposition}\label{P:L2-region}
  Let $W$ be a noise as in Definition~\ref{D:Noise}, defined on $\mathbb{R}^d$
  with $d \ge 3$. We assume that the constant $\Upsilon(0)$ defined
  by~\eqref{E:Dalang} is finite and the parameter $\beta$ is such that
  $\beta<\beta_0$, with
  \begin{align}\label{E:beta_0}
    \beta_0 = \frac{1}{2\Upsilon(0)^{1/2}}.
  \end{align}
  Consider the solution $u$ to equation~\eqref{E:SHE}, with initial condition
  $u(0,x) = 1$. Then the following holds true,
  \[
    \sup_{t\ge0}\sup_{x\in\mathbb{R}^d}\E\left[u(t,x)^2\right]<\infty.
  \]
  In addition, the second moment of $u(t,x)$ admits the Feynman-Kac
  representation
  \begin{align}\label{E:FK-moment}
    \E\left[u(t,x)^2\right] =
    \bE_0\left[\exp\left(\beta^2\int_0^tf(X_s-\widetilde{X}_s)\ud s\right)\right],
  \end{align}
  where $X$ and $\widetilde{X}$ are two independent Brownian motions under
  $\bP_0$ and $f$ is the covariance function of Definition~\ref{D:Noise}. For
  later use, define for $t\ge 0$
  \begin{align}\label{E:def_Lt_L8}
    L_t      \coloneqq \int_0^t                           f(X_s-\widetilde{X}_s) \ud s \quad \text{and} \quad
    L_\infty \coloneqq \lim_{t\to\infty}L_t=\int_0^\infty f(X_s-\widetilde{X}_s) \ud s.
  \end{align}
  In particular, under hypothesis~\eqref{E:beta_0}, we have
  \begin{align}\label{E:Finite_Mom_all_t}
    \bE_0\left[e^{\beta^2L_\infty}\right] < \infty, \quad \text{for all $\beta < \beta_0$.}
  \end{align}
\end{proposition}

\begin{remark}
  The conditions of Proposition~\ref{P:L2-region} are the same as those used
  in~\cite{chen.ouyang.ea:24:on} in order to get ergodicity of the parabolic
  Anderson model. This deep link might translate into ergodic properties of the
  polymer seen as solution to a stochastic differential equation.
\end{remark}

The main result of this section states that whenever the $L^2$-moments of the
stochastic heat equation are bounded, the polymer is diffusive. We summarize our
finding in the next theorem.

\begin{theorem}\label{T:diffusive-behavior}
  Let $d \ge 3$ and consider the polymer measure $\bP^{\beta,W}_T$ given in
  Definition~\ref{D:PolyMeasure}. As in Proposition~\ref{P:L2-region}, we assume
  Hypothesis~\ref{H:stren} and that the constant $\Upsilon(0)$
  in~\eqref{E:Dalang} is finite and that $\beta < \beta_{0}$ where $\beta_{0}$
  is defined by~\eqref{E:beta_0}. Then:
  \begin{enumerate}[label=\emph{\textbf{(\roman*)}}]

    \item \label{I:diffusive:i} The process $(\cz_t)_{t\ge 0}$ from~\eqref{E:Z_star} is a nonnegative martingale and satisfies $\E[\cz_t]=1$
      for every $t\ge 0$. Moreover, it converges in $L^2(\Omega)$ to a limit
      $\cz_\infty$ with $\E[\cz_\infty]=1$ and $\cz_\infty>0$ $\PP$-almost
      surely.

    \item \label{I:diffusive:ii} For every bounded continuous function $g$, we
      have
      \begin{align}\label{E:diff-behav}
        \bE_T^{\beta,W} \left[g\left(\frac{X_T}{\sqrt{T}}\right)\right] \to \nu(g)
        \quad \text{in $\PP$-probability as $T\to\infty$,}
      \end{align}
      where $\nu$ is the law of a standard Gaussian random variable.

  \end{enumerate}
\end{theorem}

\begin{remark}
  Theorem~\ref{T:diffusive-behavior}-(ii) takes the form of a central limit theorem type
  result for the process $X$ under the (quenched) polymer measure. However, from
  the perspective of the $\mathcal{F}$-measurable random variable
  \begin{align*} 
    Q_T (g) \coloneqq  \bE_T^{\beta,W} \left[g\left(\frac{X_T}{\sqrt{T}}\right)\right],
  \end{align*}
  relation~\eqref{E:diff-behav} should be seen as a law of large numbers. In the
  context of the stochastic heat equation, this type of limit has been analyzed,
  e.g., in \cites{gu.ryzhik.ea:18:edwards-wilkinson, gu.li:20:fluctuations,
  gerolla.hairer.ea:25:fluctuations}. If we want to study $Q_T (g)$ along the
  same lines as in \cite{gu.li:20:fluctuations}, let us write
  \begin{align}\label{E:rmk-diff-behav-b}
    Q_T (g) = & \frac{1}{\cz(0,0; \, T, *;\, \beta)} \int_{\R^d}  \cz(0,0; \, T, x;\, \beta) g \left(\frac{x}{{}\sqrt{T}}\right)  \ud x \nonumber \\
       =      & \frac{T^{\frac{d}{2}}}{\cz(0,0; \, T, *;\, \beta)}  \int_{\R^d}  \cz(0,0; \, T, \sqrt{T} y;\, \beta)  g (y)\ud y .
  \end{align}
  Next we resort to the chaos expansion~\eqref{E:Chaos-Exp-Z0} of $\cz(0,0; \,
  T, \sqrt{T} y;\, \beta)$, written as
  \begin{align*}
    \cz(0,0; \, T, \sqrt{T} y;\, \beta) = p_{T} \left(\sqrt{T} y \right) + H_1 (T, y),
  \end{align*}
  with
  \begin{align*}
     H_1 (T, y) \coloneqq \sum_{k=1}^{\infty} \beta^k I_k \left(g_k \left(\cdot \, ; 0,0,T,\sqrt{T} y \right) \right) .
  \end{align*}
  In the same way, invoking a slight elaboration of~\eqref{E:chaos-Atilde}, we
  get
  \[
    \cz(0,0; \, T, *;\, \beta) = 1 + H_2 (T),
  \]
  where $H_2 (T)$ is the remainder of order $1$ in the chaos expansion of
  $\cz(0,0; \, T, *;\, \beta)$. Hence relation~\eqref{E:rmk-diff-behav-b} can be
  written as
       \begin{align*}
      Q_T (g) = & \frac{T^{\frac{d}{2}}}{1 + H_2 (T)}  \int_{\R^d}  \left(p_{T} (\sqrt{T} y) + H_1 (T, y)\right)  g (y)\ud y .
    \end{align*}
  Using the scaling property of the heat kernel, this further reduces to
  \begin{align*}
    \bE_T^{\beta,W} \left[g\left(\frac{X_T}{\sqrt{T}}\right)\right]
       = & \frac{\nu (g)}{1 + H_2 (T)} + \frac{T^{d/2}}{1 + H_2(T)} \int_{\R^d}    H_1 (T, y)   g (y)\ud y .
  \end{align*}
  Theorem~\ref{T:diffusive-behavior}-(ii) is then intuitively justified, provided that
  the random contributions $H_2 (T)$ and $T^{d/2} H_1(T, \cdot)$ vanish (in
  appropriate spaces), as $T \to \infty$. We thus get that the polymer behaves
  like its expected diffusive limit.

  It is natural to ask whether a central limit theorem also holds for $Q_T (g)$
  as $T\to\infty$. Indeed, this is true if one replaces the partition function
  $\cz$ by $u$, the solution to the stochastic heat equation~\eqref{E:SHE}, with
  continuous bounded initial condition. In this setting, the first-order
  fluctuations (without the renormalizing constant $\cz(0,0; \, T, *;\, \beta)$)
  converge to an Edwards--Wilkinson limit
  \cite{gu.ryzhik.ea:18:edwards-wilkinson}.  Subsequently, it was shown that
  under flat initial conditions, the central limit theorem for the stochastic
  heat equation~\eqref{E:SHE} can be further extended to equations with
  multiplicative noise whose diffusion coefficient is Lipschitz with
  sufficiently small Lipschitz constant. This includes short-range noise
  \cite{gu.li:20:fluctuations}, long-range noise under a different scaling
  regime \cite{gerolla.hairer.ea:25:fluctuations}, as well as space-time white
  noise in the critical dimension $d = 2$ (with appropriate smoothing)
  \cites{tao:24:gaussian, dunlap.graham:25:edwards-wilkinson}.

  All of the aforementioned works assume bounded or flat initial conditions. It
  remains an open question whether a similar central limit theorem holds for
  narrow wedge initial conditions. Related problems for the spatial average of
  parabolic Anderson model with narrow wedge initial conditions have been
  investigated in \cites{chen.khoshnevisan.ea:22:spatial,
  khoshnevisan.nualart.ea:21:spatial}.
\end{remark}

The remainder of this section is devoted to the proof of
Theorem~\ref{T:diffusive-behavior}. As a preliminary step, we state a lemma giving a
representation for certain averages along the polymer measure.

\begin{lemma}\label{L:F-K_Mom_Polymer}
  Assume all conditions in Theorem~\ref{T:diffusive-behavior}. For $\beta > 0$, let
  $\bP_T^{\beta,W}$ be the polymer measure given in
  Definition~\ref{D:PolyMeasure}. Recall the notation $\cz_t$
  from~\eqref{E:Z_star}. Let $F$ be a continuous and bounded function in
  $C_b(\mathbb{R}^d)$. Then we have the following identity for the average $\E$
  over the noise $W$:
  \begin{align}\label{E:FK_Mom}
    \E\left[\left(\bE_T^{\beta,W} \left[F(X_T)\right]\cz_T\right)^2\right]
    =  \bE_0\left[e^{\beta^2 L_T} F(X_T)F(\widetilde{X}_T)\right],
  \end{align}
  where in the right-hand side of~\eqref{E:FK_Mom}, the processes $X$ and
  $\widetilde{X}$ stand for two independent copies of the Brownian motion $X$
  under $\bP_{0}$, and we recall that $L_T$ is defined by~\eqref{E:def_Lt_L8}.
\end{lemma}

\begin{proof}
  Let $f_\epsilon(x)$ be a smoothed version of the covariance function $f(x)$.
  The corresponding noise and polymer expectation are denoted by $W^\epsilon$
  and $\bE_T^{\beta,W^\epsilon}$, respectively. Recalling our
  notation~\eqref{E:Z_star}, we also define a smoothed version of the random
  variable $\cz_t$:
  \begin{align}\label{E:Zeps_T}
    \cz_T^\epsilon \coloneqq \cz^\epsilon(0,0;T,*;\beta)
    = Z^\epsilon_T\exp\left(-\frac{\beta^2}{2}f_{\epsilon}(0)T\right),
  \end{align}
  where we recall that $\cz^\epsilon(0,0;T,*;\beta)$ and $Z^{\epsilon}_T$ are
  introduced in~\eqref{E:Zeps_star_FK} and~\eqref{E:Zeps_t}, respectively.
  According to~\eqref{E:fdd-eps}, for our generic function $F\in
  C_b(\mathbb{R}^d)$ we have
  \[
    \bE_T^{\beta,W^\epsilon}\left[F(X_T)\right]
    =\frac{1}{Z^\epsilon_T}\bE_0\left[e^{\beta\int_0^T\dot{W}^\epsilon(u,X_u)\ud u}F(X_T)\right].
  \]
  Hence, invoking~\eqref{E:Zeps_T} we get
  \begin{align}\label{E:FK_moment_polymer_mid1}
    \bE_T^{\beta,W^\epsilon}\left[F(X_T)\right]\cz_T^\epsilon
    = \bE_0\left[e^{\beta\int_0^T\dot{W}^\epsilon(u,X_u)\ud u}F(X_T)\right]e^{-\frac{\beta^2}{2}f_\epsilon(0)T}.
  \end{align}
  In addition, one can tensorize relation~\eqref{E:FK_moment_polymer_mid1} in
  order to get
  \begin{align}\label{E:FK_moment_polymer_mid2}
       \left( \bE_T^{\beta,W^\epsilon} \left[F(X_T)\right]\cz_T^\epsilon\right)^2
       = \bE_0\left[e^{\beta\int_0^T\dot{W}^\epsilon(u,X_u)\ud u+\beta\int_0^T\dot{W}^\epsilon(u,\widetilde{X}_u)\ud u}F(X_T)F(\widetilde{X}_T)\right] e^{-\beta^2f_\epsilon(0)T}.
  \end{align}
  At this point we take expected values with respect to $W^\epsilon$ (recall
  that this expectation is written $\E$), we apply Fubini in order to evaluate
  Gaussian quantities and we use the fact that for a centered Gaussian random
  variable $G$ we have
  \[
    \E\left[\exp(G)\right]
    =\exp\left(\frac{1}{2}\E[G^2]\right).
  \]
  Performing those elementary operations on the right-hand side
  of~\eqref{E:FK_moment_polymer_mid2} we end up with
  \begin{align}\label{E:FK_Mom_eps}
    \E\left[\left(\bE_T^{\beta,W^\epsilon} \left[F(X_T)\right]\cz_T^\epsilon\right)^2\right]
    = \bE_0\left[e^{\beta^2\int_0^Tf_\epsilon(X_s-\widetilde{X}_s) \ud s}F(X_T)F(\widetilde{X}_T)\right].
  \end{align}
  To conclude the proof, it suffices to show that both sides
  of~\eqref{E:FK_Mom_eps} converge to the corresponding limits when sending
  $\epsilon \downarrow 0$. To this end, we divide the proof into several steps.
  \medskip

  \noindent\textit{Step~1.} We first show that there exists a subsequence
  $\epsilon_n\downarrow0$ such that, for every bounded measurable function
  $\phi$,
  \begin{align}\label{E:polymer_eps_weak_conv}
    \lim_{n\to\infty} \bE_T^{\beta,W^{\epsilon_n}} \left[\phi(X_T)\right]
    = \bE_T^{\beta,W} \left[\phi(X_T)\right], \quad \text{$\PP$-almost surely.}
  \end{align}
  Indeed, applying definition~\eqref{E:fdd-test} we have
  \begin{align}\label{E:F-K_Mom_Polymer-e}
    \bE_T^{\beta,W^{\epsilon_n}} \left[\phi(X_T)\right]
    = \int_{\mathbb{R}^d}\frac{\cz^{\epsilon_n} (0,0;T,x;\beta)}{\cz_T^{\epsilon_n}}\phi(x) \ud x,
  \end{align}
  where we have set (for notational sake) $\cz_T^\epsilon \coloneqq
  \cz^{\epsilon} (0,0;T,*;\beta)$. By Lemma~\ref{L:holder-eps-tx} (applied with
  $s = 0$, $y = 0$ and $t = t' = T$), for each $R > 0$ one obtains a uniform (in
  $\epsilon$) moment bound on the spatial increments of $x\mapsto
  \cz^\epsilon(0,0;T,x;\beta)$ over $B(0,R)$, of the form required in
  Lemma~\ref{L:as-compact}. Together with the pointwise convergence
  $\cz^\epsilon(0,0;T,x;\beta)\to\cz(0,0;T,x;\beta)$ in probability for each
  fixed $x$ (e.g., by the $L^2(\PP)$ convergence established
  in~\cite{chen.huang:19:comparison}*{Theorem~1.9}), Lemma~\ref{L:as-compact}
  yields a
  subsequence $\epsilon_n \downarrow 0$ such that for every $R>0$,
  \[
    \sup_{x\in B(0,R)}
    \left|\cz^{\epsilon_n}(0,0;T,x;\beta)-\cz(0,0;T,x;\beta)\right|
    \to 0,
    \qquad \PP\text{-almost surely.}
  \]
  Passing to a further subsequence if needed (still denoted $(\epsilon_n)$), we
  also have $\cz_T^{\epsilon_n} \to \cz_T$ almost surely by the $L^2(\PP)$
  convergence, and since $\cz_T > 0$ almost surely, we obtain that the densities
  $l_T^{\epsilon_n}(x) \coloneqq \cz^{\epsilon_n}(0,0;T,x)/\cz_T^{\epsilon_n}$
  converge to $l_T(x)\coloneqq \cz(0,0;T,x)/\cz_T$ pointwise for all
  $x\in\mathbb{R}^d$ (after a diagonal extraction over $R\in\mathbb{N}$).
  Therefore, along this subsequence we have, as probability densities,
  $\PP$-almost surely,
  \[
    l_T^{\epsilon_n}(x)        \coloneqq \frac{\cz^{\epsilon_n}(0,0;T,x)}{\cz_T^{\epsilon_n}}\longrightarrow
    \frac{\cz(0,0;T,x)}{\cz_T} \eqqcolon l_T(x),\quad \mathrm{for\ all\ } x\in\mathbb{R}^d.
  \]
  Here $\cz_T$ is defined in~\eqref{E:Z_star}. We have thus established that
  $\PP$-a.s. $\lim_{n\to\infty}l^{\epsilon_n}_T(x) = l_T(x)$ for all $x \in
  \mathbb{R}^d,$ where $l^{\epsilon_n}_T$ is the density of $X_T$ under
  $\bP_T^{\beta,W^{\epsilon_n}}$ and $l_T$ is the density of $X_T$ under
  $\bP_T^{\beta, W}$. Thanks to Scheff\'{e}'s lemma (see, e.g.,
  \cite{resnick:92:adventures}*{Lemma 8.2.1}), we get that
  \[
    \cl\left( X_{T} \bigm| \, \bP_T^{\beta,W^{\epsilon_n}} \right) \stackrel{\mathrm{TV}}{\longrightarrow }
    \cl\left( X_{T} \bigm| \, \bP_T^{\beta,W}              \right),
  \]
  where TV stands for the total variation distance on probability measures.
  Therefore, the laws of $X_T$ under $\bP_T^{\beta,W^{\epsilon_n}}$ converge in
  total variation (and hence weakly) to the law of $X_T$ under
  $\bP_T^{\beta,W}$. This finishes the proof of~\eqref{E:polymer_eps_weak_conv}.
  \medskip

  \noindent\textit{Step~2.} Fix the subsequence $(\epsilon_n)_{n\ge1}$ from
  Step~1. Our next objective is to prove
  \begin{align}\label{E:polymer_eps_L2_conv}
    \E\left[\left(\bE_T^{\beta,W^{\epsilon_n}} \left[F(X_T)\right]\cz_T^{\epsilon_n}\right)^2\right]
    \to \E\left[\left(\bE_T^{\beta,W} \left[F(X_T)\right]\cz_T\right)^2\right], \quad \text{as $n\to\infty$.}
  \end{align}
  In other words, we will prove that the left-hand side of~\eqref{E:FK_Mom_eps}
  converges along the subsequence $(\epsilon_n)$. To this aim, we first resort
  to~\eqref{E:polymer_eps_weak_conv} with $\phi = F$. This yields as
  $n\to\infty$,
  \begin{align}\label{E:polymer_eps_conv_F}
    \left(\bE_T^{\beta,W^{\epsilon_n}} \left[F(X_T)\right]\right)^2
    \stackrel{\PP-\mathrm{a.s.}}{\longrightarrow} \left(\bE_T^{\beta,W} \left[F(X_T)\right]\right)^2.
  \end{align}
  Moreover, since we assume $F\in C_b(\mathbb{R}^d)$, the left hand side
  of~\eqref{E:polymer_eps_conv_F} is uniformly bounded. In addition, thanks
  to~\cite{chen.huang:19:comparison}*{Theorem~1.9} (applied with constant-one
  initial condition), we have
  \begin{align}\label{E:polymer_eps_conv_Z}
    \cz_T^\epsilon
    = \cz^\epsilon(0,0;T,*;\beta)\stackrel{L^2(\PP)}{\longrightarrow}\cz(0,0;T,*;\beta)
    =\cz_T.
  \end{align}
  Set $A_n \coloneqq \big(\bE_T^{\beta,W^{\epsilon_n}}[F(X_T)]\big)^2$ and $A
  \coloneqq \big(\bE_T^{\beta,W}[F(X_T)]\big)^2$. Then $A_n\to A$ $\PP$-a.s.
  by~\eqref{E:polymer_eps_conv_F} and $\sup_n|A_n|\le \Norm{F}_\infty^2$. Next
  set $B_n \coloneqq (\cz_T^{\epsilon_n})^2$ and $B\coloneqq (\cz_T)^2$. By the
  Cauchy--Schwarz inequality, we have
  \[
    \E\left[\left|B_n-B\right|\right]
    = \E\left[\left|\cz_T^{\epsilon_n}-\cz_T\right|\left|\cz_T^{\epsilon_n}+\cz_T\right|\right]
    \le \Norm{\cz_T^{\epsilon_n}-\cz_T}_{L^2(\PP)}\,\Norm{\cz_T^{\epsilon_n}+\cz_T}_{L^2(\PP)}
    \xrightarrow[n \uparrow \infty]{}0,
  \]
  where the first factor tends to $0$ by~\eqref{E:polymer_eps_conv_Z}. The
  second factor is bounded by
  $2\sup_n \Norm{\cz_T^{\epsilon_n}}_{L^2(\PP)} + 2\Norm{\cz_T}_{L^2(\PP)} <
  \infty$. Therefore, $B_n\to B$ in $L^1(\PP)$, and
  \[
    \left|\E\left[A_nB_n\right]-\E\left[AB\right]\right|
    \le \E\left[\left|A_n\right|\left|B_n-B\right|\right]
    + \left|\E\left[(A_n-A)B\right]\right|
    \xrightarrow[n \uparrow \infty]{}0.
  \]
  Indeed, the first term vanishes since $|A_n|\le \Norm{F}_\infty^2$ and $B_n\to
  B$ in $L^1(\PP)$, while the second term vanishes by dominated convergence
  because $B\in L^1(\PP)$ and $|A_n-A|\le 2\Norm{F}_\infty^2$. Therefore, we
  have proved our claim~\eqref{E:polymer_eps_L2_conv}. \medskip

  \noindent\textit{Step~3.} We now examine the right-hand side
  of~\eqref{E:FK_Mom_eps}. First we observe that the product
  $F(X_T)F(\widetilde{X}_T)$ is bounded if $F\in C_b(\mathbb{R}^d)$, where we
  recall that $X$ and $\widetilde{X}$ stand for two independent Brownian motion
  under the measure $\bP_0$. Next, set $Y_s \coloneqq X_s-\widetilde{X}_s$.
  Recall the spectral measure $\widehat{f}(\ud\xi)$ from
  Definition~\ref{D:Noise}. By~\eqref{E:fhat_eps} and the identity
  $\bE_0\left[e^{i\xi\cdot Y_s}\right] = e^{-s|\xi|^2}$, we have for $s>0$,
  \begin{equation}\label{E:f_eps_bd}
    \bE_0\left[f_\epsilon(Y_s)\right]
    = (2\pi)^{-d}\int_{\R^d} e^{-s|\xi|^2}e^{-\epsilon|\xi|^2}\widehat{f}(\ud\xi)
    \le (2\pi)^{-d}\int_{\R^d} e^{-s|\xi|^2}\widehat{f}(\ud\xi)
    = \bE_0\left[f(Y_s)\right].
  \end{equation}
  In particular,
  \begin{equation}\label{E:int_fY_s}
    \int_0^T \bE_0\left[f(Y_s)\right]\ud s
    = (2\pi)^{-d}\int_{\R^d}\frac{1-e^{-T|\xi|^2}}{|\xi|^2}\,\widehat{f}(\ud\xi)
    \le 2\max(T,1)\int_{\R^d}\frac{\widehat{f}(\ud\xi)}{1+|\xi|^2}
    < \infty,
  \end{equation}
  by Dalang's condition~\eqref{E:Dalang}. Next we claim that
  \begin{align}\label{E:F-K_Mom_Polymer-d}
    \int_0^T \left|f_\epsilon(Y_s)-f(Y_s)\right|\ud s
    \xrightarrow[\epsilon\downarrow0]{L^1(\bP_0)} 0.
  \end{align}
  Indeed, for $R > 0$ set $f^{(R)} \coloneqq f \1_{B(0,R)}$ and $f^{(R)}_c
  \coloneqq f-f^{(R)}$. Since $f^{(R)}\in L^1(\R^d)$ and $f_\epsilon =
  p_{2\epsilon}*f$, for fixed $s>0$ we have
  \begin{align*}
    \MoveEqLeft \bE_0\left[\left|f_\epsilon(Y_s)-f(Y_s)\right|\right]                          \\
     & \le \bE_0\left[\left|\left(p_{2\epsilon}*f^{(R)}\right)(Y_s)-f^{(R)}(Y_s)\right|\right]
    + \bE_0\left[\left(p_{2\epsilon}*f^{(R)}_c\right)(Y_s)\right]
    + \bE_0\left[f^{(R)}_c(Y_s)\right]                                                         \\
     & \le \Norm{p_{2s}}_{\infty}\Norm{p_{2\epsilon}*f^{(R)}-f^{(R)}}_{L^1(\R^d)}
    + \bE_0\left[f^{(R)}_c(Y_{s+\epsilon})\right]
    + \bE_0\left[f^{(R)}_c(Y_s)\right],
  \end{align*}
  where in the last line we used that $Y_s$ has density $p_{2s}$ and the
  identity
  $\bE_0\left[(p_{2\epsilon}*g)(Y_s)\right]=\bE_0\left[g(Y_{s+\epsilon})\right]$
  for $g\ge 0$. Letting $\epsilon\downarrow0$, the first term vanishes since
  $p_{2\epsilon}$ is an approximate identity in $L^1(\R^d)$, and the second
  term converges to $\bE_0[f^{(R)}_c(Y_s)]$ by dominated convergence. Then
  letting $R\to\infty$ yields
  $\bE_0\left[\left|f_\epsilon(Y_s)-f(Y_s)\right|\right] \to 0$ for every $s
  \in (0,T]$. In addition, by the bound in~\eqref{E:f_eps_bd}, we have
  $\bE_0\left[\left|f_\epsilon(Y_s)-f(Y_s)\right|\right]\le
  2\bE_0\left[f(Y_s)\right]$, and thanks to~\eqref{E:int_fY_s}, another
  application of dominated convergence yields the $L^1(\bP_0)$ convergence
  in~\eqref{E:F-K_Mom_Polymer-d}. Consequently,
  \[
    \int_0^T f_\epsilon(Y_s)\ud s \xrightarrow[\epsilon\downarrow0]{\bP_0\text{-prob.}} \int_0^T f(Y_s)\ud s,
  \]
  and by continuity of the exponential function, we obtain the following
  $L^0(\bP_0)$ convergence:
  \begin{equation}\label{E:exp_fY_Pr}
    \exp\left(\beta^2\int_0^T f_\epsilon(Y_s)\ud s\right)
    \xrightarrow[\epsilon\downarrow0]{\bP_0\text{-prob.}}
    \exp\left(\beta^2\int_0^T f(Y_s)\ud s\right).
  \end{equation}
  Moreover, under the assumption $\beta < \beta_0$ of
  Proposition~\ref{P:L2-region}, the family $\exp(\beta^2\int_0^T
  f_\epsilon(Y_s)\ud s)$, $\epsilon>0$, is uniformly integrable. Indeed, choose
  $p>1$ such that $\sqrt{p}\,\beta < \beta_0$. Applying the estimate in
  Proposition~\ref{P:L2-region} to the mollified model (with covariance
  $f_\epsilon$ and parameter $\sqrt{p}\,\beta$), and using~\eqref{E:fhat_eps}
  so that $\Upsilon_\epsilon(0)\le \Upsilon(0)$, yields
  \[
    \sup_{\epsilon>0}\bE_0\left[\exp\left(p\beta^2\int_0^T f_\epsilon(Y_s)\ud s\right)\right]<\infty.
  \]
  In particular, $\sup_{\epsilon>0}\|\exp(\beta^2\int_0^T f_\epsilon(Y_s)\ud
  s)\|_{L^p(\bP_0)}<\infty$, and hence $\{\exp(\beta^2\int_0^T
  f_\epsilon(Y_s)\ud s)\}_{\epsilon>0}$ is uniformly integrable. Therefore, the
  $L^0(\bP_0)$ convergence in~\eqref{E:exp_fY_Pr} improves to $L^1(\bP_0)$
  convergence. Combining this with the boundedness of
  $F(X_T)F(\widetilde{X}_T)$, we thus get the following asymptotic behavior for
  the right-hand side of~\eqref{E:FK_Mom_eps}:
  \begin{align}\label{E:polymer_eps_conv_kernel}
    \lim_{\epsilon\downarrow0}\bE_0\left[e^{\beta^2\int_0^Tf_\epsilon(X_s-\widetilde{X}_s)\ud s}F(X_T)F(\widetilde{X}_T)\right]
    = \bE_0\left[e^{\beta^2\int_0^Tf(X_s-\widetilde{X}_s)\ud s}F(X_T)F(\widetilde{X}_T)\right].
  \end{align}
  Taking $\epsilon=\epsilon_n$ in~\eqref{E:FK_Mom_eps} and letting
  $n\to\infty$, we combine~\eqref{E:polymer_eps_L2_conv}
  with~\eqref{E:polymer_eps_conv_kernel} to obtain~\eqref{E:FK_Mom}. This
  completes the proof of Lemma~\ref{L:F-K_Mom_Polymer}.
\end{proof}

\begin{remark}\label{R:holder-relevance}
  The H\"{o}lder continuity of $\cz$ or $\cz^{\epsilon}$ as a random field looks
  a priori unrelated to basic properties of the polymer. It is remarkable that
  it plays a crucial role in Lemma~\ref{L:F-K_Mom_Polymer}, specifically in the
  proof of relation~\eqref{E:F-K_Mom_Polymer-e}.
\end{remark}

We also need the following lemma.

\begin{lemma}\label{L:triple-ind}
  Let $X$ and $\widetilde{X}$ be two independent Brownian motions under the
  Wiener measure $\bP_0$. Recall that the covariance function $f$ is introduced
  in Definition~\ref{D:Noise}. Assume that $\beta<\beta_0$ where $\beta_0$ is
  defined in~\eqref{E:beta_0}. Then as $T \to \infty$, we have
  \begin{align}\label{E:convg-triple-id}
    \left(L_T,\; \frac{X_T}{\sqrt{T}},\; \frac{\widetilde{X}_T}{\sqrt{T}}\right)
    \longrightarrow\  (L_\infty, Z, \widetilde{Z}), \quad \text{in distribution,}
  \end{align}
  where $L_t$ and $L_\infty$ are defined in~\eqref{E:def_Lt_L8}, $Z$ and
  $\widetilde{Z}$ are two standard Gaussian random vectors in $\R^d$, and in
  addition, the random elements $L_\infty$, $Z$, and $\widetilde{Z}$ are
  independent.
\end{lemma}
\begin{proof}
  Recall the notation $L_t$ and $L_\infty$ from~\eqref{E:def_Lt_L8}. Under
  hypothesis~\eqref{E:beta_0}, we have $L_\infty<\infty$ $\bP_0$-a.s. thanks
  to~\eqref{E:Finite_Mom_all_t}, and it has exponential moments. Let $\phi, g,
  \widetilde{g}$ be bounded continuous functions. For $T>0$, consider
  \[
    Q_T \coloneqq \bE_0 \left[\phi(L_T) g\left(\frac{X_T}{\sqrt{T}}\right) \widetilde{g}\left(\frac{\widetilde{X}_T}{\sqrt{T}}\right)\right].
  \]
  Fix $t > 0$ and assume $T > t$. In order to decouple the endpoint fluctuations
  from the initial segment $[0,t]$, we write
  \[
    Q_T=Q^0_T+I_1+I_2+I_3,
  \]
  where we set
  \begin{align}
    Q^0_T
    \coloneqq
    \bE_0 \left[\phi(L_t) g\left(\frac{X_T-X_t}{\sqrt{T}}\right)
      \widetilde{g}\left(\frac{\widetilde{X}_T-\widetilde{X}_t}{\sqrt{T}}\right)\right]
    \label{E:Q0_T},
  \end{align}
  and where the terms $I_1, I_2, I_3$ are respectively defined by
  \begin{align*}
     & I_1 = \bE_0\left[\left(\phi(L_T)-\phi(L_t)\right) g\left(\frac{X_T}{\sqrt{T}}\right)
    \widetilde{g}\left(\frac{\widetilde{X}_T}{\sqrt{T}}\right)\right],                                                            \\
     & I_2 = \bE_0\left[\phi(L_t)\left(g\left(\frac{X_T}{\sqrt{T}}\right) - g\left(\frac{X_T-X_t}{\sqrt{T}}\right)\right)
    \widetilde{g}\left(\frac{\widetilde{X}_T}{\sqrt{T}}\right)\right],                                                            \\
     & I_3 = \bE_0\left[\phi(L_t) g\left(\frac{X_T-X_t}{\sqrt{T}}\right)
      \left(\widetilde{g}\left(\frac{\widetilde{X}_T}{\sqrt{T}}\right) - \widetilde{g}\left(\frac{\widetilde{X}_T-\widetilde{X}_t}{\sqrt{T}}\right)\right)\right].
  \end{align*}

  We first compute the limit of $Q^0_T$ as $T\to\infty$. By independent
  increments, the random variable $L_t$ (which is measurable with respect to the
  path on $[0,t]$) is independent of the increments $X_T-X_t$ and
  $\widetilde{X}_T-\widetilde{X}_t$. Moreover, $X_T-X_t$ and
  $\widetilde{X}_T-\widetilde{X}_t$ are independent. Therefore,
  \[
    Q^0_T
    = \bE_0\left[\phi(L_t)\right]
    \bE_0\left[g\left(\frac{X_T-X_t}{\sqrt{T}}\right)\right]
    \bE_0\left[\widetilde{g}\left(\frac{\widetilde{X}_T-\widetilde{X}_t}{\sqrt{T}}\right)\right].
  \]
  Next, by Brownian scaling, we have
  \begin{align}\label{E:CLT-increments}
    \frac{X_T-X_t}{\sqrt{T}} \longrightarrow Z \quad\text{and}\quad
    \frac{\widetilde{X}_T-\widetilde{X}_t}{\sqrt{T}} \longrightarrow \widetilde{Z},
    \quad\text{as $T \uparrow \infty$, in distribution,}
  \end{align}
  where $Z$ and $\widetilde{Z}$ are the standard Gaussian random vectors in the
  statement. Moreover, by definition~\eqref{E:def_Lt_L8}, we have $L_t \uparrow
  L_\infty$ $\bP_0$-almost surely as $t \uparrow \infty$. Combining this
  with~\eqref{E:CLT-increments} in~\eqref{E:Q0_T}, we thus obtain
  \begin{align}\label{E:limit-Q0_T}
    \lim_{t\to\infty}\lim_{T\to\infty}Q^0_T
    = \bE_0 \left[\phi\left(L_\infty\right)\right]\,
    \bE \left[g(Z)\right]\,
    \bE \left[\widetilde{g}(\widetilde{Z})\right].
  \end{align}

  We now claim that
  \begin{align}\label{E:Convg-I}
    \lim_{t\to\infty}\limsup_{T\to\infty}|I_\ell| = 0, \quad \text{for $ \ell = 1, 2, 3$.}
  \end{align}
  To justify~\eqref{E:Convg-I}, first, since $g$ and $\widetilde{g}$ are bounded and
  $L_T\uparrow L_\infty$ $\bP_0$-a.s. from~\eqref{E:def_Lt_L8},
  \[
    \limsup_{T\to\infty}|I_1|
    \le \Norm{g}_\infty\Norm{\widetilde{g}}_\infty
    \limsup_{T\to\infty}\bE_0\left[\left|\phi(L_T)-\phi(L_t)\right|\right]
    = \Norm{g}_\infty\Norm{\widetilde{g}}_\infty\,\bE_0\left[\left|\phi(L_\infty)-\phi(L_t)\right|\right],
  \]
  by dominated convergence in $T$. Letting $t\to\infty$ yields~\eqref{E:Convg-I}
  for $\ell = 1$. Next, $I_2$ and $I_3$ are treated in the same way, so we only
  need to bound $I_2$. Notice that
  \[
    |I_2|
    \le \Norm{\phi}_\infty\Norm{\widetilde{g}}_\infty\,
    \bE_0\left[\left|g\left(\frac{X_T}{\sqrt{T}}\right)-g\left(\frac{X_T-X_t}{\sqrt{T}}\right)\right|\right].
  \]
  Set $U_T \coloneqq (X_T-X_t)/\sqrt{T}$, so that
  $X_T/\sqrt{T}=U_T+X_t/\sqrt{T}$. Fix an arbitrary $\delta>0$. Set $R =
    \sqrt{\delta^{-1} 2 \Norm{g}_\infty d}$. By uniform continuity of $g$ on the
  compact set $\{x\in\R^d: |x| \le R+1 \}$, there exists $\eta \in (0,1]$ such
  that $|g(x+h)-g(x)|\le\delta$ whenever $|x|\le R$ and $|h|\le\eta$. Then by
  Chebyshev's inequality,
  \begin{align*}
    \bE_0\left[\left|g\left(U_T+\frac{X_t}{\sqrt{T}}\right)-g(U_T)\right|\right]
    \le & \,\delta + 2\Norm{g}_\infty\,\bP_0\left(|U_T|>R\right)
    + 2\Norm{g}_\infty\,\bP_0\left(\frac{|X_t|}{\sqrt{T}}>\eta\right)        \\
    \le & \,\delta + 2\Norm{g}_\infty\,\frac{\bE_0\left[|U_T|^2\right]}{R^2}
    + 2\Norm{g}_\infty\,\frac{\bE_0\left[|X_t|^2\right]}{\eta^2T}            \\
    =   & \,\delta + 2\Norm{g}_\infty\,\frac{d (1 - t/T)}{R^2}
    + 2\Norm{g}_\infty\,\frac{dt}{\eta^2T}                                   \\
    \le & \,2\delta + 2\Norm{g}_\infty\,\frac{dt}{\eta^2T}.
  \end{align*}
  Therefore, letting $T\to\infty$ yields
  \[
    \limsup_{T\to\infty}\bE_0\left[\left|g\left(X_T/\sqrt{T}\right)-g\left((X_T-X_t)/\sqrt{T}\right)\right|\right]
    \le 2 \delta.
  \]
  Since $\delta > 0$ is arbitrary, we obtain $\lim_{T\to\infty}|I_2| = 0$ for
  each fixed $t>0$, and the same argument gives $\lim_{T\to\infty}|I_3| = 0$.
  This proves~\eqref{E:Convg-I} for $\ell = 2, 3$. \medskip

  Finally, by the decomposition $Q_T=Q_T^0+I_1+I_2+I_3$, the limit of $Q_T$ is
  given by the limit of $Q_T^0$. Indeed, for each fixed $t>0$,
  $\limsup_{T\to\infty}\left|Q_T-Q_T^0\right| \le
    \sum_{\ell=1}^3\limsup_{T\to\infty}|I_\ell|$, and letting $t\to\infty$ yields
  $\lim_{t\to\infty}\limsup_{T\to\infty}|Q_T-Q_T^0|=0$ by~\eqref{E:Convg-I}.
  Together with~\eqref{E:limit-Q0_T}, we obtain
  \[
    \lim_{T\to\infty}Q_T
    = \bE_0 \left[\phi\left(L_\infty\right)\right]\,
    \bE \left[g(Z)\right]\,
    \bE \left[\widetilde{g}(\widetilde{Z})\right],
  \]
  which implies~\eqref{E:convg-triple-id}. This completes the proof of
  Lemma~\ref{L:triple-ind}.
\end{proof}

Now we are ready to prove Theorem~\ref{T:diffusive-behavior}.

\begin{proof}[Proof of Theorem~\ref{T:diffusive-behavior}]
  We prove item~\emph{\ref{I:diffusive:i}} and~\emph{\ref{I:diffusive:ii}} in three steps.

  \noindent\textit{Step~1: proof of~\ref{I:diffusive:i}.} In this step, we show
  that the process $\cz_t$ from~\eqref{E:Z_star} satisfies
  \begin{align}\label{E:Z_infty}
    \cz_\infty > 0 \quad \text{$\PP$-a.s,}\quad \text{where}\quad
    \cz_\infty \coloneqq \lim_{T\to\infty}\cz_T \quad \text{in $L^2(\Omega)$}.
  \end{align}
  Indeed, by integrating the mild formulation~\eqref{E:Mild-Z} (with $(s,y) =
  (0,0)$) in $x$, using stochastic Fubini's theorem and the identity
  $\int_{\R^d} p_{t-r}(x-z)\,\ud x = 1$, we obtain
  \[
    \cz_t
    = 1 + \beta \int_0^t \int_{\R^d} \cz(0,0;\,r,z;\beta)\, W(\ud r, \ud z).
  \]
  In particular, $(\cz_t)_{t\ge 0}$ is a nonnegative martingale with respect to
  the standard filtration $\{\mathcal{F}_t\}$ introduced in
  Remark~\ref{R:Filtration}. Moreover, it is bounded in $L^2$ since
  \begin{equation}\label{E:Z_t_L2_bound}
    \sup_{t\ge 0}\E\left[\cz_t^2\right]
    = \sup_{t\ge 0}\bE_0\left[e^{\beta^2 L_t}\right]
    \le \bE_0\left[e^{\beta^2L_\infty}\right] < \infty,
  \end{equation}
  where we have used~\eqref{E:FK_Mom} with $F\equiv 1$ for the first identity,
  and the monotonicity of $L_t$ plus~\eqref{E:Finite_Mom_all_t} for the rest.
  Hence, Doob's $L^2$ martingale convergence theorem (see,
  e.g.,~\cite{kallenberg:21:foundations}*{Corollary 9.23}) yields the existence
  of $\cz_\infty$ as an $L^2(\Omega)$ limit in~\eqref{E:Z_infty}. Since the
  $L^2$ bound in~\eqref{E:Z_t_L2_bound} implies uniform integrability, we also
  have $\E[\cz_\infty] = \lim_{T\to\infty}\E[\cz_T] = 1$.

  It remains to prove that $\cz_\infty > 0$ $\PP$-almost surely. To this end,
  we claim that
  \begin{equation}\label{E:tail-sigma-field}
    \left\{\cz_\infty = 0\right\} \in \mathcal{T}, \quad \text{where} \quad
    \mathcal{T} \coloneqq \bigcap_{n\ge 1}\sigma\!\left(W((n,\infty)\times A): A\in\mathcal{B}(\R^d)\right).
  \end{equation}
  In order to prove~\eqref{E:tail-sigma-field}, first fix $n\ge 1$. For
  $y\in\R^d$ set
  \[
    \cz_\infty^{(n)}(y)\coloneqq \lim_{m\to\infty}\cz(n,y;\,n+m,*;\beta), \quad \text{in $L^2(\Omega)$}.
  \]
  where the $L^2(\Omega)$-limit exists by the same $L^2$-boundedness argument
  as above. By stationarity in $(s,y)$ (see
  Theorem~\ref{T:transition-properties}-\ref{T:transition-properties:stationarity}), the process $m\mapsto
  \cz(n,y;\,n+m,*;\beta)$ has the same law as $m\mapsto
  \cz(0,0;\,m,*;\beta)\equiv \cz_m$ (see~\eqref{E:Z_star}), since shifting
  $(s,y)$ by $(n,-y)$ and then changing variables in the spatial integral
  defining $\cz(\cdot,*;\beta)$ removes the dependence on~$y$. Moreover,
  $\cz_\infty^{(n)}(y)$ is measurable with respect to
  $\sigma(W((n,\infty)\times A):A\in\mathcal{B}(\R^d))$. Since for each $m\ge
  1$ the map $(\omega,y)\mapsto\cz(n,y;\,n+m,*;\beta)$ is jointly measurable,
  we fix a version of $y \mapsto \cz_\infty^{(n)}(y)$ which is jointly
  measurable as well (so that the integrals below are well-defined). By the
  Chapman--Kolmogorov identity~\eqref{E:chapman-kolmogorov}, for every $m\ge 1$ we have
  \begin{equation}\label{E:Z_nm}
    \cz_{n+m}=\int_{\R^d}\cz(0,0;\,n,y;\beta)\,\cz(n,y;\,n+m,*;\beta)\,\ud y.
  \end{equation}
  Set
  \[
    I_n\coloneqq \int_{\R^d}\cz(0,0;\,n,y;\beta)\,\cz_\infty^{(n)}(y)\,\ud y.
  \]
  To pass to the limit as $m\to\infty$ in~\eqref{E:Z_nm}, note that
  \[
    \E\left[\left|\cz_{n+m}-I_n\right|\right]
    \le \int_{\R^d}\E\left[\cz(0,0;\,n,y;\beta)\left|\cz(n,y;\,n+m,*;\beta)-\cz_\infty^{(n)}(y)\right|\right]\ud y.
  \]
  Since $\cz(0,0;\,n,y;\beta)$ depends only on the noise on $(0,n]$ while
  $\cz(n,y;\,n+m,*;\beta)$ depends only on the noise on $(n,n+m]$, the two are
  independent. Using also the centrality property
  $\E[\cz(0,0;\,n,y;\beta)]=p_n(y)$, spatial stationarity in $y$ (owing to
  Theorem~\ref{T:transition-properties}-\ref{T:transition-properties:stationarity} again), and
  $\int_{\R^d}p_n(y)\ud y = 1$, we obtain
  \[
    \E\left[\left|\cz_{n+m}-I_n\right|\right]
    \le \E\left[\left|\cz(n,0;\,n+m,*;\beta)-\cz_\infty^{(n)}(0)\right|\right].
  \]
  The right-hand side tends to $0$ since $\cz(n,0;\,n+m,*;\beta) \to
  \cz_\infty^{(n)}(0)$ in $L^2(\Omega)$, hence in $L^1(\Omega)$. Therefore,
  $\cz_{n+m} \to I_n$ in $L^1(\Omega)$. Since $\cz_{n+m} \to \cz_\infty$ in
  $L^2(\Omega)$ (hence in $L^1(\Omega)$), we deduce that $\cz_\infty = I_n$
  almost surely. Hence, since $\cz(0,0;\,n,y;\beta) > 0$ for all $y$ almost
  surely (Theorem~\ref{T:transition-properties}-\ref{T:transition-properties:positive}), we see that
  \begin{align*}
    \left\{\cz_\infty=0\right\}
    = & \left\{\cz_\infty^{(n)}(y)=0 \text{ for Lebesgue-a.e.\ $y\in\R^d$}\right\}               \\
    = & \bigcap_{k\ge 1}\left\{\int_{B(0,k)}\big(\cz_\infty^{(n)}(y)\wedge 1\big)\ud y=0\right\}
    \in  \: \sigma\!\left(W((n,\infty)\times A):A\in\mathcal{B}(\R^d)\right).
  \end{align*}
  Since $n\ge 1$ is arbitrary, we conclude that $\left\{\cz_\infty=0\right\}
  \in \mathcal{T}$, which proves~\eqref{E:tail-sigma-field}.

  With~\eqref{E:tail-sigma-field}, one can conclude this step easily. That is,
  $\PP(\cz_\infty=0) \in \{0,1\}$ by the Kolmogorov $0$--$1$ law. Since
  $\E[\cz_\infty]=1$, we cannot have $\PP(\cz_\infty=0) = 1$. This
  shows~\eqref{E:Z_infty} and concludes the proof of
  item~\emph{\ref{I:diffusive:i}}.

  \smallskip

  \noindent\textit{Step~2.} In this step, we show that
  \begin{align}
    \lim_{T\to\infty}\E\left[\mathcal{M}_T^2\right]=0,\label{E:M_T-to-0}
    \quad \text{with} \quad
    \mathcal{M}_T \coloneqq \bE_T^{\beta,W}\left[g\left(\frac{X_T}{\sqrt{T}}\right)-\nu(g)\right]\cz_T.
  \end{align}
  To prove~\eqref{E:M_T-to-0}, we first apply Lemma~\ref{L:F-K_Mom_Polymer} to
  get
  \begin{align}\label{E:FK-moment-polymer-special}
    \E\left[\mathcal{M}_T^2\right]=\bE_0\left[e^{\beta^2L_T}\left(g\left(\frac{X_T}{\sqrt{T}}\right)-\nu(g)\right)\left(g\left(\frac{\widetilde{X}_T}{\sqrt{T}}\right)-\nu(g)\right)\right],
  \end{align}
  where $L_T$ is as in~\eqref{E:def_Lt_L8}. Next we apply
  Lemma~\ref{L:triple-ind} to the right-hand-side
  of~\eqref{E:FK-moment-polymer-special} via a truncation argument. Set
  $\bar{g}\coloneqq g-\nu(g)$, so that $\bar{g} \in C_b(\R^d)$ and
  $\bE[\bar{g}(Z)] = 0$ for a standard Gaussian $Z$. For $n \ge 1$, define
  $\phi_n(x) \coloneqq \exp\big(\beta^2(x\wedge n)\big)$. By
  Lemma~\ref{L:triple-ind},
  \[
    \lim_{T\to\infty}\bE_0\left[\phi_n(L_T) \, \bar{g}\left(\frac{X_T}{\sqrt{T}}\right)\bar{g}\left(\frac{\widetilde{X}_T}{\sqrt{T}}\right)\right]
    = \bE_0\left[\phi_n(L_\infty)\right]\bE\left[\bar{g}(Z)\right]\bE\left[\bar{g}(\widetilde{Z})\right]
    = 0,
  \]
  where $L_\infty$ is the monotone limit in~\eqref{E:def_Lt_L8}. Moreover, since
  $L_T\le L_\infty$ $\bP_0$-a.s., we have
  \[
    0\le e^{\beta^2 L_T}-\phi_n(L_T)\le e^{\beta^2 L_\infty}\mathbf{1}_{\{L_\infty>n\}},
  \]
  and therefore, using $\Norm{\bar{g}}_\infty \le 2 \Norm{g}_\infty$,
  \[
    \sup_{T\ge 0}\left|\bE_0\left[\left(e^{\beta^2 L_T}-\phi_n(L_T)\right)\bar{g}\left(\frac{X_T}{\sqrt{T}}\right)\bar{g}\left(\frac{\widetilde{X}_T}{\sqrt{T}}\right)\right]\right|
    \le 4\Norm{g}_\infty^2\,\bE_0\left[e^{\beta^2 L_\infty}\mathbf{1}_{\{L_\infty>n\}}\right]
    \xrightarrow[n\to\infty]{} 0,
  \]
  where the last limit follows from~\eqref{E:Finite_Mom_all_t} and the dominated
  convergence theorem. Combining the last two displays yields
  $\lim_{T\to\infty}\E\left[\mathcal{M}_T^2\right] = 0$, that is, we have
  established~\eqref{E:M_T-to-0}. \medskip

  \noindent\textit{Step~3: Proof of~\ref{I:diffusive:ii}.} Let
  $\varepsilon,\delta>0$. Then
  \[
    \PP\left(\left|\bE_T^{\beta,W}\left[g\left(\frac{X_T}{\sqrt{T}}\right)\right]-\nu(g)\right|>\varepsilon\right)
    =\PP\left(|\mathcal{M}_T|>\varepsilon \cz_T\right)
    \le \PP\left(|\mathcal{M}_T|>\varepsilon\delta\right) + \PP\left(\cz_T<\delta\right).
  \]
  For fixed $\varepsilon,\delta>0$, Chebyshev's inequality
  and~\eqref{E:M_T-to-0} give
  \[
    \PP\left(|\mathcal{M}_T|>\varepsilon\delta\right)
    \le \frac{\E\left[\mathcal{M}_T^2\right]}{\varepsilon^2\delta^2}
    \xrightarrow[T\to\infty]{}0.
  \]
  For the second term, we have seen in
  item~\emph{\ref{I:diffusive:i}} that $\cz_T\to\cz_\infty$ in probability.
  Therefore, $\limsup_{T\to\infty}\PP(\cz_T<\delta) \le
  \PP(\cz_\infty\le\delta)$. The latter tends to $0$ as $\delta\downarrow 0$
  since $\cz_\infty>0$ $\PP$-a.s. This completes the proof of
  Theorem~\ref{T:diffusive-behavior}.
\end{proof}

\begin{remark}
  For comparison, a discrete analogue of the argument in Step~1 of the proof of
  Theorem~\ref{T:diffusive-behavior} can be found
  in~\cite{comets:17:directed}*{Theorem~3.1}.
\end{remark}

\appendix
\section{Some auxiliary integral inequalities}\label{S:Gronwall}

The aim of this appendix is to prove a space-time version of Gronwall's
inequality, as stated in Theorem~\ref{T:IntIneq} below. It is a generalization
of~\cite{chen.huang:19:comparison}*{Lemma~2.2} from the case $\alpha = 0$ to
$\alpha \in (0, \eta)$ by employing the stronger assumption---the strengthened
Dalang's condition~\eqref{E:stren}---instead of the general Dalang's
condition~\eqref{E:Dalang}.

\begin{theorem}\label{T:IntIneq}
  Suppose that $\mu$ is a signed measure that satisfies
  \begin{align}\label{E:J0finite}
    \int_{\R^d} e^{-b |x|^2} |\mu|(\ud x) <+\infty\;, \quad \text{for all $b>0$}\;,
  \end{align}
  and let $J_0(t,x)$ be the solution to the homogeneous equation as
  in~\eqref{E:def-J0}, i.e., $J_0(t,x) \coloneqq (p_t*\mu)(x)$. Suppose that for
  some $\eta \in (0,1)$, the strengthened Dalang's condition~\eqref{E:stren}
  holds. We also assume that with an $\alpha \in (0, \eta)$, a nonnegative
  (measurable) function $g:\R_+\times\R^{d}\mapsto\R_+$ satisfies that for all
  $t>0$ and $x\in\R^d$,
  \begin{align*}
    \int_0^t\ud s\iint_{\R^{2d}} & p_{t-s}(x-y_1)p_{t-s}(x-y_2)
    f(y_1-y_2) g(s,y_1)g(s,y_2)\ud y_1\ud y_2<+\infty ,
  \end{align*}
  as well as the Gr\"{o}nwall-type condition
  \begin{align}\label{E:IntIneq1}
    \begin{aligned}
      g(t,x)^2\le t^{-\alpha} J_0^2(t,x)+ \lambda^2 \int_0^t\ud s \iint_{\R^{2d}} \:
       & f(y_1-y_2)p_{t-s}(x-y_1) p_{t-s}(x-y_2)  \\
       & \times  g(s,y_1)g(s,y_2)\ud y_1\ud y_2 .
    \end{aligned}
  \end{align}
  Then, for all $T > 0$, there is some constant $C_{T,\eta,\alpha, \lambda} > 0$
  such that
  \begin{align}\label{E:IntIneq2}
    \begin{aligned}
      g(t,x) \le C_{T,\eta,\alpha, \lambda}\, t^{-\alpha/2} \left(|\mu| * p_t \right)(x), \quad \text{for all $t \in (0, T)$ and $x \in \R^d$}.
    \end{aligned}
  \end{align}
\end{theorem}
We now start a series of intermediate results towards the proof of
Theorem~\ref{T:IntIneq}. The first one is a recursion for functions defined by
weighted integrals, which is an elaboration of Lemma~\ref{L:hn-upper-bound}.

\begin{proposition}\label{P:hn-ea}
  Let $\alpha, \eta \in (0, 1)$, set $h_0^{(\eta, \alpha)}(t) \equiv 1$, and for
  $n \ge 1$ define
  \[
    h_n^{(\eta, \alpha)} (t) \coloneqq \int_0^t h_{n-1}^{(\eta, \alpha)} (s)\,s^{-\alpha}\,(t-s)^{\eta-1}\,\ud s,\qquad t>0.
  \]
  Then for $n \ge 1$ and all $t > 0$, provided the integral is finite, we have
  \begin{equation}\label{E:hn-ea}
    h_n^{(\eta, \alpha)} (t) = C_n\,t^{\,n(\eta-\alpha)} \quad \text{and} \quad 
    C_n                      = \big[\Gamma(\eta)\big]^n\prod_{j=0}^{n-1} \frac{\Gamma\!\big(1-\alpha+j(\eta-\alpha)\big)}{\Gamma\!\big(1-\alpha+\eta+j(\eta-\alpha)\big)}.
  \end{equation}
  Moreover, the following are equivalent:
  \begin{enumerate}[label=\textbf{(\roman*)}]

    \item For every $t>0$, the following series converges absolutely for all
      $\gamma > 0$:
      \begin{equation}\label{E:H-ea-def}
        H^{(\eta, \alpha)} (t;\gamma) \coloneqq \sum_{n=0}^\infty \gamma^n h_n^{(\eta, \alpha)} (t) .
      \end{equation}

    \item $\eta > \alpha$.

  \end{enumerate}
  When $\eta = \alpha$, one has
  $h_n^{(\eta, \eta)} (t) \equiv \left(\frac{\pi}{\sin(\pi\eta)}\right)^n$, and
  $\sum_{n = 0}^\infty \gamma^n h_n^{(\eta, \eta)} (t)$ converges iff
  $|\gamma| < \pi^{-1}\sin(\pi\eta)$. If $\eta<\alpha$, the Beta factor
  in~\eqref{E:hn-ea} is finite only while
  $1 - \alpha + (n-1)(\eta-\alpha) > 0$, so only finitely many $h_n$ are finite.
  When $\eta > \alpha$, for some constant $C_\eta$ only depending on $\eta$, it
  holds that
  \begin{align}\label{E:H-ea}
    H^{(\eta, \alpha)} (t;\gamma) \le
    C_\eta\, E_\eta\left(\frac{\gamma\,\Gamma(\eta)\,t^\delta}{\delta^\eta}\right)
    \quad \text{for all $t, \gamma > 0$, with $\delta \coloneqq \eta - \alpha > 0$},
  \end{align}
  where we recall that $E_{\eta} = E_{\eta, 1}$, and $E_{\alpha, \beta}$ is the
  Mittag--Leffler function introduced in~\eqref{E:mittag-leffler}.
\end{proposition}

\begin{remark}
  When $\alpha = 0$, then $C_n$ defined in~\eqref{E:hn-ea} reduces to $C_n =
  \Gamma(\eta)^n/\Gamma(n\eta+1)$. In this case, this proposition is consistent
  with Lemma~\ref{L:hn-upper-bound}.
\end{remark}

\begin{proof}[Proof of Proposition~\ref{P:hn-ea}]
  Throughout the proof, we write $\mathfrak{h}_n(t) \coloneqq h_n^{(\eta,
  \alpha)}(t)$ for brevity, and note this should not be confused with $h_n$
  defined in~\eqref{E:hn-recursion} in the rest of this paper. We first
  prove~\eqref{E:hn-ea} by induction. When $n=1$, we have
  \[
    \mathfrak{h}_1(t) = \int_0^t s^{-\alpha}\,(t-s)^{\eta-1}\,ds
    = t^{\eta-\alpha} B\!\big(1-\alpha,\,\eta\big)
    = C_1\,t^{\eta-\alpha},
  \]
  where $B(a,b)$ is the Beta function. Assume $\mathfrak{h}_{n-1}(t) =
  C_{n-1}t^{(n-1)(\eta-\alpha)}$. Then
  \[
    \begin{aligned}
      \mathfrak{h}_n(t)
       & =C_{n-1}\int_0^t s^{(n-1)(\eta-\alpha)-\alpha}\,(t-s)^{\eta-1}\,ds                                                              \\
       & = C_{n-1} t^{n(\eta-\alpha)} B\!\big(1-\alpha+(n-1)(\eta-\alpha),\,\eta\big)                                                    \\
       & = C_{n-1} t^{n(\eta-\alpha)} \frac{\Gamma(1-\alpha+(n-1)(\eta-\alpha))\,\Gamma(\eta)}{\Gamma(1-\alpha+\eta+(n-1)(\eta-\alpha))}
      = C_n\,t^{n(\eta-\alpha)},
    \end{aligned}
  \]
  which yields~\eqref{E:hn-ea}. \medskip

  \noindent We now turn to the proof of ``(i) $\implies$ (ii)'' in the
  proposition. If $\eta = \alpha$, the expression for $\mathfrak{h}_n(t)$
  in~\eqref{E:hn-ea} reduces to
  \[
    \mathfrak{h}_n(t) \equiv \left[ \Gamma(\eta)\Gamma(1-\eta) \right]^n = \left(\frac{\pi}{\sin(\pi\eta)}\right)^n,
  \]
  where we have used the reflection formula for the Gamma function (see, e.g.,
  Formula~5.5.3 in~\cite{olver.lozier.ea:10:nist}). So $\sum_n \gamma^n
  \mathfrak{h}_n(t)$ is convergent only if
  $|\gamma|<(\Gamma(\eta)\Gamma(1-\eta))^{-1}$; hence (i) fails for large values
  of $|\gamma|$. Additionally, if $\eta<\alpha$, then
  $1-\alpha+(n-1)(\eta-\alpha)\downarrow -\infty$, and the Beta integral
  diverges for large $n$ so that $1-\alpha+(n-1)(\eta-\alpha) \leq -1$. This
  contradicts (i), thereby completing the proof of ``(i) $\implies$ (ii)''.
  \medskip

  \noindent To justify ``(ii) $\implies$ (i)'', denote $\delta \coloneqq \eta -
    \alpha > 0$. Recall the so-called \textit{Gautschi inequality} (see, e.g.,
  Formula~5.6.4 in~\cite{olver.lozier.ea:10:nist})
  \begin{align*}
    x^{1-s} < \frac{\Gamma(x+1)}{\Gamma(x+s)}
    \iff \frac{\Gamma (x + s)}{\Gamma (x + 1)} < x^{s - 1}, \quad \text{for all $x > 0$ and $0 < s < 1$} .
  \end{align*}
  Applying this inequality with $x = j \delta + \eta - \alpha$, $s = 1 - \eta$, we obtain
  \[
    \frac{\Gamma(1-\alpha+j\delta)}{\Gamma(1-\alpha+\eta+j\delta)} \le \left(j\delta + \eta - \alpha \right)^{-\eta} = ((j+1)\delta)^{-\eta},
    \quad \text{for all } j \ge 0.
  \]
  Recall $C_n$ defined in~\eqref{E:hn-ea}. The above inequality suggests that
  \[
    C_n \le \left[\Gamma(\eta)\right]^n \times
    \left(\prod_{j=0}^{n-1} (\delta^{-\eta} (j+1)^{-\eta})\right)
    = \left(\frac{\Gamma(\eta)}{\delta^\eta}\right)^{\!n}\,\frac{1}{(n!)^{\eta}}.
  \]
  Therefore, for any $t>0$,
  \[
    |\gamma|^n\,\mathfrak{h}_n(t)
    = |\gamma|^n\, C_n\, t^{n\delta}
    \le \left(\frac{|\gamma|\,\Gamma(\eta)\,t^\delta}{\delta^\eta}\right)^{\!n}\,\frac{1}{(n!)^{\eta}}
    \le \left(\frac{|\gamma|\,\Gamma(\eta)\,t^\delta}{\delta^\eta}\right)^{\!n}\,\frac{C_{\eta}}{\Gamma(n\eta+1)},
  \]
  where the last inequality is due to the Stirling formula. Summing over $n$
  proves~\eqref{E:H-ea} and hence, (i). This completes the proof of
  Proposition~\ref{P:hn-ea}.
\end{proof}

We will need the following lemma borrowed
from~\cite{chen.huang:19:comparison}*{Lemma~B.1}.

\begin{lemma}\label{L:shatf}
  If $g(t)$ is a monotone function over $[0,T]$, then for all $\beta>0$ and
  $t\in (0,T]$,
  \begin{align*}
    \int_0^t g(t-s)
    \exp & \left(
    -\frac{2 \beta s(t-s)}{t}\right)\ud s
    =\int_0^t g(s)
    \exp\left(-\frac{2\beta s(t-s)}{t}\right)\ud s \\[0.3em]
         & \le
    \begin{dcases}
      2 \int_0^t g(s)e^{-\beta (t-s)} \ud s, & \text{if $g$ is nondecreasing,} \\[0.8em]
      2 \int_0^t g(s)e^{-\beta s}     \ud s, & \text{if $g$ is nonincreasing.}
    \end{dcases}
  \end{align*}
\end{lemma}

Now we are ready to prove Theorem~\ref{T:IntIneq}.

\begin{proof}[Proof of Theorem~\ref{T:IntIneq}]
  Fix an arbitrary $T > 0$. Recall $k(t)$ is defined in~\eqref{E:k_t}. Because
  the strengthened Dalang's condition~\eqref{E:stren} holds, we can apply
  Lemma~\ref{L:k-upper-bound} to bound the kernel $k(t)$ as in~\eqref{E:kt-bdd}. In
  particular, for $t \in (0,T)$,
  \[
    k(t) \le C_\eta^* \Upsilon_\eta (t/T)^{\eta-1} = (C_\eta^* T^{1-\eta})\,\Upsilon_\eta\, t^{\eta-1}.
  \]
  In the sequel we will also set $C_{T,\eta} \coloneqq C_\eta^* T^{1-\eta}$ and
  $\gamma \coloneqq 2 \lambda^2 C_{T,\eta} \Upsilon_\eta$, where $C_\eta^*$
  comes from Lemma~\ref{L:k-upper-bound} and $\Upsilon_\eta$ is defined in~\eqref{E:stren}.

  We will prove Theorem~\ref{T:IntIneq} using Picard's iteration. This iteration
  is defined as follows: let $g_0(t,x) = t^{-\alpha/2} \left(|\mu| *
  p_t\right)(x)$, and for $n\ge 1$, set
  \begin{multline}\label{E:gn}
    g_{n}^2(t,x) \coloneqq t^{-\alpha} J_0^2(t,x)
    +  \lambda^2 \int_0^t\ud s\iint_{\R^{2d}} p_{t-s}(x-y_1) p_{t-s}(x-y_2) \\
    \times g_{n-1}(s,y_1)g_{n-1}(s,y_2)f(y_1-y_2)\ud y_1 \ud y_2.
  \end{multline}
  We shall prove by induction that
  \begin{align}\label{E:Indt-Lp}
    g_n(t,x) \le g_0(t,x) \left(\sum_{\ell=0}^n \gamma^\ell h_{\ell}^{(\eta,\alpha)} (t)\right)^{1/2},\quad\text{for all $n\ge 0$,}
  \end{align}
  where $h_k^{(\eta,\alpha)}(t)$ are defined in Proposition~\ref{P:hn-ea}. It is
  clear that~\eqref{E:Indt-Lp} holds for $n = 0$. Now suppose
  that~\eqref{E:Indt-Lp} is true for $n \ge 0$ and propagate the induction. To
  this aim, note that
  \begin{align}\label{E:IntIneq-c}
    g_{n+1}^2(t,x)
    =       & \: t^{-\alpha} J_0^2(t,x)
    +\lambda^2 \int_0^t \iint_{\R^{2d}}
    p_{t-s}(x-y_1)
    p_{t-s}(x-y_2) f(y_1-y_2)   \nonumber                                                \\
              & \hspace{10em}\times g_n(s,y_1) g_n(s,y_2) \ud s\ud y_1\ud y_2  \nonumber \\
    \eqqcolon & \: t^{-\alpha} J_0^2(t,x) + \lambda^2\:  I(t,x).
  \end{align}
  By the induction assumption~\eqref{E:Indt-Lp}, we get
  \begin{align*}
    I(t,x)\le &
    \int_0^t \ud s \: s^{-\alpha} \iint_{\R^{2d}} \ud y_1 \ud y_2\:f(y_1-y_2)
    p_{t-s}(x-y_1)
    p_{t-s}(x-y_2)     \\
              & \times
    |J_0(s,y_1)| \:
    |J_0(s,y_2)|
    \left(\sum_{\ell=0}^n \gamma^{\ell} h^{(\eta, \alpha)}_{\ell} (s)\right).
  \end{align*}
  Hence, bounding the terms $J_0$ thanks to the definition~\eqref{E:def-J0}, we
  discover that
  \begin{align*}
    I(t,x)\le &
    \int_0^t \ud s \: s^{-\alpha} \iint_{\R^{2d}}\ud y_1\ud y_2\:f(y_1-y_2)
    p_{t-s}(x-y_1)
    p_{t-s}(x-y_2)                                                                                                                                                     \\
              & \times \iint_{\R^{2d}}|\mu|(\ud z_1)|\mu|(\ud z_2) p_{s}(y_1-z_1) p_{s}(y_2-z_2) \left(\sum_{\ell=0}^n \gamma^\ell h^{(\eta, \alpha)}_\ell (s)\right).
  \end{align*}
  We now apply~\eqref{E:heat-kernel-factorization} to the products $p_{t - s} (x - y_1) p_s (y_1 -
  z_1)$ and $p_{t - s} (x - y_1) p_{s} (y_2 - z_2)$, and end up with
  \begin{align}\label{E:IntIneq-a}
    I(t,x)
    \le & \int_0^t \ud s \: s^{-\alpha} \left(\sum_{\ell=0}^n \gamma^\ell h^{(\eta, \alpha)}_\ell (s)\right) \iint_{\R^{2d}}\ud y_1\ud y_2\:f(y_1-y_2) \nonumber \\
        & \times p_{\frac{s(t-s)}{t}}\left(y_1-z_1-\frac{s}{t}(x-z_1)\right)
    p_{\frac{s(t-s)}{t}}\left(y_2-z_2-\frac{s}{t}(x-z_2)\right)       \nonumber                                                                                  \\
        & \times \iint_{\R^{2d}}|\mu|(\ud z_1)|\mu|(\ud z_2) \: p_{t}(x-z_1) p_{t}(x-z_2).
  \end{align}
  On the right-hand side above, we isolate the double integral with respect to
  the $\ud y_1 \ud y_2$ variables. We first apply the change of variables $y_1 -
  (s/t) x \mapsto y_1 $ and $y_2 - (s/t) x \mapsto y_2$, then use Parseval's
  identity. This enables to recast this double integral as
  \begin{align*}
    (2\pi)^{-d} \int_{\R^d} \widehat{f}(\ud \xi)
    \exp\left(i \frac{t-s}{t}(z_1-z_2) \cdot \xi - \frac{s(t-s)}{t}|\xi|^2\right).
  \end{align*}
  Since $\widehat{f}$ is nonnegative, this integral is bounded by
  \begin{align*}
    (2\pi)^{-d}\int_{\R^d} \widehat{f}(\ud \xi) \exp\left(- \frac{s(t-s)}{t}|\xi|^2\right).
  \end{align*}
  Plugging this inequality back in~\eqref{E:IntIneq-a} we obtain
  \begin{equation}\label{E:IntIneq-b}
    I(t,x) \le g_0^2(t,x) \int_0^t\ud s\: s^{-\alpha} \left(\sum_{\ell=0}^n \gamma^\ell h^{(\eta, \alpha)}_\ell (s)\right) (2\pi)^{-d}\int_{\R^d} \widehat{f}(\ud \xi) \exp\left(- \frac{s(t-s)}{t}|\xi|^2\right).
  \end{equation}
  With~\eqref{E:IntIneq-b} in hand, we recall that the family $\{h^{(\eta,
  \alpha)}_\ell ; 0 \leq \ell \leq n\}$ satisfies~\eqref{E:hn-ea}. In
  particular, since $\alpha < \eta$, we have that $t \rightarrow h^{(\eta,
  \alpha)}_\ell (t)$ is nondecreasing. Hence, applying Lemma~\ref{L:shatf} with
  $\beta = |\xi|^2/2$ and recalling the definition~\eqref{E:k_t} for $k$, we see
  that
  \begin{align*}
    I(t,x)
    \le & 2\ g_0^2(t,x) \int_{0}^t\ud s\: s^{-\alpha} \left(\sum_{\ell=0}^n \gamma^\ell h^{(\eta, \alpha)}_\ell (s)\right) (2\pi)^{-d}\int_{\R^d} \widehat{f}(\ud \xi) \exp\left(- \frac{t-s}{2}|\xi|^2\right) \\
    \le & 2\ g_0^2(t,x) \int_{0}^t\ud s\: s^{-\alpha} \left(\sum_{\ell=0}^n \gamma^\ell h^{(\eta, \alpha)}_\ell(s)\right) k(t-s).
  \end{align*}
  Then by~\eqref{E:kt-bdd} and~\eqref{E:hn-ea}, we see that for all $t \in (0,
  T)$,
  \begin{align*}
    I(t,x)
     & \le 2C_{T, \eta} \Upsilon_\eta \ g_0^2(t,x) \int_0^t \ud s\: s^{-\alpha} \left(\sum_{\ell=0}^n \gamma^\ell h^{(\eta, \alpha)}_\ell(s)\right) (t-s)^{\eta-1} \\
     & = 2C_{T, \eta} \Upsilon_\eta\ g_0^2(t,x) \sum_{\ell=0}^n \gamma^\ell h^{(\eta, \alpha)}_{\ell+1}(t).
  \end{align*}
  Plugging this inequality back in~\eqref{E:IntIneq-c} and setting $\gamma =
  2C_{T, \eta} \Upsilon_\eta\lambda^2$, we obtain
  \begin{align*}
    g_{n+1}^2(t,x)
    \le g_0^2(t,x) + \underbrace{2C_{T, \eta} \Upsilon_\eta\lambda^2}_{=\gamma}\, g_0^2(t,x) \sum_{\ell=0}^n \gamma^\ell h^{(\eta, \alpha)}_{\ell+1}(t)
    \le g_0^2(t,x) \sum_{\ell=0}^{n+1} \gamma^\ell h^{(\eta, \alpha)}_\ell (t).
  \end{align*}
  This proves the claim in~\eqref{E:Indt-Lp} by induction. Finally, owing to
  our assumption~\eqref{E:IntIneq1}, for all $t \in (0, T)$ and $x \in \R^d$, it
  holds that
  \begin{align*}
    g(t,x)
    \le \lim_{n\rightarrow\infty} g_n(t,x)
    \le g_0(t,x) \left(\sum_{\ell=0}^{\infty} \gamma^\ell h^{(\eta, \alpha)}_\ell (t)\right)^{1/2},
  \end{align*}
  where the convergence of the series under the square root follows from
  Proposition~\ref{P:hn-ea}. This completes the proof of
  Theorem~\ref{T:IntIneq}.
\end{proof}

The last lemma of this section is an estimate which is used in heat kernel
evaluations in Lemma~\ref{L:holder}.
\begin{lemma}\label{L:Bridge-t}
  For all $t > s' > s > 0$ and $x \in \mathbb{R}$ ,
  it holds that
  \begin{equation}\label{E:Bridge-t}
    \int_s^{s'} \exp\left(-\frac{(t-r)(r-s)}{t-s}\,x^2\right)\ud r
    \le
    \begin{dcases}
      2 \dfrac{1-e^{-\frac{x^2}{2}(s'-s)}}{x^2},                              & s' \in \big[s, (t + s)/2\big], \\
      2 \dfrac{1 - 2e^{-\frac{x^2}{4}(t-s)} + e^{-\frac{x^2}{2}(t-s')}}{x^2}, & s' \in \big((t + s)/2, t\big].
    \end{dcases}
  \end{equation}
  Moreover, for all $\eta \in (0,1]$, there exists $C_\eta > 0$ such that for
  all $t \ge s' > s \ge 0$ and $x \in \mathbb{R}$, it holds that
  \begin{align}\label{E:Bridge-tx}
    \int_s^{s'} \exp\left(-\frac{(t-r)(r-s)}{t-s}\,x^2\right)\ud r
    \le C_\eta\, e^{\frac{t-s}{4}} \dfrac{(s'-s)^\eta}{(1+x^2)^{1-\eta}}.
  \end{align}
\end{lemma}
\begin{proof}
  Set $m \coloneqq (t+s)/2$. Note that $(t-r)(r-s)$ is symmetric about $m$, and
  on $[s, m]$, $t-r \ge \tfrac{t-s}{2}$ while on $[m, t]$, $r-s \ge
  \tfrac{t-s}{2}$, or equivalently,
  \begin{align}\label{E:LinearQuad}
    \frac{(t-r)(r-s)}{t-s} \ge \frac{1}{2}\min\{t-r, r-s\} =
    \begin{dcases}
      \dfrac{1}{2}(r-s), & r \in \big[s, (t + s)/2\big], \\
      \dfrac{1}{2}(t-r), & r \in \big[(t + s)/2, t\big].
    \end{dcases}
  \end{align}
  Then~\eqref{E:Bridge-t} is obtained by replacing the quadratic term in the
  exponent by the linear term in~\eqref{E:LinearQuad}. As
  for~\eqref{E:Bridge-tx}, we will need the elementary inequalities:
  \begin{equation}\label{E:exp_ineq}
    1 - e^{-z} \le z^\eta \quad \text{and} \quad
    |e^{-z} - e^{-w}| \le |z - w|^\eta, \quad \text{for all $\eta \in (0, 1)$ and $z, w \ge 0$.}
  \end{equation}
  Notice that
  \begin{align*}
    I
    \equiv & \int_s^{s'} \exp\left(-\frac{(t-r)(r-s)}{t-s}\left(1+x^2\right) + \frac{(t-r)(r-s)}{t-s}\right)\ud r \\
    \le    & e^{\frac{t-s}{4}} \int_s^{s'} \exp\left(-\frac{(t-r)(r-s)}{t-s}\left(1+x^2\right)\right)\ud r.
  \end{align*}
  Then, one applies~\eqref{E:Bridge-t} with $x^2$ replaced by $1+x^2$ to see
  that
  \begin{align*}
    I \le
    2 e^{\frac{t-s}{4}}
    \begin{dcases}
      \dfrac{1-e^{-\frac{1+x^2}{2}(s'-s)}}{1+x^2},                                & s' \in \big[s, (t + s)/2\big], \\
      \dfrac{1 - 2e^{-\frac{1+x^2}{4}(t-s)} + e^{-\frac{1+x^2}{2}(t-s')}}{1+x^2}, & s' \in \big((t + s)/2, t\big].
    \end{dcases}
  \end{align*}
  When $s' \in [s, (t + s)/2]$, we can apply the first inequality
  in~\eqref{E:exp_ineq} to obtain~\eqref{E:Bridge-tx} with $C_\eta =
  2^{1-\eta}$. It remains to prove~\eqref{E:Bridge-tx} for the case $s' \in ( (t
  + s)/2, t]$. Let $a \coloneqq \frac{1+x^2}{4}$ and fix an arbitrary $\eta \in
  (0, 1)$.   Decompose the expression in three parts:
  \[
    1 - 2e^{-a(t-s)} + e^{-2a(t-s')}
    = \underbrace{(1 - e^{-a(t-s)})}_{T_1}
    + \underbrace{(e^{-a(t-s')} - e^{-a(t-s)})}_{T_2}
    + \underbrace{(e^{-2a(t-s')} - e^{-a(t-s')})}_{T_3}.
  \]
  By~\eqref{E:exp_ineq}, we have
  \[
    T_1   \le a^\eta    |t-s|^\eta,  \quad
    |T_2| \le a^\eta    |s'-s|^\eta, \quad
    |T_3| = e^{-a(t-s')} \left|1-e^{-a(t-s')}\right| \le  a^\eta |t-s'|^\eta.
  \]
  Now, since $s' > (t+s)/2$, it follows that $t - s' < s' - s$ and $t - s \le
  2(s' - s)$. Substituting these into the above bounds yields
  \[
    \begin{aligned}
      1 - 2e^{-a(t-s)} + e^{-2a(t-s')}
       & \le a^\eta \left( (2|s'-s|)^\eta + |s'-s|^\eta + |t-s'|^\eta \right) \\
       & \le a^\eta(2^\eta + 2)\,|s'-s|^\eta                                  \\
       & = 4^{-\eta}(2^\eta + 2) \left(1+x^2\right)^\eta\,|s'-s|^\eta.
    \end{aligned}
  \]
  Hence,~\eqref{E:Bridge-tx} holds with $C_\eta = 4^{-\eta}(2^\eta + 2)$. This
  completes the proof of Lemma~\ref{L:Bridge-t}.
\end{proof}

\section{An almost sure compactness lemma}\label{S:as-compact}

In the main text, we will need to upgrade pointwise convergence of random fields
to uniform convergence on compact sets. The next lemma records a standard
compactness argument on $C(B(0,R))$, combining a uniform (in $n$) moment bound
on spatial increments (via Kolmogorov--Chentsov) with a finite-net approximation
and a Borel--Cantelli argument. Throughout this section, we write $B(0,R)
\coloneqq \{x \in \R^d : |x| \le R\}$ for the closed ball of radius $R$ centered
at the origin.

\begin{lemma}\label{L:as-compact}
  Let $X, X_1, X_2, \dots$ be random fields on $\R^d$. Assume:
  \begin{enumerate}[label=\textbf{(\roman*)}]
    \item There exist $p > 1$ and $\alpha \in (0, 1)$ with $p\alpha > d$ such
      that for any $R > 0$,
      \[
        \sup_{n \ge 1} \E\big[|X_n(x) - X_n(y)|^p\big]
        \le C_R |x - y|^{\alpha p},
        \quad \text{for all $x, y \in B(0, R)$,}
      \]
      where $C_R > 0$ depends only on $R$.

    \item For all $x \in \R^d$, $X_n(x) \to X(x)$ in probability.
  \end{enumerate}
  Then, for every $R > 0$, there exist a subsequence $\{X_{n_k}\}_{k \ge 1}$ and
  a modification $\widetilde X$ of $X$ such that
  \begin{equation}\label{E:as-compact}
    \sup_{x \in B(0,R)} |X_{n_k}(x) - \widetilde X(x)| \to 0,
    \qquad \text{almost surely.}
  \end{equation}
\end{lemma}
\begin{proof}
  Fix $R > 0$. By assumption (i) and the Kolmogorov--Chentsov continuity theorem
  (see, e.g., \cite{klenke:14:probability}*{Theorem 21.6}), for any $\gamma \in
  \big(0, \alpha - d/p\big)$, each $X_n$ admits a modification (still denoted by
  $X_n$) whose restriction to $B(0,R)$ is $\gamma$-H\"older continuous a.s.
  Moreover, a standard application of Garsia-Rademich-Rumsey's lemma shows that
  for every $\delta \in (0, 1)$ there exists a finite constant $K = K(\gamma,
  \delta, R)$ such that
  \[
    \inf_{n \ge 1} \PP\left(
    \sup_{x \ne y \in B(0,R)}
    \frac{|X_n(x) - X_n(y)|}{|x - y|^\gamma} \le K
    \right) \ge 1 - \delta.
  \]
  Define the modulus of continuity
  \[
    w_R(X_n, \rho)
    \coloneqq
    \sup_{\substack{x \ne y \in B(0,R) \\ |x - y| \le \rho}} |X_n(x) - X_n(y)|.
  \]
  Choose $\rho \coloneqq (\delta/K)^{1/\gamma}$. Then, on the event
  $\big\{\sup_{x \ne y \in B(0,R)} |X_n(x) - X_n(y)|/|x-y|^\gamma \le K\big\}$
  we have $w_R(X_n,\rho) \le K\rho^\gamma = \delta$, and hence
  \begin{equation}\label{E:modulus-cont}
    \sup_{n \ge 1} \PP\big(w_R(X_n,\rho) > \delta\big) \le \delta.
  \end{equation}

  Next, let $\mathcal{N}_\rho$ be a truncated $(\rho/\sqrt{d})$--lattice:
  \[
    \mathcal{N}_\rho \coloneqq B(0, R + \rho) \cap \frac{\rho}{\sqrt{d}}\ZZ^d.
  \]
  Then $\mathcal{N}_\rho$ is a finite $\rho$-net of $B(0,R)$. By assumption
  (ii), for each $z \in \mathcal{N}_\rho$ we have $X_n(z) \to X(z)$ in
  probability. Since $\mathcal{N}_\rho$ is finite, there exists an index
  $N_\delta$ such that for all $n, n' \ge N_\delta$,
  \begin{equation}\label{E:net-cauchy}
    \PP\left(
    \max_{z \in \mathcal{N}_\rho}
    |X_n(z) - X_{n'}(z)| > \delta
    \right) < \delta.
  \end{equation}
  For any $n, n' \ge N_\delta$, define the event
  \[
    E_{n,n'}
    \coloneqq
    \{w_R(X_n,\rho) \le \delta\}
    \cap \{w_R(X_{n'},\rho) \le \delta\}
    \cap \left\{\max_{z \in \mathcal{N}_\rho} |X_n(z) - X_{n'}(z)| \le \delta\right\}.
  \]
  By~\eqref{E:modulus-cont} and~\eqref{E:net-cauchy}, we have
  \begin{align*}
    \PP(E_{n,n'})
    \ge 1
     & - \PP\big(w_R(X_n,\rho) > \delta\big)
    - \PP\big(w_R(X_{n'},\rho) > \delta\big)                                         \\
     & - \PP\left(\max_{z \in \mathcal{N}_\rho} |X_n(z) - X_{n'}(z)| > \delta\right)
    \ge 1 - 3\delta.
  \end{align*}
  On $E_{n,n'}$, for any $x \in B(0,R)$ choose $z_x \in \mathcal{N}_\rho$ such
  that $|x - z_x| \le \rho$. Then
  \[
    |X_n(x) - X_{n'}(x)|
    \le |X_n(x) - X_n(z_x)|
    + |X_n(z_x) - X_{n'}(z_x)|
    + |X_{n'}(z_x) - X_{n'}(x)|
    \le 3\delta,
  \]
  and consequently, for all $n, n' \ge N_\delta$,
  \[
    \PP\left(
    \sup_{x \in B(0,R)} |X_n(x) - X_{n'}(x)| > 3\delta
    \right) \le 3\delta.
  \]

  For each $k \ge 1$, let $\delta_k \coloneqq 2^{-k}$, and let $N_k$ be a choice
  of index $N_{\delta_k}$ from~\eqref{E:net-cauchy}. Define $n_1 \coloneqq N_1$
  and $n_{k+1} \coloneqq \max\{n_k + 1, N_{k+1}\}$. Then, for every $k \ge 1$,
  we have $n_k, n_{k+1} \ge N_k$ and therefore
  \[
    \PP\left(
    \sup_{x \in B(0,R)} |X_{n_{k+1}}(x) - X_{n_k}(x)| > 3\delta_k
    \right) \le 3\delta_k.
  \]
  Since $\sum_{k \ge 1} 3\delta_k < \infty$, the Borel--Cantelli lemma implies
  that, for all large $k$, almost surely,
  \[
    \sup_{x \in B(0,R)} |X_{n_{k+1}}(x) - X_{n_k}(x)| \le 3\delta_k.
  \]
  In particular, $\{X_{n_k}\}_{k \ge 1}$ forms an
  a.s. Cauchy sequence in $C(B(0,R))$ equipped with the supremum norm, and hence
  converges uniformly on $B(0,R)$ to a random field $\widetilde X$ with
  continuous sample paths on $B(0,R)$.

  Finally, fix $x \in \R^d$. By assumption (ii), $X_{n_k}(x) \to X(x)$ in
  probability, while by the uniform convergence we have $X_{n_k}(x) \to
  \widetilde X(x)$ almost surely. Therefore, $\widetilde X(x) = X(x)$ almost
  surely, and $\widetilde X$ is a modification of $X$. This completes the proof
  of Lemma~\ref{L:as-compact}.
\end{proof}

\section*{Acknowledgments}
L.~C. was partially supported by NSF grants DMS-2246850/2443823 and a
Collaboration Grant for Mathematicians (\#959981) from the Simons Foundation.
C.O. was partially supported by a Collaboration Grant for Mathematicians
(\#851792) from the Simons Foundation. S.~T. was partially supported by the NSF
grant DMS-2450734.

\medskip

\end{document}